\documentclass[11pt]{amsart}%
\usepackage{amsmath}%
\usepackage{amssymb}%
\usepackage{amscd}%
\usepackage[british,french]{babel}
\usepackage{array}
\usepackage{color}%
\usepackage{mathrsfs}%
\usepackage[all]{xy}%
\usepackage{stmaryrd}
\usepackage{hyperref}
\usepackage[T1]{fontenc}

\def\ol#1{\overline{#1}}
\def\wh#1{\widehat{#1}}
\def\wt#1{\widetilde{#1}}

\theoremstyle{plain}
    \newtheorem{theorem}{Theorem}[section]
    
    \newtheorem{proposition}[theorem]{Proposition}
    \newtheorem{lemma}[theorem]{Lemma}
    \newtheorem{corollary}[theorem]{Corollary}
      
      \newtheorem{proposition-definition}[theorem]{Proposition-Definition}
      
\theoremstyle{definition}
    \newtheorem{definition}[theorem]{Definition}

    \newtheorem{remark}[theorem]{Remark}
    \newtheorem*{acknowledgments}{Acknowledgements}
    
\renewcommand\labelenumi{(\roman{enumi})}
\renewcommand\theenumi\labelenumi

\def\Alphabet{A,B,C,D,E,F,G,H,I,J,K,L,M,N,O,P,Q,R,S,T,U,V,W,X,Y,Z}
\def\alphabet{a,b,c,d,e,f,g,h,i,j,k,l,m,n,o,p,q,r,s,t,u,v,w,x,y,z}
\def\endpiece{xxx}
\def\makeAlphabet[#1]{\expandafter\makeA#1,xxx,}
\def\makealphabet[#1]{\expandafter\makea#1,xxx,}
\def\makeA#1,{\def\temp{#1}\ifx\temp\endpiece\else%
\mkbb{#1}\mkfrak{#1}\mkbf{#1}\mkcal{#1}\mkscr{#1}\mkbs{#1}\expandafter\makeA\fi}%
\def\makea#1,{\def\temp{#1}\ifx\temp\endpiece\else\mkfrak{#1}\mkbf{#1}\mkbs{#1}\expandafter\makea\fi}%
\def\mkbb#1{\expandafter\def\csname bb#1\endcsname{\mathbb{#1}}}
\def\mkfrak#1{\expandafter\def\csname fr#1\endcsname{\mathfrak{#1}}}
\def\mkbf#1{\expandafter\def\csname b#1\endcsname{\mathbf{#1}}}
\def\mkcal#1{\expandafter\def\csname c#1\endcsname{\mathcal{#1}}}
\def\mkscr#1{\expandafter\def\csname s#1\endcsname{\mathscr{#1}}}
\def\mkbs#1{\expandafter\def\csname bs#1\endcsname{{\boldsymbol{#1}}}}
\def\makeop[#1]{\xmakeop#1,xxx,}
\def\mkop#1{\expandafter\def\csname #1\endcsname{{\mathrm{#1}}}} %
\def\xmakeop#1,{\def\temp{#1}\ifx\temp\endpiece\else\mkop{#1}\expandafter\xmakeop\fi}%
\def\makeup[#1]{\xmakeup#1,xxx,}
\def\mkup#1{\expandafter\def\csname #1\endcsname{{\mathrm{#1}\,}}} %
\def\xmakeup#1,{\def\temp{#1}\ifx\temp\endpiece\else\mkup{#1}\expandafter\xmakeup\fi}%
\makeAlphabet[\Alphabet]
\makealphabet[\alphabet]
\makeop[id,Alt,End,Ext,Hom,Mod,Sym,Tor,dR,rig,cris,conv,HK,an,unip,Li,Cone,Tot,res,Res,pol,Gr,Ker,Coker,Re,Vec,Rep,GL,abs,holim,hocolim,dim,Mod,PD,conv,fin,zar,Fil,isoc,Tr,Ber,fp,tor,nil,ord]
\makeup[Spec,Proj,Spwf,Sp,Im,Isoc,MIC,OC,Str,Conv,Spf]

  \makeatletter
    
    \@addtoreset{equation}{section}
  \makeatother

\newcommand{\et}{\mathrm{\acute{e}t}}
\newcommand*{\sheafhom}{\cH\kern -.5pt om}

\begin{document}
\selectlanguage{british}

\title{Poincar\'e duality for rigid analytic Hyodo--Kato theory}
\author{Veronika Ertl and Kazuki Yamada}

\date{\today}

\begin{abstract}
	The purpose of this paper is to establish Hyodo--Kato theory with compact support for semistable schemes through rigid analytic methods.
	To this end we introduce several types of log rigid cohomology with compact support.
	Moreover we show that the additional structures on the (rigid) Hyodo--Kato cohomology and the Hyodo--Kato map introduced in our previous paper are compatible with Poincar\'e duality.
	Compared to the crystalline approach, the constructions are explicit yet versatile, and hence suitable for computations.
\end{abstract}

\thanks{The first named author's research was supported in part by the EPSRC grant  EP/R014604/1 and by the DFG grant SFB 1085.
The second named author was supported by KAKENHI Grant Numbers 18H05233, 22K13899, and 23KJ0332.}
\maketitle

\tableofcontents


%
\section*{Introduction}\label{sec: introduction}
%

Let $p$ be a prime number, $V$ a complete discrete valuation ring of mixed characteristic $(0,p)$ with fraction field $K$ and perfect residue field $k$, $W:=W(k)$ the ring of Witt vectors of $k$, and $F$ the fraction field of $W$.
Let $V^\sharp$ be the formal scheme $\Spf V$ equipped with the canonical log structure  $\cN_{V^\sharp}$. 
Let $W^0$ be the formal scheme $\Spf W$ equipped with the log structure associated to the monoid homomorphism $(\bbN\rightarrow W;\ 1\mapsto 0)$, and let $k^0=(k^0,\cN_{k^0}):=W^0\otimes_Wk$.
Furthermore, for a choice of uniformiser $\pi$ of $V$, let $i_\pi: k^0\hookrightarrow V^\sharp$ be the morphism sending $\pi\in V\setminus\{0\}=\Gamma(V^\sharp,\cN_{V^\sharp})$ to $(1,1)\in\bbN\oplus k^\times=\Gamma(k^0,\cN_{k^0})$.

Consider a fine log scheme $X$ which is proper log smooth over $V^\sharp$, and assume that $Y_\pi:=X\otimes_{V^\sharp}k^0$ is of Cartier type over $k^0$.
In their celebrated paper, Hyodo and Kato \cite{HyodoKato1994} constructed a morphism called the (crystalline) Hyodo--Kato map
	\[\Psi_\pi^\cris\colon R\Gamma_\cris(Y_\pi/W^0)\rightarrow R\Gamma_\dR(X_K/K).\]
It depends on the choice of a uniformiser $\pi\in V$ and induces an isomorphism after tensoring with $K$.

The domain $R\Gamma_\cris(Y_\pi/W^0)$ is often called the (crystalline) Hyodo--Kato cohomology.
It is endowed with two endomorphisms called Frobenius operator and monodromy operator.
The codomain $R\Gamma_\dR(X_K/K)$ is the log de~Rham cohomology and comes with a descending filtration called Hodge filtration.
Under the identification via $\Psi_\pi^\cris$, the associated cohomology groups give rise to a filtered $(\varphi,N)$-module.
When $X$ is semistable, this object contains information as rich as the $p$-adic \'{e}tale cohomology of $X_{\overline{K}}$ (where $\overline{K}$ is the algebraic closure of $K$) as a Galois representation, by the semistable conjecture which is now a theorem (\cite{Tsuji1999-a}, \cite{Faltings2002}, \cite{Niziol2008}, \cite{Beilinson2013}).

Among the generalisations of the semistable conjecture one should mention Yamashita--Yasuda's work  (based on an unpublished paper by Yamashita) on the potentially semistable conjecture for non-proper varieties  and its compatibility with the product structures \cite{YamashitaYasuda2014}.
They crucially make use of log crystalline cohomology with (partial) compact support.
In the proof, they also show the compatibility of the Hyodo--Kato map with the product structures (and hence with Poincar\'{e} duality).

Poincar\'{e} duality for crystalline Hyodo--Kato cohomology was proved by Tsuji \cite{Tsuji1999}.
Note that he more generally proved Poincar\'{e} duality of log crystalline cohomology over $W^0$ and $W^\varnothing$ for a fine saturated log scheme which is log smooth and universally saturated over $k^0$ and $k^\varnothing$, respectively, where $W^\varnothing$ and $k^\varnothing$ are affine (formal) schemes equipped with the trivial log structure.
(However the compatibility with the additional structures and the Hyodo--Kato map was not addressed. )

The role which Hyodo--Kato theory plays in the (potentially) semistable conjecture is an indicator of its importance in arithmetic geometry, including the research on special values of $p$-adic $L$-functions.
To advance prominent conjectures in this field, such as the Bloch--Kato conjecture, it is of advantage to be able to describe explicitly certain cohomology classes, both in the usual Hyodo--Kato theory and the compactly supported Hyodo--Kato theory, which are related to Euler systems.
For instance, this was carried out for $4$-dimensional spin Galois representations
arising from Siegel modular forms for the group $\Sp_4(\bbZ)$ by Loeffler and Zerbes in \cite{LoefflerZerbes}.

Such explicit computations are possible through an approach to Hyodo--Kato theory 
 for strictly semistable log schemes based on log rigid cohomology instead of log crystalline cohomology provided by the authors in \cite{ErtlYamada}.
It allows for a definition of the Hyodo--Kato map
	\[\Psi_{\pi,\log}\colon R\Gamma_\rig(Y_\pi/W^0)\rightarrow R\Gamma_\rig(Y_\pi/V^\sharp)\cong R\Gamma_\dR(X_K/K)\]
in a straight forward and intuitive way which is independent of the choice of a uniformiser $\pi$ in the sense that for two uniformiser $\pi,\pi'$ there is a canonical quasi-isomorphism 
$R\Gamma_\rig(Y_\pi/W^0)\xrightarrow{\sim}R\Gamma_\rig(Y_{\pi'}/W^0)$ 
and we have $\Psi_{\pi,\log}=\Psi_{\pi',\log}$. (It rather depends on the choice of a branch $\log\colon K^\times\rightarrow K$ of the $p$-adic logarithm.)
Therefore, it is suitable not only for explicit computation but also for generalisations.
For example, the second author studied coefficients for rigid analytic Hyodo--Kato theory in \cite{Yamada2020}, and generalised the additional structures on the rigid Hyodo--Kato cohomology and the Hyodo--Kato map to cohomology with coefficients.

The construction of rigid Hyodo--Kato theory with compact support in the present article can be seen as an application of \cite{Yamada2020}.
More precisely, we introduce several types of log rigid cohomology with compact support as log rigid cohomology with coefficients in certain types of log overconvergent isocrystals. 
We compare them with Yamashita--Yasuda's crystalline constructions in \cite{YamashitaYasuda2014} and prove Poincar\'{e} duality.
Because of the simplicity of our construction, the duality is clearly compatible with the additional structures and the Hyodo--Kato map.

This simplicity together with the rigid analytic nature of our construction makes it a useful candidate to advance some of the most difficult conjectures in arithmetic geometry as demonstrated in \cite{LoefflerZerbes}.

\subsubsection*{Outline of this paper}

In \S \ref{sec: isoc}, we recall notions concerning log overconvergent ($F$-)isocrystals which were introduced in \cite{Yamada2020}.
In particular, we recall definitions of the log rigid cohomology and the rigid Hyodo--Kato cohomology with coefficients in log overconvergent ($F$-)isocrystals.

In \S \ref{section : log substructures}, we introduce a general construction of log overconvergent isocrystals associated to monogenic log substructures.
Let $Y=(Y,\cN_Y)$ be a fine log scheme over a base fine weak formal log scheme $\cT$.
A log substructure $\cM\subset\cN_Y$ is called \textit{monogenic} if  locally it has a chart of the form $\bbN_Y\rightarrow\cM$.
Then we may define a log overconvergent isocrystal $\sO_{Y/\cT}(\cM)$ on $Y$ over $\cT$ as a locally free sheaf of rank $1$ generated by local lifts of generators of $\cM$.

As a special case, in \S \ref{sec: c rig} we consider a proper strictly semistable log scheme $Y$ over $k^0$ with horizontal divisor $D$, and several weak formal log schemes as base: $W^0$, $W^\varnothing$, $V^\sharp$, and $\cS$, where $\cS$ is $\Spwf W\llbracket s\rrbracket$ equipped with the log structure associated to $(\bbN\rightarrow W\llbracket s\rrbracket;\ 1\mapsto s)$.
To the horizontal divisor $D$ we associate on $Y$ a certain monogenic log substructure $\cM_D\subset\cN_Y$.
We define the log rigid cohomology (resp.\,rigid Hyodo--Kato cohomology) with compact support as the log rigid cohomology (resp.\,rigid Hyodo--Kato cohomology) with coefficients in $\sO_{Y/\cT}(\cM_D)$.
Since the cohomology with compact support is computed by usual cohomology via a spectral sequence, one can generalise several results to cohomology with compact support as follows.

\begin{theorem}[Corollary \ref{cor: compare rigid}]
	Let $Y$ be a proper strictly semistable log scheme over $k^0$ and $\sE$ a unipotent log overconvergent isocrystal on $Y$ over $W^\varnothing$.
	 There exists a diagram of natural morphisms
		\begin{equation}\label{eq: diag intro}
	\xymatrix{
	\ar@{}[rd]|>>>>\circlearrowright& R\Gamma_{\rig,c}(Y/W^\varnothing,\sE)\ar[ld]\ar[d]\ar[rd] &\ar@{}[ld]|>>>>\circlearrowright \\
	R\Gamma_{\rig,c}(Y/W^0,\sE) & R\Gamma^\rig_{\HK,c}(Y,\sE)\ar[l]\ar[r]^-{\Psi_{\pi,\log}} \ar[d]& R\Gamma_{\rig,c}(Y/V^\sharp,\sE)_\pi,\\
	\ar@{}[ru]|>>>>\circlearrowright& R\Gamma_{\rig,c}(Y/\cS,\sE)\ar[lu]\ar[ru] &\ar@{}[lu]|>>>>{(*)}
	}\end{equation}
	where all triangles except for $(*)$ commute.
	The triangle $(*)$ commutes if $\log$ is taken to be the branch of the $p$-adic logarithm with $\log(\pi)= 0$.
	These morphisms induce isomorphisms
	\begin{align*}
	&R\Gamma_{\rig,c}(Y/W^\varnothing,\sE)\xrightarrow{\cong}\Cone(R\Gamma^\rig_{\HK,c}(Y,\sE)\xrightarrow{N}R\Gamma^\rig_{\HK,c}(Y,\sE)),\\
	&R\Gamma^\rig_{\HK,c}(Y,\sE)\xrightarrow{\cong} R\Gamma_{\rig,c}(Y/W^0,\sE),\\
	&R\Gamma^\rig_{\HK,c}(Y,\sE)\otimes_FF\{s\}\xrightarrow{\cong}R\Gamma_{\rig,c}(Y/\cS,\sE),\\
	&R\Gamma^\rig_{\HK,c}(Y,\sE)\otimes_FK\xrightarrow[\Psi_{\pi,\log,K}]{\cong}R\Gamma_{\rig,c}(Y/V^\sharp,\sE)_\pi.
	\end{align*}
	Here $F\{s\}$ denotes the global section of the generic fiber of $\cS$, that is the ring of functions which converge on the open unit disk, and $\Psi_{\pi,\log,K}$ is induced from $\Psi_{\pi,\log}$ by tensoring with $K$.
	The associated cohomology groups are finite over appropriate base rings.
	When we consider a Frobenius structure $\Phi$ on $\sE$, then the induced Frobenius operator on $H^{\rig,i}_{\HK,c}(Y,(\sE,\Phi))$ is bijective.
\end{theorem}

In \S \ref{sec: crys}, we recall and reformulate the definitions of log crystalline cohomology with compact support and the crystalline Hyodo--Kato map of \cite{YamashitaYasuda2014}.
As in \cite{ErtlYamada}, we introduce two affine formal schemes $\cS'_{\PD}$ and $\cS_{\PD}$ as bases of log crystalline cohomology, and compare them.
The log crystalline cohomology over $\cS'_{\PD}$ is used for the crystalline Hyodo--Kato map, and that over $\cS_{\PD}$ is related with the log rigid cohomology over $\cS$.

In \S \ref{sec: comparison}, we compare the log rigid and log crystalline cohomology with compact support with each other, and prove that the crystalline Hyodo--Kato map $\Psi_\pi^\cris$ and the rigid Hyodo--Kato map $\Psi_{\pi,\log}^\rig$ are compatible when $\log$ is taken to be the branch of $p$-adic logarithm with $\log(\pi)=0$.

In \S \ref{sec: dual}, we deduce Poincar\'{e} duality for log rigid cohomology from that for log crystalline cohomology given by Tsuji \cite{Tsuji1999}.
As an advantage of our construction, we clearly have the compatibility of  Poincar\'{e} duality with all additional structures and comparison maps.
Our main result is as follows.

\begin{theorem}[Theorem \ref{thm: dual}]
	Let $Y$ be a proper strictly semistable log scheme over $k^0$ of dimension $d$.
	Let $(\sE,\Phi)$ be a unipotent log overconvergent $F$-isocrystal, and denote the dual $F$-isocrystal by $(\sE^\vee,\Phi^\vee):=\sheafhom((\sE,\Phi),\sO_{Y/W^\varnothing})$.
	Then there exist canonical isomorphisms
	\begin{align}
	\nonumber &H^i_\rig(Y/W^\varnothing,(\sE,\Phi)) \xrightarrow{\cong}H^{2d-i+1}_{\rig,c}(Y/W^\varnothing,(\sE^\vee,\Phi^\vee))^*(-d-1)&&\text{in }\Mod_F(\varphi),\\
	\nonumber&H^{\rig,i}_\HK(Y,(\sE,\Phi)) \xrightarrow{\cong}H^{\rig,2d-i}_{\HK,c}(Y,(\sE^\vee,\Phi^\vee))^*(-d)&&\text{in }\Mod_F(\varphi,N),\\
	\nonumber & H^i_\rig(Y/W^0,(\sE,\Phi)) \xrightarrow{\cong}H^{2d-i}_{\rig,c}(Y/W^0,(\sE^\vee,\Phi^\vee))^*(-d)&&\text{in }\Mod_F(\varphi),\\
	\nonumber & H^i_\rig(Y/V^\sharp,\sE)_\pi \xrightarrow{\cong}H^{2d-i}_{\rig,c}(Y/V^\sharp,\sE^\vee)_\pi^*&&\text{in }\Mod_K,
	\end{align}
	where for an integer $r\in\bbZ$ we denote by $(r)$ the Tate twist defined by multiplying $\varphi$ by $p^{-r}$.
	The Hyodo--Kato maps are compatible with Poincar\'{e} duality in the sense that after tensoring with $K$ we have $\Psi_{\pi,\log,K}^*=\Psi_{\pi,\log,K}^{-1}$.
\end{theorem}

\subsubsection*{Notation}
In this paper we use the convention that $\bbN$ contains $0$.

Let $V$ be a complete discrete valuation ring of mixed characteristic $(0,p)$ with 
residue field $k$ and fraction field $K$.
Assume that $k$ is algebraic over $\bbF_p$.
Denote by $W:= W(k)$ the ring of Witt vectors of $k$ and by $F$ the fraction field of $W$.

In accordance with \cite{ErtlYamada}, for any (fine) log scheme $Y$ over $k$, we always assume that its underlying scheme is separated, locally of finite type over $k$, and admits an affine covering indexed by a countable set.

Weak formal schemes as defined in \cite[Def.\,1.9]{ErtlYamada} are omnipresent in the constructions of this paper.
Locally, they are weak formal spectra of pseudo-weakly complete finitely generated (pseudo-wcfg) algebras over a base $(R,I)$ where $R$ is a noetherian ring and $I\subset R$ and ideal.  
The weak formal schemes and pseudo-wcfg algebras in this paper are considered over $(R,I)=(W,pW)$.
We freely use results from \cite[\S\,1]{ErtlYamada} where they were studied in detail and generality.

For a weak formal scheme $\cX$, 
we denote by $\cX_\bbQ$ the generic fibre of $\cX$, which is a dagger space over $F$.
For a morphism $f\colon\cX'\rightarrow\cX$, we denote again by $f\colon\cX'_\bbQ\rightarrow\cX_\bbQ$ the morphism on dagger spaces induced by $f$.

For a pseudo-wcfg algebra \cite[Def.\, 1.3]{ErtlYamada} $A$, we denote by $A^\varnothing$ the weak formal scheme $\Spwf A$ endowed with the trivial log structure.

For any finitely generated monoid $M$, we denote $W[M]^\dagger$ the weak completion of the monoid ring $W[M]$ with respect to $(W,pW)$.
In particular we have $W[\bbN^n]^\dagger\cong W[x_1,\ldots,x_n]^\dagger$.

Consider the weak formal log scheme $\cS= \Spwf W\llbracket s\rrbracket$ endowed with the log structure associated to the map $(\bbN  \rightarrow  W\llbracket s\rrbracket;\ 1\mapsto s)$.
Let $W^0$ and $k^0$ be the exact closed weak formal log scheme of $\cS$ defined by the ideal $(s)$ and $(s,p)$, respectively.
Furthermore, let $V^\sharp$ be the weak formal  log scheme with underlying scheme  $\Spwf V$ endowed with the canonical log structure, whose global section is $V\setminus\{0\}$.
Let $i_0\colon k^0 \hookrightarrow W^0$, $j_0\colon W^0 \hookrightarrow\cS$ be the canonical embeddings, and set $\tau:= j_0\circ i_0$.
For a uniformiser $\pi\in V$, let $j_\pi\colon V^\sharp \hookrightarrow\cS$ be the exact closed immersion defined by $s\mapsto\pi$, and $i_\pi\colon k^0 \hookrightarrow V^\sharp$ the unique morphism such that $\tau= j_\pi\circ i_\pi$.

We denote by $\sigma$ the absolute Frobenius on $k^0$, $k^\varnothing$, and by abuse of notation the morphism on $W^0$ and $W^\varnothing$ induced by the Witt vector Frobenius, and likewise its extension to $\cS$ defined by $s\mapsto s^p$.

For a (weak formal) log scheme $\cZ$, we denote its log structure by $\cN_\cZ$.
We often denote the underlying (weak formal) scheme by the same symbol $\cZ$.

For an abelian category $\cM$, we denote by $D^b(\cM)$ (resp.\,$D^+(\cM)$, $D^-(\cM)$)  the bounded (resp.\,bounded below, bounded above) classical derived category of $\cM$.

\begin{acknowledgments}
We are particularly grateful towards Seidai Yasuda for giving us access to the newest version of his article with Go Yamashita.
We would like to thank Wies\l{}awa Nizio\l{} for helpful discussions related to the subject of this paper.
We thank David Loeffler and Sarah Zerbes for their ongoing interest in this work, inspiring discussions and helpful explanations of their work.
We would like to thank Kei Hagihara and Shuji Yamamoto for helpful discussions concerning extension groups of $(\varphi,N)$-modules. 
We thank Kenichi Bannai and Yoshinosuke Hirakawa for their helpful comments to the draft of this paper. 
We warmly thank the referees for their careful reading and thoughtful suggestions that helped to improve the manuscript substantially.

The first named author would like to thank the Isaac Newton Institute for Mathematical Sciences for support and hospitality during the programme ``$K$-theory, algebraic cycles and motivic homotopy theory'' when work on this paper was undertaken.
\end{acknowledgments}

%
\section{Log overconvergent isocrystals}\label{sec: isoc}
%

In this section we recall the definition of the log overconvergent site and log overconvergent isocrystals from \cite{Yamada2020}.

Until the end of Section \ref{sec: c rig}, a log structure on a (weak formal) scheme is defined as a sheaf of monoids with respect to the Zariski topology.

\begin{definition}\label{def: widening}
	\begin{enumerate}
	\item A {\it widening} is a triple $(Z,\cZ,i)$ where $Z$ is a fine log scheme over $k$, $\cZ$ is a weak formal log scheme, and $i\colon Z \hookrightarrow\cZ$ is a homeomorphic exact closed immersion over $W$.
		A morphism of widenings $f\colon(Z',\cZ',i') \rightarrow(Z,\cZ,i)$ is a pair $f= (f_k,f_W)$ of morphisms $f_k\colon Z' \rightarrow Z$ over $k^\varnothing$ and $f_W\colon\cZ' \rightarrow\cZ$ over $W^\varnothing$, such that $f_W\circ i'= i\circ f_k$.
	\item An {\it $F$-widening} is a quadruple $(Z,\cZ,i,\phi)$ where $(Z,\cZ,i)$ is a widening and $\phi\colon\cZ \rightarrow\cZ$ is a lift of the absolute Frobenius on $\cZ\times_{W^\varnothing}k^\varnothing$.
		For an $F$-widening $(Z,\cZ,i,\phi)$, the absolute Frobenius $F_Z$ on $Z$ and $\phi$ on $\cZ$ define a morphism $(F_Y,\phi)\colon(Z,\cZ,i) \rightarrow(Z,\cZ,i)$, which we denote again by $\phi$.
		A morphism of $F$-widenings $f\colon(Z',\cZ',i',\phi') \rightarrow(Z,\cZ,i,\phi)$ is a morphism $f\colon(Z',\cZ',i') \rightarrow(Z,\cZ,i)$ of widenings such that $f\circ\phi'= \phi\circ f$.	
	\item A morphism $f$ of widenings (resp.\, $F$-widenings) is said to be {\it log smooth} if $f_W\colon\cZ' \rightarrow\cZ$ is log smooth \cite[Def.\,1.52]{ErtlYamada}.
	\end{enumerate}
\end{definition}

In the following we consider a widening $(T,\cT,\iota)$ or an $F$-widening $(T,\cT,\iota,\sigma)$ as a base.
For a morphism $f\colon(Z',\cZ',i') \rightarrow(Z,\cZ,i)$ of widenings, or $f\colon(Z',\cZ',i',\phi') \rightarrow(Z,\cZ,i,\phi)$ of $F$-widenings, we often denote the morphisms $f_W\colon\cZ' \rightarrow\cZ$ again by $f$, if no confusion can arise.

\begin{definition}\label{def: log overconvergent site}
	Let $Y$ be a fine log scheme over $T$.
	We define the {\it log overconvergent site} $\OC(Y/\cT)= \OC(Y/(T,\cT,\iota))$ of $Y$ over $(T,\cT,\iota)$ as follows:
	\begin{enumerate}
	\item An object of $\OC(Y/\cT)$ is a quintuple $(Z,\cZ,i,h,\theta)$ where $h\colon(Z,\cZ,i) \rightarrow(T,\cT,\iota)$ is a morphism of widenings, and $\theta\colon Z \rightarrow Y$ is a morphism over $T$.
	\item A morphism $f\colon(Z',\cZ',i',h',\theta') \rightarrow(Z,\cZ,i,h,\theta)$ in $\OC(Y/\cT)$ is a morphism $f\colon(Z',\cZ',i') \rightarrow(Z,\cZ,i)$ of widenings over $(T,\cT,\iota)$ such that $\theta\circ f_k= \theta'$.
	\item A covering family in $\OC(Y/\cT)$ is a family of morphisms $\{f_\lambda\colon (Z_\lambda,\cZ_\lambda,i_\lambda,h_\lambda,\theta_\lambda) \rightarrow(Z,\cZ,i,h,\theta)\}_\lambda$ such that
		\begin{itemize}
		\item $f_{\lambda,W}^*\cN_\cZ= \cN_{\cZ_\lambda}$ for any $\lambda$, 
		\item the family $\{f_{\lambda,W}\colon\cZ_{\lambda,\bbQ} \rightarrow\cZ_\bbQ\}_\lambda$ is an admissible covering, 
		\item the morphism $(f_{\lambda,k},i_\lambda)\colon Z_\lambda \rightarrow Z\times_\cZ\cZ_\lambda$ is an isomorphism for any $\lambda$.
		\end{itemize}
	\end{enumerate}
\end{definition}

\begin{definition}\label{def: univ enl}
	Let $(Z,\cZ,i)$ be a widening.
	Let $\cJ$ be the ideal of $Z$ in $\cZ$.
	Let $\cU=\Spwf A\subset\cZ$ be an affine open subset. 
	For $n\geq 0$, take generators $f_1,\ldots,f_m\in A$ of the ideal $\Gamma(\cU,\cJ^n)\subset A$ and define
	\[A[n]:=\left(A[x_1,\ldots,x_m]^\dagger/(px_1-f_1,\ldots,px_m-f_m)\right)/(\text{$p$-torsion}).\]
	Then $\cU[n]:=\Spwf A[n]$ for all affine open subsets $\cU$ glue to each other by \cite[Lem.\,1.28]{ErtlYamada}, and define a $p$-adic weak formal scheme $\cZ[n]$.
	We endow $\cZ[n]$ with the pull-back log structure of $\cZ$, and let $i[n]\colon Z[n]:=\cZ[n]\times_\cZ Z\hookrightarrow\cZ[n]$ be the canonical exact closed immersion. 
	Then the family $\{(Z[n],\cZ[n],i[n])\}_n$ forms a direct system of widenings over $(Z,\cZ,i)$, which we call the \textit{universal enlargements} of $(Z,\cZ,i)$.
	
	If $(Z,\cZ,i,h,\theta)$ is an object of $\OC(Y/\cT)$, then universal enlargements define a direct system $\{(Z[n],\cZ[n],i[n],h[n],\theta[n])\}_n$ in $\OC(Y/\cT)$.
\end{definition}

For a weak formal scheme $\cZ$ and an $\cO_{\cZ}$-module $\cF$, we let $\cF_\bbQ$ be $\mathrm{sp}^*\cF=\mathrm{sp}^{-1}\cF\otimes_{\mathrm{sp}^{-1}\cO_{\cZ}}\cO_{\cZ_\bbQ}$.
If $f\colon\cZ'\rightarrow\cZ$ is a morphism of fine weak formal log schemes, then we have $\omega^m_{\cZ'/\cZ,\bbQ}|_{\cZ'[n]}\cong \omega^m_{\cZ'[n]/\cZ}\otimes_{\bbZ_p}\bbQ_p$, where $\cZ'[n]$ is endowed with the pull-back log structure from $\cZ'$.
Thus the differentials $\omega^m_{\cZ'[n]/\cZ}\rightarrow\omega^{m+1}_{\cZ'[n]/\cZ}$ induce the differentials $\omega^m_{\cZ'[n]/\cZ,\bbQ}\rightarrow\omega^{m+1}_{\cZ'[n]/\cZ,\bbQ}$. They for all $n$ glue each other and we obtain a complex $\omega^\bullet_{\cZ'/\cZ,\bbQ}$ of sheaves on $\cZ'_\bbQ$.

To a sheaf $\sE$  on $\OC(Y/\cT)$ and an object $(Z,\cZ,i,h,\theta)\in\OC(Y/\cT)$, one can associate a sheaf $\sE_\cZ= \sE_{(Z,\cZ,i,h,\theta)}$ on $\cZ_\bbQ$ as follows:
Since $\cZ_\bbQ=\bigcup_{n}\cZ[n]_\bbQ$, it suffices to give a compatible family of sheaves $\sE_{\cZ[n]}$ on $\cZ[n]_\bbQ$.
Any admissible open subset $\frU\subset\cZ[n]_\bbQ$ can be written as $\frU=\cU_\bbQ$ for some admissible blow-up $\cU\rightarrow\cZ[n]$.
We endow $\cU$ with the pull-back log structure from $\cZ[n]$ and let $U:=\cU\times_{\cZ[n]}Z[n]$.
Let $i_\cU\colon U\hookrightarrow\cU$, $h_\cU\colon(U,\cU,i_\cU)\rightarrow(T,\cT,\iota)$, and $\theta_\cU\colon U\rightarrow Y$ be the morphisms obviously induced.
We define a sheaf $\sE_{\cZ[n]}$ by
	\[\Gamma(\frU,\sE_{\cZ[n]}):=\Gamma((U,\cU,i_\cU,h_\cU,\theta_\cU),\sE).\]
	
The structure sheaf $\sO_{Y/\cT}$ on $\OC(Y/\cT)$ is defined by $\Gamma((Z,\cZ,i,h,\theta),\sO_{Y/\cT}):= \Gamma(\cZ_\bbQ,\cO_{\cZ_\bbQ})$.

\begin{definition}\label{def: isocrystal}
	A {\it log overconvergent isocrystal} on $Y$ over $(T,\cT,\iota)$ is an $\sO_{Y/\cT}$-module $\sE$ on $\OC(Y/\cT)$, such that
	\begin{itemize}
	\item For any $(Z,\cZ,i,h,\theta)\in\OC(Y/\cT)$, the sheaf $\sE_\cZ$ on $\cZ_\bbQ$ is a coherent locally free $\cO_{\cZ_\bbQ}$-module,
	\item For any morphism $f\colon(Z',\cZ',i',h',\theta') \rightarrow(Z,\cZ,i,h,\theta)$, the natural morphism $f^*\sE_{\cZ} \rightarrow\sE_{\cZ'}$ is an isomorphism.
	\end{itemize}
	A morphism $f\colon\sE' \rightarrow\sE$ of log overconvergent isocrystals is an $\sO_{Y/\cT}$-linear homomorphism.
	We denote by
		\[\Isoc^\dagger(Y/\cT)= \Isoc^\dagger(Y/(T,\cT,\iota))\]
	the category of log overconvergent isocrystals on $Y$ over $(T,\cT,\iota)$.
\end{definition}

Let $\varrho\colon (T',\cT',\iota') \rightarrow(T,\cT,\iota)$ be a morphism of widenings and consider a commutative diagram
	\[\xymatrix{
	Y'\ar[r] ^\rho \ar[d] & Y \ar[d]\\
	T'\ar[r]^{\varrho_k} &T.
	}\]
Then we may regard an object of $\OC(Y'/\cT')$ as an object of $\OC(Y/\cT)$ via $\rho$ and $\varrho$.
This defines a co-continuous functor $\OC(Y'/\cT') \rightarrow\OC(Y/\cT)$, which induces a functor
	\begin{equation}\label{eq: pull back isoc}
	(\rho,\varrho)^*\colon\Isoc^\dagger(Y/\cT) \rightarrow\Isoc^\dagger(Y'/\cT').
	\end{equation}
	
In particular, a morphism $f\colon Y' \rightarrow Y$ over $T$ induces a functor
	\begin{equation}
	f^*:= (f,\id_{(T,\cT,\iota)})^*\colon\Isoc^\dagger(Y/\cT) \rightarrow\Isoc^\dagger(Y'/\cT).
	\end{equation}

The endomorphism $\sigma$ on $(T,\cT,\iota)$ defined by  the absolute Frobenius $F_T$ on $T$ and a Frobenius lift $\sigma$ on $\cT$ together with the absolute Frobenius $F_Y$ on $Y$   induces a functor
	\begin{equation}\label{eq: Frob of isoc}
	\sigma^\ast_\isoc := (F_Y,\sigma)^*\colon\Isoc^\dagger(Y/\cT) \rightarrow\Isoc^\dagger(Y/\cT).
	\end{equation}

\begin{definition}
	A {\it log overconvergent $F$-isocrystal} on $Y$ over $(T,\cT,\iota,\sigma)$ is a pair $(\sE,\Phi)$ of a log overconvergent isocrystal $\sE\in\Isoc^\dagger(Y/\cT)$ and an isomorphism $\sigma^\ast_\isoc \sE \xrightarrow{\cong}\sE$, which we call a {\it Frobenius structure} on $\sE$.
	A morphism $f\colon(\sE',\Phi') \rightarrow(\sE,\Phi)$ of log overconvergent $F$-isocrystals is an $\sO_{Y/\cT}$-linear homomorphism $f\colon\sE' \rightarrow\sE$  is compatible with Frobenius structures.
	We denote by
		\[F\Isoc^\dagger(Y/\cT) =  F\Isoc^\dagger(Y/(T,\cT,\iota,\sigma))\]
	the category of log overconvergent $F$-isocrystals on $Y$ over $(T,\cT,\iota,\sigma)$.
\end{definition}

The structure sheaf $\sO_{Y/\cT}$ with a Frobenius structure given by the natural identification $\sigma^\ast_\isoc \sO_{Y/\cT}= \sO_{Y/\cT}$ provides a log overconvergent $F$-isocrystal, which we denote again by $\sO_{Y/\cT}$.

\begin{definition}\label{def: tensor and hom}
	Let $Y$ be a fine log scheme over $T$.
	For $\sE,\sE'\in\Isoc^\dagger(Y/\cT)$, the tensor product and internal $\Hom$ as $\sO_{Y/\cT}$-modules define the tensor product and internal $\Hom$ in the category of overconvergent isocrystals. 
	They are denoted by $\sE \otimes\sE'$ and $\sheafhom(\sE,\sE')$.
	
	For $(\sE,\Phi),(\sE',\Phi')\in F\Isoc^\dagger(Y/\cT)$, we define $(\sE,\Phi) \otimes(\sE',\Phi')$ by
	\begin{equation*}\label{eq: tensor phi}\sigma^\ast_\isoc (\sE \otimes\sE') :=  \sigma^\ast_\isoc \sE \otimes\sigma^\ast_\isoc \sE' \xrightarrow{\Phi \otimes\Phi'}\sE \otimes\sE',
	\end{equation*}
and $\sheafhom((\sE',\Phi'),(\sE,\Phi))$ by
	\begin{eqnarray*}
	\label{eq: hom phi}
	\sigma^\ast_\isoc \sheafhom(\sE',\sE)= \sheafhom(\sigma^\ast_\isoc \sE',\sigma^\ast_\isoc \sE)& \rightarrow&\sheafhom(\sE',\sE)\\
	\nonumber f&\mapsto&\Phi\circ f\circ\Phi'^{-1}.
	\end{eqnarray*}	
We call $\sE^\vee:= \sheafhom(\sE,\sO_{Y/\cT})$  and $(\sE,\Phi)^\vee:= \sheafhom((\sE,\Phi),\sO_{Y/\cT})$ the dual of $\sE$ and  $(\sE,\Phi)$ respectively.
\end{definition}

Log overconvergent ($F$-)isocrystals can be interpreted as locally free sheaves with log connections on dagger spaces.

\begin{definition}
	Let $Y$ be a fine log scheme over $T$.
	A {\it local embedding datum} for $Y$ over $(T,\cT,\iota)$ is a a finite family $(Z_\lambda,\cZ_\lambda,i_\lambda,h_\lambda,\theta_\lambda)_{\lambda\in\Lambda}$ of objects in $\OC(Y/\cT)$ such that $\{\theta_\lambda\colon Z_\lambda \rightarrow Y\}_\lambda$ is a Zariski open covering with $\theta_\lambda^*\cN_Y= \cN_{Z_\lambda}$ and the morphisms $h_\lambda\colon\cZ_\lambda \rightarrow\cT$ are log smooth.
	
	A {\it local embedding $F$-datum} for $Y$ over $(T,\cT,\iota,\sigma)$ is a finite family $(Z_\lambda,\cZ_\lambda,i_\lambda,h_\lambda,\theta_\lambda,\phi_\lambda)_{\lambda\in\Lambda}$ consisting of a local embedding datum $(Z_\lambda,\cZ_\lambda,i_\lambda,h_\lambda,\theta_\lambda)_{\lambda\in\Lambda}$ and Frobenius lifts $\phi_\lambda$ on $\cZ_\lambda$ which are compatible with $\sigma$.
\end{definition}

Note that local embedding ($F$-)data exist for any fine log scheme over $T$ \cite[Prop.\,2.34]{Yamada2020}.

\begin{definition}
	Let $\cZ$ be a weak formal log scheme over $\cT$. 
	\begin{enumerate}
	\item For a coherent locally free sheaf $\cE$  on $\cZ_\bbQ$
		an {\it integrable log connection}  over $\cT$ is an $\cO_{\cT_\bbQ}$-linear map
			\[\nabla\colon \cE \rightarrow\cE \otimes_{\cO_{\cZ_\bbQ}}\omega^1_{\cZ/\cT,\bbQ}\]
		such that
		\begin{itemize}
		\item $\nabla(f\alpha)= f\nabla(\alpha)+\alpha \otimes df$ for local sections $\alpha\in\cE$ and $f\in\cO_{\cZ_\bbQ}$,
		\item $\nabla^1\circ\nabla= 0$, where $\nabla^1\colon\cE \otimes\omega^1_{\cZ/\cT,\bbQ} \rightarrow\cE \otimes\omega^2_{\cZ/\cT,\bbQ}$ is induced from $\nabla$ by $\nabla(\alpha \otimes\eta):= \nabla(\alpha)\wedge\eta+\alpha \otimes d\eta$ for local sections $\alpha\in\cE$ and $\eta\in\omega^1_{\cZ/\cT,\bbQ}$.
		\end{itemize}	
	\item We define the category $\MIC(\cZ/\cT)$ as follows:
		\begin{itemize}
		\item An object of $\MIC(\cZ/\cT)$ is a pair $(\cE,\nabla)$ of a coherent locally free sheaf $\cE$ on $\cZ_\bbQ$ and an integrable log connection $\nabla$ on $\cE$ over $\cT$, 
		\item A morphism $f\colon(\cE',\nabla') \rightarrow(\cE,\nabla)$ in $\MIC(\cZ/\cT)$ is an $\cO_{\cZ_\bbQ}$-linear homomorphism $f\colon\cE' \rightarrow\cE$   is compatible with the log connections.
		\end{itemize}
	\item Assume that there is an endomorphism $\phi$ on $\cZ$ which lifts the absolute Frobenius on $\cZ\times_{W^\varnothing}k^\varnothing$.
		Then $\phi$ on $\cZ$ and $\sigma$ on $\cT$ together induce a functor $\phi^*\colon\MIC(\cZ/\cT) \rightarrow\MIC(\cZ/\cT)$.
		Let $(\cE,\nabla)\in\MIC(\cZ/\cT)$. 
		A {\it Frobenius structure} on $(\sE,\nabla)$ is an isomorphism $\Phi\colon\phi^*(\cE,\nabla) \xrightarrow{\cong}(\cE,\nabla)$ in $\MIC(\cZ/\cT)$.
		We denote  the category of such triples $(\cE,\nabla,\Phi)$ by $F\MIC(\cZ/\cT)$.
		For an object $(\cE,\nabla,\Phi)\in F\MIC((\cZ,\phi)/(\cT,\sigma))$, the composition $\cE \xrightarrow{\phi^*}\phi^*\cE \xrightarrow{\Phi}\cE$ together with the action of $\phi$ on $\omega^\bullet_{\cZ/\cT,\bbQ}$  induces a $\phi$-semilinear endomorphism on $\cE \otimes\omega^\bullet_{\cZ/\cT,\bbQ}$  often denoted by $\varphi$. 
	\item An object $(\cE,\nabla)\in\MIC(\cZ/\cT)$ (resp.\,$(\cE,\nabla,\Phi)\in F\MIC(\cZ/\cT)$) is said to be {\it overconvergent} if $\nabla$ is induced from a Taylor isomorphism.
		(See \cite[Def.\,2.19]{Yamada2020} and the construction before Prop.\,2.12 in {\it loc.\,cit.})
		We denote by $\MIC^\dagger(\cZ/\cT)\subset\MIC(\cZ/\cT)$ (resp.\,$F\MIC^\dagger(\cZ/\cT)\subset F\MIC(\cZ/\cT)$) the full subcategory of overconvergent objects.
	\end{enumerate}
\end{definition}

Let $Y$ be a fine log scheme over $T$, and $(Z_\lambda,\cZ_\lambda,i_\lambda,h_\lambda,\theta_\lambda)_{\lambda\in\Lambda}$  a local embedding datum for $Y$.
For $m\in\bbN$ and $\underline{\lambda}= (\lambda_0,\ldots,\lambda_m)\in\Lambda^{m+1}$, set
	\begin{equation}\label{eq: product in OC}
	(Z_{\underline{\lambda}},\cZ_{\underline{\lambda}},i_{\underline{\lambda}},h_{\underline{\lambda}},\theta_{\underline{\lambda}}):= (Z_{\lambda_0},\cZ_{\lambda_0},i_{\lambda_0},h_{\lambda_0},\theta_{\lambda_0})\times\cdots\times(Z_{\lambda_m},\cZ_{\lambda_m},i_{\lambda_m},h_{\lambda_m},\theta_{\lambda_m}),
	\end{equation}
where the product is taken in $\OC(Y/\cT)$ 
(see \cite[Prop.\,2.25]{Yamada2020} for products in the log overconvergent site).
If Frobenius lifts $\phi_\lambda$ are given, let $\phi_{\underline{\lambda}}$ be the Frobenius lift on $\cZ_{\underline{\lambda}}$ induced by $\phi_{\lambda_0}\times\cdots\times\phi_{\lambda_m}$ on $\cZ_{\lambda_0}\times_\cT\cdots\times_\cT\cZ_{\lambda_m}$.
We denote by
	\begin{align*}
	&\mathrm{pr}_j\colon(Z_{\underline{\lambda}},\cZ_{\underline{\lambda}},i_{\underline{\lambda}},h_{\underline{\lambda}},\theta_{\underline{\lambda}}) \rightarrow(Z_{\lambda_j},\cZ_{\lambda_j},i_{\lambda_j},h_{\lambda_j},\theta_{\lambda_j})&&\text{for }\ 0\leq j\leq m\\
	&\mathrm{pr}_{j,k}\colon(Z_{\underline{\lambda}},\cZ_{\underline{\lambda}},i_{\underline{\lambda}},h_{\underline{\lambda}},\theta_{\underline{\lambda}}) \rightarrow(Z_{\lambda_j},\cZ_{\lambda_j},i_{\lambda_j},h_{\lambda_j},\theta_{\lambda_j})\times(Z_{\lambda_k},\cZ_{\lambda_k},i_{\lambda_k},h_{\lambda_k},\theta_{\lambda_k})&&\text{for }\ 0\leq j<k\leq m,
	\end{align*}
the canonical projections.

\begin{definition}\label{def: MIC descent}
	Let $Y$ be a fine log scheme over $T$, and  $(Z_\lambda,\cZ_\lambda,i_\lambda,h_\lambda,\theta_\lambda)_{\lambda\in\Lambda}$ a local embedding datum for $Y$.
	We define the category $\MIC^\dagger((\cZ_\lambda)_{\lambda\in\Lambda}/\cT)$ as follows:
	\begin{itemize}
	\item An object of $\MIC^\dagger((\cZ_\lambda)_{\lambda\in\Lambda}/\cT)$ is a family $\{(\cE_\lambda,\nabla_\lambda),\rho_{\lambda_0,\lambda_1}\}$  of $(\cE_\lambda,\nabla_\lambda)\in\MIC^\dagger(\cZ_\lambda/\cT)$ for $\lambda\in\Lambda$ and isomorphisms $\rho_{\lambda_0,\lambda_1}\colon \mathrm{pr}_1^*(\cE_{\lambda_1},\nabla_{\lambda_1}) \xrightarrow{\cong}\mathrm{pr}_0^*(\cE_{\lambda_0},\nabla_{\lambda_0})$ for $(\lambda_0,\lambda_1)\in\Lambda^2$, such that for any $(\lambda_0,\lambda_1,\lambda_2)\in\Lambda^3$ the diagram
			\[\xymatrix{
			\mathrm{pr}_2^*(\cE_{\lambda_2},\nabla_{\lambda_2})\ar[rr]^{\mathrm{pr}^*_{0,2}(\rho_{\lambda_0,\lambda_2})}\ar[rd]_{\mathrm{pr}^*_{1,2}(\rho_{\lambda_1,\lambda_2})} && \mathrm{pr}_0^*(\cE_{\lambda_0},\nabla_{\lambda_0})\\
			& \mathrm{pr}_1^*(\cE_{\lambda_1},\nabla_{\lambda_1})\ar[ru]_{\mathrm{pr}_{0,1}^*(\rho_{\lambda_0,\lambda_1})}&
			}\]
		commutes.
		
	\item A morphism $f\colon \{(\cE'_\lambda,\nabla'_\lambda),\rho'_{\lambda_0,\lambda_1}\} \rightarrow\{(\cE_\lambda,\nabla_\lambda),\rho_{\lambda_0,\lambda_1}\}$ in $\MIC^\dagger((\cZ_\lambda)_{\lambda\in\Lambda}/\cT)$ is a family $f= \{f_\lambda\}_\lambda$ of morphisms $f_\lambda\colon(\cE'_\lambda,\nabla'_\lambda) \rightarrow(\cE_\lambda,\nabla_\lambda)$ in $\MIC^\dagger(\cZ_\lambda/\cT)$ which are compatible with $\rho'_{\lambda_0,\lambda_1}$ and $\rho_{\lambda_0,\lambda_1}$.
	\end{itemize}
	
	Similarly we define the category  $F\MIC^\dagger((\cZ_\lambda,\phi_\lambda)_{\lambda\in\Lambda}/(\cT,\sigma))$ for a local embedding $F$-datum $(Z_\lambda,\cZ_\lambda,i_\lambda,h_\lambda,\theta_\lambda,\phi_\lambda)_{\lambda\in\Lambda}$ for $Y$.
\end{definition}

\begin{proposition}[{\cite[Cor.\,2.36]{Yamada2020}}]\label{cor: descent realization}
	Let $Y$ be a fine log scheme over $T$.
	For a local embedding datum $(Z_\lambda,\cZ_\lambda,i_\lambda,h_\lambda,\theta_\lambda)_{\lambda\in\Lambda}$ for $Y$, there exists a canonical equivalence
		\begin{equation}
		\Isoc^\dagger(Y/\cT) \xrightarrow{\cong}\MIC^\dagger((\cZ_\lambda)_{\lambda\in\Lambda}/\cT).
		\end{equation}
	For a local embedding $F$-datum $(Z_\lambda,\cZ_\lambda,i_\lambda,h_\lambda,\theta_\lambda,\phi_\lambda)_{\lambda\in\Lambda}$, there exist canonical equivalences
		\begin{equation}\label{eq: descent realization}
		F\Isoc^\dagger(Y/\cT) \xrightarrow{\cong}F\MIC^\dagger((\cZ_\lambda,\phi_\lambda)_{\lambda\in\Lambda}/(\cT,\sigma)).
		\end{equation} 
\end{proposition}

\begin{definition}\label{def: cohomology}
	Let $Y$ be a fine log scheme over $T$, and let $\sE\in\Isoc^\dagger(Y/\cT)$.
	Let $(Z_\lambda,\cZ_\lambda,i_\lambda,h_\lambda,\theta_\lambda)_{\lambda\in\Lambda}$ be a local embedding datum for $Y$ over $\cT$, and for $m\geqslant 0$ and $\underline{\lambda}=(\lambda_0,\ldots,\lambda_m)\in\Lambda^{m+1}$ let $(Z_{\underline{\lambda}},\cZ_{\underline{\lambda}},i_{\underline{\lambda}},h_{\underline{\lambda}},\theta_{\underline{\lambda}})$ be the product in $\OC(Y/\cT)$ as in \eqref{eq: product in OC}.
	Then
	\[(Z_m,\cZ_m,i_m,h_m,\theta_m):=\coprod_{\underline{\lambda}\in\Lambda^{m+1}}(Z_{\underline{\lambda}},\cZ_{\underline{\lambda}},i_{\underline{\lambda}},h_{\underline{\lambda}},\theta_{\underline{\lambda}})\]
	for $m\geqslant 0$ form a simplicial object $(Z_\bullet,\cZ_\bullet,i_\bullet,h_\bullet,\theta_\bullet)$ of $\OC(Y/\cT)$.
	We define the log rigid cohomology
		\[R\Gamma_\rig(Y/\cT,\sE)=R\Gamma_\rig(Y/(T,\cT,\iota),\sE)\]
	of $Y$ over $\cT$ with coefficients in $\sE$ to be the complex associated to the co-simplicial complex $R\Gamma(\cZ_{\bullet,\bbQ},\sE_{\cZ_\bullet} \otimes\omega^\star_{\cZ_\bullet/\cT,\bbQ})$.
	This is in fact independent of the choice of a local embedding datum up to canonical quasi-isomorphisms \cite[Prop.\,2.38]{Yamada2020}, as an object of the derived category.
	
	For $(\sE,\Phi)\in F\Isoc^\dagger(Y/\cT)$, we may define
		\[R\Gamma_\rig(Y/\cT,(\sE,\Phi))=R\Gamma_\rig(Y/(T,\cT,\iota),(\sE,\Phi))\]
	in a similar manner by choosing a local embedding $F$-datum.
	The Frobenius structure defines a $\sigma$-semilinear endomorphism $\varphi$ on $R\Gamma_\rig(Y/\cT,(\sE,\Phi))$.
 \end{definition}

\begin{definition}\label{def: modules}
	\begin{enumerate}
	\item	We define a category $\Mod_F(N)$ as follows:
		An object of $\Mod_F(N)$ is an $F$-vector space equipped with an $F$-linear endomorphism $N$ called {\it monodromy operator}.
		A morphism in $\Mod_F(N)$ is an $F$-linear map compatible with $N$.
		
		We denote by $\Mod_F^\fin(N)$ the full subcategory of $\Mod_F(N)$ consisting of objects of finite dimension with nilpotent monodromy.
	\item We define a category $\Mod_F(\varphi,N)$ as follows:
		An object of $\Mod_F(\varphi,N)$ is an object of $\Mod_F(N)$ equipped with a $\sigma$-semilinear endomorphism $\varphi$ called {\it Frobenius operator}, such that $N\varphi= p\varphi N$.
		A morphism in $\Mod_F(\varphi,N)$ is an $F$-linear map compatible with $\varphi$ and $N$.
		
		We say a $(\varphi,N)$-module $M$ is \textit{finite} if it has finite dimension and its Frobeinus operator is bijective.
		We denote by $\Mod_F^\fin(\varphi,N)$ the full subcategory consisting of finite $(\varphi,N)$-modules.
		Note that the monodromy operator of a finite $(\varphi,N)$-module is automatically nilpotent.
	\end{enumerate}
\end{definition}

\begin{definition}\label{def: HK coh}
	\begin{enumerate}
	\item For any weak formal log scheme $\cZ$ over $\cS$, we define the \textit{Kim--Hain complex} $\omega^\bullet_{\cZ/W^\varnothing,\bbQ}[u]$ to be the CDGA on $\cZ_\bbQ$ generated by $\omega^\bullet_{\cZ/W^\varnothing,\bbQ}$ and degree zero elements $u^{[i]}$ for $i\geqslant 0$ with relations
		\begin{align*}
		u^{[0]}= 1,&&du^{[i+1]}= -d\log s\cdot u^{[i]},&&u^{[i]}\wedge u^{[j]}= \frac{(i+j)!}{i!j!}u^{[i+j]}.
		\end{align*}
		Let $N\colon \omega^\bullet_{\cZ/W^\varnothing,\bbQ}[u]\rightarrow\omega^\bullet_{\cZ/W^\varnothing,\bbQ}[u]$ be the $\omega^\bullet_{\cZ/W^\varnothing,\bbQ}$-linear endomorphism defined by $u^{[i+1]}\mapsto u^{[i]}$.
		When a Frobenius lift $\phi$ on $\cZ$ which is compatible with $\sigma$ on $\cS$ is given, we extend its action on $\omega^\bullet_{\cZ/W^\varnothing,\bbQ}$ to $\omega^\bullet_{\cZ/W^\varnothing,\bbQ}[u]$ by $u^{[i]}\mapsto p^iu^{[i]}$.
		
		For $m\geqslant 0$, we let $\omega^\bullet_{\cZ/W^\varnothing,\bbQ}[u]_m$ be the subcomplex of $\omega^\bullet_{\cZ/W^\varnothing,\bbQ}[u]$ whose degree $n$ component consists of sections of the form $\sum_{i=0}^m\eta_iu^{[i]}$ with $\eta_i\in\omega^\bullet_{\cZ/W^\varnothing,\bbQ}$.
		Then for any $(\cE,\nabla)\in\MIC^\dagger(\cZ/W^\varnothing)$, $\cE\otimes\omega^\bullet_{\cZ/W^\varnothing,\bbQ}[u]_m$ forms a complex, which is stable under $N$ (and $\varphi$ if we consider a Frobenius structure).
		
	\item\label{item: HK coh def} Let $Y$ be a fine log scheme over $k^0$, and let $\sE\in\Isoc^\dagger(Y/W^\varnothing)$.
		To a choice of local embedding datum for $Y$ over $\cS$, we may associate a simplicial object $(Z_\bullet,\cZ_\bullet,i_\bullet,h_\bullet,\theta_\bullet)$ of $\OC(Y/\cS)$ as in Definition \ref{def: cohomology}.
	Note that we may regard $(Z_\bullet,\cZ_\bullet,i_\bullet,h_\bullet,\theta_\bullet)$ as a simplicial object of $\OC(Y/W^\varnothing)$, and this can be used to compute log rigid cohomology of $Y$ over $W^\varnothing$, since $\cS$ is log smooth over $W^\varnothing$.
		The complex
		\[R\Gamma^\rig_\HK(Y,\sE):=\holim_n\hocolim_mR\Gamma(\cZ_{n,\bbQ},\sE_{\cZ_n}\otimes\omega^\bullet_{\cZ_n/W^\varnothing,\bbQ}[u]_m),\]
		where $\holim$ is taken for the face maps and $\hocolim$ is indexed by $\bbN$,
		 is independent of the choice of a local embedding datum up to canonical quasi-isomorphisms \cite[Prop.\,4.6]{Yamada2020}, hence defines an object of $D^+(\Mod_F(N))$.
		We call it the \textit{rigid Hyodo--Kato cohomology} of $Y$ with coefficients in $\sE$.
		
		For $(\sE,\Phi)\in F\Isoc^\dagger(Y/\cT)$, we may define $R\Gamma_\HK^\rig(Y,(\sE,\Phi))$ in a similar manner by choosing a local embedding $F$-datum.
		The Frobenius structure defines a $\sigma$-semilinear endomorphism $\varphi$ on $R\Gamma^\rig_\HK(Y,(\sE,\Phi))$.
		Therefore 	$R\Gamma^\rig_\HK(Y,(\sE,\Phi))$ gives an object of $D^+(\Mod_F(\varphi,N))$.
	\end{enumerate}
\end{definition}

\section{Log overconvergent isocrystals associated to log substructures}\label{section : log substructures}

In this section, we discuss how to associate log overconvergent isocrystals to certain log substructures of a given log structure.

\begin{definition}
	Let $X$ be a (weak formal) scheme and $\cM$ a log structure on $X$.
	We say that $\cM$ is {\it monogenic} if locally on $X$ there exists a chart of the form $\psi\colon\bbN_X\rightarrow\cM$.
	We call $\psi(1)$ a {\it (local) generator} of $\cM$. 
\end{definition}

\begin{definition}
	Let $f\colon \cZ'\rightarrow\cZ$ be a morphism of weak formal log schemes.
	For any log substructure $\cM\subset\cN_{\cZ'}$, we denote by $f_+\cM$ the push forward as a sheaf of sets (not as a log structure).
	Moreover, let $f_\star\cM$ be the preimage of $f_+\cM$ under the morphism $\cN_\cZ\rightarrow f_+\cN_{\cZ'}$ induced by $f$. 
	This is a log substructure of $\cN_\cZ$.
\end{definition}

\begin{proposition}\label{prop: monogenic}
	Let $(Z,\cZ,i)$ be a widening and $\cM\subset\cN_Z$ a monogenic log substructure.
	Then $i_\star\cM$ is monogenic.
	For any point $z\in\cZ$ and local generators $a,b\in(i_\star\cM)_z$, there exists a unique element $x\in\cO_{\cZ,z}^\times$ such that $a=xb$.
\end{proposition}

For the proof we will need the following lemma:

\begin{lemma}\label{lem: units}
	Let $R$ be a Noetherian ring with ideal $I$.
	Let $A$ be a pseudo-wcfg algebra with respect to $(R,I)$ and $J\subset A$ an ideal of definition.
	Then an element $x\in A$ is invertible if and only if its image in $A/J$ is invertible.
	In particular, the natural projection $A^\times \rightarrow (A/J)^\times$ is surjective.
\end{lemma}

\begin{proof}
	For any $x\in A$, we denote by $\ol{x}$ the image of $x$ in $A/J$.
	Let $a\in A$ and assume that $\ol{a}\in(A/J)^\times$.
	Take an element $b\in A$ such that $\ol{b}= \overline{a}^{-1}$.
	Since $\ol{ab}= 1$, we have $ab\in 1+J$. As $1+J\subset A^\times$ by \cite[Thm.\,1.6]{MonskyWashnitzer}, we have $a\in A^\times$.
\end{proof}

\begin{proof}[Proof of Proposition \ref{prop: monogenic}]
	By Lemma \ref{lem: units}, the natural morphism $\cO_{\cZ}^\times\rightarrow\cO_Z^\times$ is surjective, hence the morphism $\cN_{\cZ}\rightarrow\cN_\cZ\oplus_{\cO_\cZ^\times}\cO_Z^\times=\cN_Z$ is also surjective.
	Thus, by definition of $i_\star\cM$, the morphism $i_\star\cM\rightarrow\cM$ is surjective.
	So, for a local generator $a$ of $\cM$, there exists a local lift $\wt{a}$ of $a$ to $i_\star\cM$, which we may assume to be an element of $\Gamma(\cZ,i_\star\cM)$ by shrinking $\cZ$.
	To prove the first assertion, it suffices to show that the homomorphism $\psi\colon\bbN_\cZ\rightarrow i_\star\cM;\ 1\mapsto\wt{a}$ is a chart of $i_\star\cM$, that is, the log structure $\cM'$ associated to the composite $\bbN_\cZ\xrightarrow{\psi}i_\star\cM\rightarrow\cO_\cZ$ is isomorphic to $i_\star\cM$.
	For this, it suffices to prove that the induced homomorphism $\ol\psi\colon\cM'/\cO_\cZ^\times\rightarrow i_\star\cM/\cO_\cZ^\times$ is an isomorphism.
	Since $i^*\cM'=\cM$, we have a canonical isomorphism $\cM'/\cO_\cZ^\times\cong\cM/\cO_Z^\times$.
	On the other hand, $i_\star\cM/\cO_\cZ^\times$ is equal to the image of $i_\star\cM$ by the composite $\cN_\cZ\rightarrow\cN_\cZ/\cO_\cZ^\times\cong\cN_Z/\cO_Z^\times$, which is equal to the composite $\cN_\cZ\rightarrow\cN_Z\rightarrow\cN_Z/\cO_Z^\times$.
	Therefore, by definition of $i_\star\cM$ and the surjectivity of $\cN_\cZ\rightarrow\cN_Z$, we see that $i_\star\cM/\cO_\cZ^\times$ is canonically isomorphic to $\cM/\cO_Z^\times$. Thus $\ol\psi$ is an isomorphism, as required.
	
	Next we prove the second assertion.
	The uniqueness follows from the fact that the monoid $(i_\star\cM)_z$ is integral.
	We show the existence.
	There exist $m,n\in\bbN$ and $u,v\in\cO_{\cZ,z}^\times$ with $a=ub^m$ and $b=va^n$.
	Then we have $b=u^nvb^{mn}$.
	Since $(i_\star\cM)_z$ is integral, it follows that $mn=1$ or $b\in\cO_{\cZ,z}^\times$.
	In the former case, we obtain $m=1$ and hence $a=ub$, as required.
	In the latter case, the assertion holds because $ub^{m-1}\in\cO_{\cZ,z}^\times$.
\end{proof}

We have now established everything to associate an overconvergent isocrystal to a monogenic substructure of a log structure.
For this consider again a widening $(T,\cT,\iota)$ as a base.

\begin{definition}
	Let $Y$ be a fine log scheme over $T$, and $\cM\subset\cN_Y$ a monogenic log substructure.
	For any object $(Z,\cZ,i,h,\theta)\in\OC(Y/\cT)$, the log substructure $i_\star\theta^*\cM\subset\cN_\cZ$ is monogenic by Proposition \ref{prop: monogenic}.
	Let $\cO_\cZ(\cM)$ be the locally free $\cO_{\cZ}$-module locally generated by a generator of $i_\star\theta^*\cM$.
	More precisely, we take an open covering $\{\cU_\lambda\}_{\lambda\in\Lambda}$ of $\cZ$ such that $(i_\star\theta^*\cM)|_{\cU_\lambda}$ is generated by a section $a_\lambda\in\Gamma(\cU_\lambda,i_\star\theta^*\cM)$.
	To this we may associate an $\cO_{\cU_\lambda}$-module $e_{a_\lambda}\cO_{\cU_\lambda}$ where $e_{a_\lambda}$ is a free generator.
	If we take another generator $b_\lambda$ of $(i_\star\theta^*\cM)|_{\cU_\lambda}$, by Proposition \ref{prop: monogenic} there exists a unique element $x\in\Gamma(\cU_\lambda,\cO_\cZ^\times)$ such that $a_\lambda=xb_\lambda$, hence $e_{a_\lambda}\mapsto xe_{b_\lambda}$ induces a canonical isomorphism
		\begin{equation}\label{eq: basis}	e_{a_\lambda}\cO_{\cU_\lambda}\xrightarrow{\cong}e_{b_\lambda}\cO_{\cU_\lambda}.
		\end{equation}
	Let $\alpha\colon\cN_\cZ\rightarrow\cO_\cZ$ be the structure morphism of the log structure.
	There is a canonical $\cO_{\cU_\lambda}$-linear homomorphism
		\begin{equation}\label{eq: basis3}
		e_{a_\lambda}\cO_{\cU_\lambda}\rightarrow\cO_{\cU_\lambda},\ e_{a_\lambda}\mapsto\alpha(a_\lambda).
		\end{equation}
	For $\lambda,\lambda'\in\Lambda$, both of $a_\lambda$ and $a_{\lambda'}$ generate $(i_\star\theta^*\cM)|_{\cU_\lambda\cap\cU_{\lambda'}}$.
	Thus there is a canonical isomorphism
		\begin{equation}\label{eq: basis2}
		e_{a_\lambda}\cO_{\cU_\lambda\cap\cU_{\lambda'}}\cong e_{a_{\lambda'}}\cO_{\cU_\lambda\cap\cU_{\lambda'}}
		\end{equation}
	as in \eqref{eq: basis}.
	Therefore we may glue the modules $e_{a_\lambda}\cO_{\cU_\lambda}$, $\lambda\in\Lambda$, via the isomorphisms \eqref{eq: basis2}.
	The resulting locally free $\cO_\cZ$-module is independent of the choice of generators $a_\lambda$ up to canonical isomorphisms  \eqref{eq: basis}, and we denote it by $\cO_\cZ(\cM)$.
	It is clear that the homomorphisms \eqref{eq: basis3} glue as well resulting in a canonical $\cO_\cZ$-linear homomorphism
		\begin{equation}\label{eq: basis4}
		\cO_\cZ(\cM)\rightarrow\cO_\cZ.
		\end{equation}
	We denote by $\cO_{\cZ_\bbQ}(\cM)$ the locally free $\cO_{\cZ_\bbQ}$-module induced by $\cO_\cZ(\cM)$.
	
	For a morphism $f\colon (Z',\cZ',i',h',\theta')\rightarrow(Z,\cZ,i,h,\theta)$ in $\OC(Y/\cT)$, there exist isomorphisms
	\begin{align}\label{eq: pull-back OM}
	f^*\cO_{\cZ}(\cM)\xrightarrow{\cong}\cO_{\cZ'}(\cM),&&f^*\cO_{\cZ_\bbQ}(\cM)\xrightarrow{\cong}\cO_{\cZ'_\bbQ}(\cM)
	\end{align}
	induced by the isomorphisms
	\[f_\lambda^*(e_{a_\lambda}\cO_{\cU_\lambda})\xrightarrow{\cong}e_{f_\lambda^*a_\lambda}\cO_{f^{-1}\cU_\lambda},\]
	where $f_\lambda\colon f^{-1}(\cU_\lambda)\rightarrow\cU_\lambda$ denotes the restriction of $f$.
	Under the identification \eqref{eq: pull-back OM}, we define $\sO_{Y/\cT}(\cM)$ to be the sheaf on $\OC(Y/\cT)$ such that $\sO_{Y/\cT}(\cM)_\cZ= \cO_{\cZ_\bbQ}(\cM)$.
	By construction $\sO_{Y/\cT}(\cM)$ is in fact a log overconvergent isocrystal on $Y$ over $(T,\cT,\iota)$.
	The homomorphism \eqref{eq: basis4} induces a canonical $\sO_{Y/\cT}$-linear homomorphism
		\begin{equation}\label{eq: basis5}
		\sO_{Y/\cT}(\cM)\rightarrow\sO_{Y/\cT}.
		\end{equation}		
For any object $\sE\in\Isoc^\dagger(Y/\cT)$, we define
	\[\sE(\cM):=\sE\otimes\sO_{Y/\cT}(\cM).\]
\end{definition}

Let $d\colon\cO_{\cZ_\bbQ}(\cM)\rightarrow\cO_{\cZ_\bbQ}(\cM) \otimes\omega^1_{\cZ/\cT,\bbQ}$ be the log connection corresponding to $\sO_{Y/\cT}(\cM)$.
Note that this is compatible with the differential $d$ on $\cO_{\cZ_\bbQ}$ through the homomorphism \eqref{eq: basis4}, so there is no risk of confusion. 
Locally, $d$ is given by
\begin{equation}\label{eq: description of connection}
d(e_a)= e_a\otimes d\log a,
\end{equation}
where $a$ is a local generator of $i_\star\theta^*\cM$.
Indeed, suppose that $a\in\Gamma(\cZ,i_\star\theta^*\cM)$ is a generator, and let $\cZ(1)$ be the exactification of the diagonal embedding $Z\hookrightarrow\cZ\times\cZ$ and $p_i\colon \cZ(1)_\bbQ\rightarrow\cZ_\bbQ$ for $i=1,2$ the natural projections.
Let $u(a)\in\Gamma(\cZ(1),\cO_{\cZ(1)}^\times)\subset\Gamma(\cZ(1),\cN_{\cZ(1)})$ be the unique element satisfying $p_2^*(a)=p_1^*(a) u(a)$.
Then the composite
\[\cO_{\cZ(1)_\bbQ}\otimes p_2^{-1}(e_a\cO_{\cZ_\bbQ})\xrightarrow[\eqref{eq: pull-back OM}]{\cong} e_{p_2^*a}\cO_{\cZ(1)_\bbQ}\xrightarrow[\eqref{eq: basis}]{\cong} e_{p_1^*a}\cO_{\cZ(1)_\bbQ}\xleftarrow[\eqref{eq: pull-back OM}]{\cong} p_1^{-1}(e_a\cO_{\cZ_\bbQ})\otimes\cO_{\cZ(1)_\bbQ}\]
maps $1\otimes e_a$ to $e_a\otimes u(a)$.
Thus \eqref{eq: description of connection} follows from the description of the log connection associated to a log stratification given in \cite[(2.13)]{Yamada2020}.

Next we study the action of Frobenius in this context.

\begin{lemma}
Consider an $F$-widening $(T,\cT,\iota,\sigma)$ as a base, and let $Y$ be a fine log scheme over $T$.
There exists a canonical $\sO_{Y/\cT}$-linear homomorphism
	\begin{equation}\label{eq: pseudo Frobenius}
	\sigma^\ast_\isoc (\sO_{Y/\cT}(\cM))\rightarrow \sO_{Y/\cT}(\cM).
	\end{equation}
which is compatible with the canonical Frobenius structure  $\sigma^\ast_\isoc \sO_{Y/\cT}\rightarrow\sO_{Y/\cT}$ on $\sO_{Y/\cT}$.
\end{lemma}

\begin{proof}
Note that by definition we have
	\[\sigma^\ast_\isoc (\sO_{Y/\cT}(\cM))= \sO_{Y/\cT}(F_Y^*\cM)\]
where $F_Y$ is the absolute Frobenius of $Y$.
Since $F_Y^*\cM$ is locally generated by the $p$-th power of a local generator of $\cM$, we may write for any object $(Z,\cZ,i,h,\theta)\in\OC(Y/\cT)$  
	\begin{align}\label{eq: local gen}
	\cO_\cZ(F_Y^*\cM)|_\cU= e_{a^p}\cO_\cU,&&\cO_\cZ(\cM)|_\cU= e_a\cO_\cU
	\end{align}
by taking a generator $a$ of $(i_\star\theta^*\cM)|_\cU$ for sufficiently small open subsets $\cU\subset\cZ$.
Then the morphisms for varying $\cU$
	\[e_{a^p}\cO_\cU\rightarrow e_a\cO_\cU; \ e_{a^p}\mapsto \alpha(a)^{p-1}e_a\]
are independent of the generator $a$.
Therefore they glue to each other and define for any $(Z,\cZ,i,h,\theta)$ an $\cO_\cZ$-linear homomorphism
	\begin{equation}\label{eq: pseudo Frobenius on Z}
	\cO_\cZ(F_Y^*\cM)\rightarrow\cO_\cZ(\cM)
	\end{equation}
and hence an $\sO_{Y/\cT}$-linear homomorphism
	$$
	\sigma^\ast_\isoc (\sO_{Y/\cT}(\cM))\rightarrow \sO_{Y/\cT}(\cM).
	$$
By construction this is compatible with the canonical Frobenius structure  $\sigma^\ast_\isoc \sO_{Y/\cT}\rightarrow\sO_{Y/\cT}$ on $\sO_{Y/\cT}$, via the natural homomorphism \eqref{eq: basis5}.
\end{proof}

\begin{remark}
Note that \eqref{eq: pseudo Frobenius} is not a Frobenius structure in general, since it is not an isomorphism if $\cM$ is non-trivial.
\end{remark}

\begin{definition}
Let again $(T,\cT,\iota,\sigma)$ be an $F$-widening, 
let $Y$ be a fine log scheme over $T$.
For an object $(Z,\cZ,i,h,\theta)\in\OC(Y/\cT)$ such that $\cZ$ admits a Frobenius lift $\phi$,  we may define an $\cO_\cZ$-linear isomorphism
	\begin{equation}\label{eq: frob M}
	\phi^*(\cO_\cZ(\cM))\xrightarrow{\cong}\cO_\cZ(F_Y^*\cM)
	\end{equation}
by $\phi^*e_a\mapsto e_{\phi(a)}$ under the local identities \eqref{eq: local gen}, where we note that $\phi(a)$ is a local generator of  $i_\star \theta^*F_Y^*\cM$ as well.
Note that $\phi^*e_a$ denotes $1\otimes e_a\in\cO_\cZ\otimes_{\phi^{-1}\cO_\cZ}\phi^{-1}\cO_\cZ(\cM)$.
Hence  \eqref{eq: frob M} is well-defined because $\phi^*e_{xa}= \phi(x)e_a$ for any $x\in\cO_\cZ^\times$.
For any $(\sE,\Phi)\in F\Isoc^\dagger(Y/\cT)$, we denote by
	\begin{equation}\label{eq: phi on M}
	\varphi\colon\sE(\cM)_\cZ\rightarrow\sE(\cM)_\cZ
	\end{equation}
the composition of the natural $\phi$-semilinear homomorphism $\sE(\cM)_\cZ\rightarrow\phi^*(\sE(\cM)_\cZ)$ and the $\cO_\cZ$-linear homomorphism
	\[\phi^*(\sE(\cM)_\cZ) =  \phi^*\sE_\cZ\otimes\phi^*(\cO_\cZ(\cM))\rightarrow\sE_\cZ\otimes\cO_\cZ(\cM) = \sE(\cM)_\cZ\]
defined by \eqref{eq: frob M}, \eqref{eq: pseudo Frobenius on Z}, and $\Phi$ on $\sE_\cZ$.
Note that \eqref{eq: phi on M} depends on $\phi$.
\end{definition}

The following definition will lay the groundwork to define compactly supported log rigid cohomology. We loosely call it ``log rigid cohomology supported towards $\cM$''.

\begin{definition}\label{def: coh for monogenic}
	Let $(T,\cT,\iota)$ be a widening, 
	$Y$ a fine log scheme over $T$, 
	and $\cM\subset\cN_Y$ a monogenic log substructure.
	\begin{enumerate}
	\item For any $\sE\in\Isoc^\dagger(Y/\cT)$, we set $\sE(\cM):= \sE \otimes\sO_{Y/\cT}(\cM)$ and define the \textit{log rigid cohomology with compact support towards $\cM$} of $Y$ over $\cT$ with coefficients in $\sE$ to be
			\[R\Gamma_{\rig,\cM}(Y/\cT,\sE)=R\Gamma_{\rig,\cM}(Y/(T,\cT,\iota),\sE):= R\Gamma_\rig(Y/(T,\cT,\iota),\sE(\cM)).\]
		For the case $\sE= \sO_{Y/\cT}$, we simply write
			\[R\Gamma_{\rig,\cM}(Y/\cT):= R\Gamma_{\rig,\cM}(Y/\cT,\sO_{Y/\cT}).\]
			
		Suppose that a Frobenius lift $\sigma$ on $\cT$ is given.
		Let $(\sE,\Phi)\in F\Isoc^\dagger(Y/\cT)$, and take a local embedding $F$-datum $(Z_\lambda,\cZ_\lambda,i_\lambda,h_\lambda,\theta_\lambda,\phi_\lambda)_\lambda$ for $Y$ over $\cT$.
		Let $(Z_\bullet,\cZ_\bullet,i_\bullet,h_\bullet,\theta_\bullet)$ and $\phi_\bullet$ be the associated simplicial object of $\OC(Y/\cT)$ and the Frobenius lift on it, as in Definition \ref{def: cohomology}.
		Then $\varphi\colon\sE(\cM)_{\cZ_m}\rightarrow\sE(\cM)_{\cZ_m}$ of \eqref{eq: phi on M} and the endomorphisms on $\omega^\bullet_{\cZ_m/\cT,\bbQ}$ induced by $\phi_m$ for $m\geqslant 0$ together provide a $\sigma$-semilinear endomorphism 
			\[\varphi\colon R\Gamma_{\rig,\cM}(Y/\cT,\sE)\rightarrow R\Gamma_{\rig,\cM}(Y/\cT,\sE),\]
		which is independent of the choice of a local embedding $F$-datum.
		We denote by 
			\[R\Gamma_{\rig,\cM}(Y/\cT,(\sE,\Phi))=R\Gamma_{\rig,\cM}(Y/(T,\cT,\iota),(\sE,\Phi))\]
		the complex $R\Gamma_{\rig,\cM}(Y/\cT,\sE)$ equipped with the endomorphism $\varphi$.
	\item Assume that $(T,\cT,\iota,\sigma)= (k^0,W^\varnothing,\tau,\sigma)$.
		For any $\sE\in\Isoc^\dagger(Y/W^\varnothing)$, we define the \textit{Hyodo--Kato cohomology with compact support towards $\cM$} of $Y$ with coefficients in $\sE$ to be
			\[R\Gamma_{\HK,\cM}^{\rig}(Y,\sE):= R\Gamma_\HK^\rig(Y,\sE(\cM)),\]
		which is an object of $D^+(\Mod_F(N))$.
		
		For any $(\sE,\Phi)\in F\Isoc^\dagger(Y/W^\varnothing)$, \eqref{eq: phi on M} and the association $u^{[i]}\mapsto p^iu^{[i]}$ together induce a $\sigma$-semilinear endomorphism
			\[\varphi\colon R\Gamma_{\HK,\cM}^{\rig}(Y,\sE)\rightarrow R\Gamma_{\HK,\cM}^{\rig}(Y,\sE).\]
		We denote by $R\Gamma_{\HK,\cM}^{\rig}(Y,(\sE,\Phi))$ the complex $R\Gamma_{\HK,\cM}^{\rig}(Y,\sE)$ equipped with the endomorphism $\varphi$.
		This is an object of $D^+(\Mod_F(\varphi,N))$.
	\end{enumerate}
\end{definition}

%
\section{Log rigid cohomology with compact support}\label{sec: c rig}
%

In this section we will define and study log rigid cohomology with compact support for strictly semistable log schemes.
Here we recall the definition of a strictly semistable log scheme.

\begin{definition}\label{def: monogenic D}
\begin{enumerate}
	\item For integers $n \geqslant 1$ and $m \geqslant 0$, let $k^0(n,m)  \rightarrow  k^0$ be the morphism of fine log schemes induced by the commutative diagram
		\[\xymatrix{
		\bbN^n\oplus\bbN^m\ar[d]_-\alpha & \bbN\ar[d]^-{\alpha_0}\ar[l]_-\beta\\
		k[\bbN^n/\Delta(\bbN)\oplus\bbN^m] & k\ar[l],
		}\]
		where $\Delta\colon\bbN\rightarrow\bbN^n$ is the diagonal map, $\alpha_0$ is induced by the canonical chart of $k^0$, $\alpha$ is the composition of the natural morphisms $\bbN^n\oplus\bbN^m\rightarrow\bbN^n/\Delta(\bbN)\oplus\bbN^m\rightarrow k[\bbN^n/\Delta(\bbN)\oplus\bbN^m]$,
		and $\beta$ is the composition of $\Delta$ and the canonical injection $\bbN^n  \rightarrow \bbN^n\oplus\bbN^m$.
	
	\item A log scheme $Y$ over $k^0$ is called {\it strictly semistable}, if Zariski locally on $Y$ there exists a strict log smooth morphism $Y\rightarrow k^0(n,m)$ over $k^0$.
		We call the closed subset
		\[
		D:=\{y\in Y\mid\text{$\exists a\in\cN_{Y,y}$ $\forall b\in\cN_{Y,y}$ $ab$ is not contained in the image of $\Gamma(k^0,\cN_{k^0})\xrightarrow{f^\sharp} \cN_{Y,y}$ }\},\]
		where $f^\sharp$ is induced by the structure morphism $f:Y \rightarrow k^0$, 
		with the reduced structure the \textit{horizontal divisor} of $Y$.
		The horizontal divisor can be empty because we allow the case $m= 0$.

		We let $\cM_D\subset\cN_Y$ be the monogenic log substructure locally generated by the equation defining $D$. 
		In other words, $\cM_D$ is locally generated by the image of $(1,\ldots,1)\in\bbN^m$ under the homomorphism $\bbN^m_Y\rightarrow\cN_Y$.
		
		A {\it horizontal component} of $Y$ is a Cartier divisor on $Y$ whose support is contained in $D$ and which can not be written as a sum of two non-trivial Cartier divisors.
		Namely, a horizontal component is locally defined by the image of $1_i$ under the homomorphism $\bbN^m_Y\rightarrow\cN_Y\rightarrow\cO_Y$ for some $i=1,\ldots,m$, where $1_i$ is the element of $\bbN^m$ whose $i$-th component is $1$ and other components are $0$.
	\end{enumerate}\label{def: st over k}
\end{definition}

\begin{remark}
Let $Y/k^0$ be a strictly semistable log scheme. 
In the literature, the log structure on the underlying scheme of $Y$ associated to the horizontal divisor $D$, 
which we shall denote by 
$\cM_D^{\mathrm{classical}}$, is locally generated by the equations of the horizontal components of $D$.
On the other hand, the log structure $\cM_D$ introduced above is generated by the equation of $D$.
If $Y=k^0(n,m)$, we have a sequence $\cM_D\hookrightarrow \cM_D^{\mathrm{classical}} \hookrightarrow \cN_Y$ of log substructures on $Y$ induced from the injections between their charts $\bbN\xrightarrow{\Delta}\bbN^m\rightarrow\bbN^n\oplus\bbN^m$, where $\Delta$ denotes the diagonal map.
This applies also in the smooth case which we address at the end of this section. 
See \cite[Ex.\,2.5]{Kato1989} for the description of $\cM_D^{\mathrm{classical}}$ in this case.
\end{remark}

\begin{definition}\label{def: coh with compact supp}
	Let $Y$ be a proper strictly semistable log scheme over $k^0$.
	Let $(T,\cT,\iota)$ (resp.\,$(T,\cT,\iota,\sigma)$) be a widening (resp.\,an $F$-widening) equipped with a morphism $Y\rightarrow T$.
	(For example we consider $(T,\cT,\iota)= (k^\varnothing,W^\varnothing,\iota),(k^0,\cS,\tau),(k^0,W^0,i_0),(k^0,V^\sharp,i_\pi)$.)
	For any $\sE\in\Isoc^\dagger(Y/\cT)$ (resp.\,$(\sE,\Phi)\in F\Isoc^\dagger(Y/\cT)$), we define the \textit{log rigid cohomology with compact support} of $Y$ over $\cT$ with coefficients in $\sE$ (resp.\,$(\sE,\Phi)$) to be
		\begin{eqnarray}
		R\Gamma_{\rig,c}(Y/\cT,\sE)=R\Gamma_{\rig,c}(Y/(T,\cT,\iota),\sE)&:= &R\Gamma_{\rig,\cM_D}(Y/(T,\cT,\iota),\sE)\\
		\nonumber(\text{resp.\,}R\Gamma_{\rig,c}(Y/\cT,(\sE,\Phi))=R\Gamma_{\rig,c}(Y/(T,\cT,\iota),(\sE,\Phi))&:= &R\Gamma_{\rig,\cM_D}(Y/(T,\cT,\iota),(\sE,\Phi))).
		\end{eqnarray}
\end{definition}

\begin{definition}
	Let $Y$ be a proper strictly semistable log scheme over $k^0$.
	For any $\sE\in\Isoc^\dagger(Y/W^\varnothing)$ (resp.\,$(\sE,\Phi)\in F\Isoc^\dagger(Y/W^\varnothing)$), we define the \textit{rigid Hyodo--Kato cohomology with compact support} of $Y$ with coefficients in $\sE$ (resp.\,$(\sE,\Phi)$) to be
		\begin{eqnarray}
		R\Gamma_{\HK,c}^{\rig}(Y,\sE)&:= &R\Gamma_{\HK,\cM_D}^{\rig}(Y,\sE)\\
			\nonumber(\text{resp.\,}R\Gamma_{\HK,c}^{\rig}(Y,(\sE,\Phi))&:= &R\Gamma_{\HK,\cM_D}^{\rig}(Y,(\sE,\Phi))).
		\end{eqnarray}
\end{definition}

For a choice of a uniformiser $\pi\in V$, we use notations
	\[\Isoc^\dagger(Y/V^\sharp)_\pi:=\Isoc^\dagger(Y/(k^0,V^\sharp,i_\pi))\]
for a fine log scheme $Y$ over $k^0$ and
	\begin{align*}
	R\Gamma_\rig(Y/V^\sharp,\sE)_\pi:=R\Gamma_\rig(Y/(k^0,V^\sharp,i_\pi),\sE)
	\end{align*}
for $\sE\in\Isoc^\dagger(Y/V^\sharp)_\pi$.

We recall that a {\it branch of $p$-adic logarithm} over $K$ is a group homomorphism $\log\colon K^\times\rightarrow K$ whose restriction to $1+\frm$ is given by the Taylor expansion of the logarithm function.

\begin{definition}
	Let $Y$ be a fine log scheme over $k^0$, 
	$\pi\in V$ a uniformiser, 
	and $\log\colon K^\times\rightarrow K$ a branch of $p$-adic logarithm over $K$.
	Let $\sE\in\Isoc^\dagger(Y/W^\varnothing)$ and denote its inverse image to $\Isoc^\dagger(Y/V^\sharp)_\pi$ again by $\sE$. 
	The association $u^{[i]}\mapsto \frac{(-\log \pi)^i}{i!}$ induces a morphism	
		\[\Psi_{\pi,\log}\colon R\Gamma^\rig_\HK(Y,\sE)\rightarrow R\Gamma_\rig(Y/V^\sharp,\sE)_\pi.\]
	If $Y$ is proper and strictly semistable, we moreover obtain
		\[\Psi_{\pi,\log}\colon R\Gamma_{\HK,c}^{\rig}(Y,\sE)\rightarrow R\Gamma_{\rig,c}(Y/V^\sharp,\sE).\]
	We call these morphisms the {\it rigid Hyodo--Kato map}.
\end{definition}

As in the case without compact support, there exists a diagram
	\begin{equation}\label{eq: diag crig}
	\xymatrix{
	\ar@{}[rd]|>>>>\circlearrowright& R\Gamma_{\rig,c}(Y/W^\varnothing,\sE)\ar[ld]\ar[d]\ar[rd] &\ar@{}[ld]|>>>>\circlearrowright \\
	R\Gamma_{\rig,c}(Y/W^0,\sE) & R\Gamma^\rig_{\HK,c}(Y,\sE)\ar[l]\ar[r]^-{\Psi_{\pi,\log}} \ar[d]^\psi& R\Gamma_{\rig,c}(Y/V^\sharp,\sE)_\pi.\\
	\ar@{}[ru]|>>>>\circlearrowright& 
	R\Gamma_{\rig,c}(Y/\cS,\sE)\ar[lu]_{j_0^\ast}\ar[ru]^{j_\pi^\ast} &
	\ar@{}[lu]|>>>>{(*)}
	}\end{equation}
Here the morphisms between log rigid cohomologies with compact support are naturally induced as in \cite[(2.53)]{Yamada2020} 
(in particular the lower two morphisms are induced 
from the natural maps of weak formal log schemes 
$W^0\xrightarrow{j_0} \cS \xleftarrow{j_\pi} V^\sharp$), 
and the morphism $\psi:R\Gamma^\rig_{\HK,c}(Y,\sE)\rightarrow R\Gamma_{\rig,c}(Y/\cS,\sE)$ is induced from the morphism $\sE_{\cZ_n}\otimes\omega^\bullet_{\cZ_n/W^\varnothing,\bbQ}[u]\rightarrow\sE_{\cZ_n}\otimes\omega^\bullet_{\cZ_n/\cS,\bbQ}$ sending $u^{[i]}$ for $i>0$ to $0$, where $\cZ_n$ is taken as in Definition \ref{def: HK coh} \ref{item: HK coh def}.
All triangles in \eqref{eq: diag crig} except  $(*)$ commute.
The triangle $(*)$ commutes if the branch $\log$ is taken to be $\log(\pi)= 0$.

In order to show that $\Psi_{\pi,\log}$ induces a quasi-isomorphism after tensoring with $K$, we compute the cohomology with compact support of a strictly semistable log scheme by cohomology without compact support of horizontal divisors.

Let $Y$ be a strictly semistable log scheme over $k^0$ with the horizontal divisor $D$.
Let $Y= \bigcup_{i\in\Upsilon_Y}Y_i$ and $D= \bigcup_{j\in \Upsilon_D}D_j$ be the decompositions into irreducible components and horizontal components, respectively.
Let $(Z_\lambda,\cZ_\lambda,i_\lambda,h_\lambda,\theta_\lambda,\phi_\lambda)_{\lambda\in\Lambda}$ be a local embedding $F$-datum for $Y$ over $(k^0,\cS,\tau,\sigma)$, and $(Z_\bullet,\cZ_\bullet,i_\bullet,h_\bullet,\theta_\bullet)$ the associated simplicial object of $\OC(Y/\cS)$ with Frobenius lifts $\phi_\bullet$.
To any $m\in\bbN$ and $i\in\Upsilon_Y$, we may naturally associate a divisor $\cY_{m,i}$ of $\cZ_m$ lifting $Y_i$.
Indeed, each irreducible component defines a monogenic log substructure $\cM_i\subset\cN_Y$, and hence $i_{m,\star}\theta_m^*\cM_i$ is also monogenic, and the image of its local generators in $\cO_{\cZ_m}$ define a divisor $\cY_{m,i}$.
Similarly, to any $j\in\Upsilon_D$ we may associate a divisor $\cD_{m,j}$ of $\cZ_m$.
Note that $\cY_{m,i}$ and $\cD_{m,j}$ can be empty.

We denote by $Y^\flat$ the log scheme whose underlying scheme is that of $Y$ and whose log structure is generated by log substructures corresponding to all $i\in\Upsilon_Y$.
Similarly, we denote by $\cZ_m^\flat$ the weak formal log scheme whose underlying weak formal scheme is that of $\cZ_m$ and whose log structure is associated to $\cY_m:= \bigcup_{i\in\Upsilon_Y}\cY_{m,i}$.
For $J\subset\Upsilon_D$, let
	\begin{align*}
	D_J:= \bigcap_{j\in J}D_j,&&\cD_{m,J}:= \bigcap_{j\in J}\cD_{m,j}.
	\end{align*}
Endow $D_J$ with the pull-back log structure from $Y^\flat$, denoted by $D_J^\flat$.
Similarly, we denote by $\cD^\flat_{m,J}$ the weak formal log scheme given by $\cD_{m,J}$ equipped with the pull-back log structure of $\cZ^\flat_m$.
For $r\in\bbN$, let 
	\begin{align*}
	D^{(r),\flat}:= \coprod_{J\subset\Upsilon_D,\lvert J\rvert= r}D_J^\flat,&&\cD_m^{(r),\flat}:= \coprod_{J\subset\Upsilon_D,\lvert J\rvert= r}\cD_{m,J}^\flat,
	\end{align*}
and $\iota_m^{(r)}\colon\cD_{m,\bbQ}^{(r),\flat}\rightarrow\cZ_{m,\bbQ}^\flat$ be the canonical morphism.
Note that we define $D^{(0),\flat}:= Y^\flat$ and $\cD_m^{(0),\flat}:= \cZ_m^\flat$ for $r= 0$.
Then the complexes $\iota^{(r)}_{m,*}\omega^\bullet_{\cD_m^{(r),\flat}/W^\varnothing,\bbQ}$ and $\iota^{(r)}_{m,*}\omega^\bullet_{\cD_m^{(r),\flat}/\cS,\bbQ}$ for $r\in\bbN$ form co-simplicial complexes $\iota^{(\bullet)}_{m,*}\omega^\star_{\cD_m^{(\bullet),\flat}/W^\varnothing,\bbQ}$ and $\iota^{(\bullet)}_{m,*}\omega^\star_{\cD_m^{(\bullet),\flat}/\cS,\bbQ}$ of sheaves on $\cZ_{m,\bbQ}$.

Moreover, we let
	\begin{align*}
	&\cZ_{m,0}:= \cZ_m\times_\cS W^0,&&\cD_{m,0}^{(r),\flat}:= \cD_m^{(r),\flat}\times_\cS W^0,\\
	&\cZ_{m,\pi}:= \cZ_m\times_{\cS,j_\pi}V^\sharp,&&\cD_{m,\pi}^{(r),\flat}:= \cD_m^{(r),\flat}\times_{\cS,j_\pi}V^\sharp,
	\end{align*}
and let $\iota_{m,0}^{(r)}\colon\cD_{m,0,\bbQ}^{(r),\flat}\rightarrow\cZ_{m,0,\bbQ}$ and $\iota_{m,\pi}^{(r)}\colon\cD_{m,\pi,\bbQ}^{(r),\flat}\rightarrow\cZ_{m,\pi,\bbQ}$ be the canonical morphisms.
Then we obtain co-simplicial complexes $\iota^{(\bullet)}_{m,0,*}\omega^\star_{\cD_{m,0}^{(\bullet),\flat}/W^0,\bbQ}$ on $\cZ_{m,0,\bbQ}$ and $\iota^{(\bullet)}_{m,\pi,*}\omega^\star_{\cD_{m,\pi}^{(\bullet),\flat}/V^\sharp,\bbQ}$ on $\cZ_{m,\pi,\bbQ}$.

\begin{proposition}\label{prop: D}
	The composition of natural morphisms
		\begin{equation*}
		\cO_{\cZ_{m,\bbQ}}(\cM_D)\otimes\omega^\star_{\cZ_m/W^\varnothing,\bbQ}  \rightarrow \omega^\star_{\cZ_m/W^\varnothing,\bbQ}\rightarrow\iota_{m,*}^{(\bullet)}\omega^\star_{\cD_m^{(\bullet),\flat}/W^\varnothing,\bbQ},
		\end{equation*}
	where the first morphism is induced by \eqref{eq: basis4},
	induces a quasi-isomorphism
	\begin{equation}\label{eq: sr1}R\Gamma_{\rig,c}(Y/W^\varnothing) \xrightarrow{\sim} R\Gamma_\rig(D^{(\bullet),\flat}/W^\varnothing).\
	\end{equation}
	
	Similarly we have quasi-isomorphisms
	\begin{align}
		\label{eq: sr2}&R\Gamma_{\HK,c}^{\rig}(Y) \xrightarrow{\sim} R\Gamma_{\HK}^\rig(D^{(\bullet),\flat}),\\
		\label{eq: sr3}&R\Gamma_{\rig,c}(Y/\cS) \xrightarrow{\sim} R\Gamma_\rig(D^{(\bullet),\flat}/\cS),\\
		\label{eq: sr4}&R\Gamma_{\rig,c}(Y/W^0)  \xrightarrow{\sim} R\Gamma_\rig(D^{(\bullet),\flat}/W^0),\\
		\label{eq: sr5}&R\Gamma_{\rig,c}(Y/V^\sharp)_\pi \xrightarrow{\sim} R\Gamma_\rig(D^{(\bullet),\flat}/V^\sharp)_\pi.
		\end{align}
\end{proposition}

\begin{proof}
	We only prove the statement for \eqref{eq: sr1} as an example.
	The statement for \eqref{eq: sr3}, \eqref{eq: sr4}, and \eqref{eq: sr5} can be proved similarly.
	The statement for \eqref{eq: sr2} follows from \eqref{eq: sr1} by \cite[Lem.\,3.9 and Lem.\,3.14]{ErtlYamada}.
	
	Since the $\cD_m^{(r),\flat}$ for $m\geqslant 0$ give a local embedding datum for $D^{(r),\flat}$ over $\cS$, we have
		\[R\Gamma_\rig(D^{(r),\flat}/W^\varnothing)= R\Gamma(\cD^{(r),\flat}_{\bullet,\bbQ},\omega^\star_{\cD^{(r),\flat}_{(\bullet)}/W^\varnothing,\bbQ})= R\Gamma(\cZ_{\bullet,\bbQ},\iota^{(r)}_{\bullet,*}\omega^\star_{\cD^{(r),\flat}_\bullet/W^\varnothing,\bbQ}).\]
	Thus it is enough to show that the sequence of $\cO_{\cZ_{m,\bbQ}}$-modules
		\[0  \rightarrow \cO_{\cZ_{m,\bbQ}}(\cM_D)\otimes\omega^k_{\cZ_m/W^\varnothing,\bbQ}  \rightarrow \iota_{m,*}^{(0)}\omega^k_{\cD_m^{(0),\flat}/W^\varnothing,\bbQ}  \rightarrow \iota_{m,*}^{(1)}\omega^k_{\cD_m^{(1),\flat}/W^\varnothing,\bbQ}  \rightarrow \iota_{m,*}^{(2)}\omega^k_{\cD_m^{(2),\flat}/W^\varnothing,\bbQ}  \rightarrow \cdots\]
	is exact for any $k \geqslant 0$.
	For $n\geqslant 1$ and $n'\geqslant 0$, let $\cS(n,n')\rightarrow\cS$ be the morphism of fine weak formal log schemes induced by the commutative diagram
	\[\xymatrix{
	\bbN^n\oplus\bbN^{n'}\ar[d]^-\alpha&\bbN\ar[l]^-\beta\ar[d]^-{\alpha_0}\\
	W\llbracket s\rrbracket[\bbN^n,\bbN^{n'}]^\dagger/(s-\alpha\circ\beta(1))&W\llbracket s\rrbracket\ar[l]
	}\]
	where $\alpha_0$ is induced by the canonical chart of $\cS$, $\alpha$ is the natural inclusion, and $\beta$ is the composition of the diagonal map $\bbN\rightarrow\bbN^n$ and the canonical inclusion $\bbN^n\rightarrow\bbN^n\oplus\bbN^{n'}$.
	Applying the construction in \cite[\S 5.1]{GrosseKlonne2005}, we take a local embedding $F$-datum so that, for any $\lambda\in\Lambda$, there exists an integer $r_\lambda\geqslant 0$ and a strict \'{e}tale morphism
	$\cZ_\lambda \rightarrow \cS(n,n')\times_{W^\varnothing}W[\bbN^r]^{\dagger,\varnothing}$ over $\cS$, such that the image of $1$ under the composite
	\[\bbN\rightarrow\bbN^{n'}\rightarrow\Gamma(\cS(n,n'),\cN_{\cS(n,n')})\rightarrow\Gamma(\cZ_\lambda,\cN_{\cZ_\lambda})\rightarrow\Gamma(\cZ_\lambda,\cO_{\cZ_\lambda})\]
	lifts the equation of $D\cap Z_\lambda$.
	For $\underline{\lambda}=(\lambda_0,\ldots,\lambda_m)\in\Lambda^{m+1}$, 
	the natural projection $\cZ_{\underline{\lambda}}\rightarrow\cZ_{\lambda_0}$ is log smooth by construction of $\cZ_{\underline{\lambda}}$ and \cite[Prop.\,1.53(3) and Prop.\,1.56]{ErtlYamada}.
	Let $i'_{\lambda_0}\colon Z_{\underline{\lambda}}\hookrightarrow\cZ'_{\lambda_0}$ be the exactification of the composite $Z_{\underline{\lambda}}\hookrightarrow Z_{\lambda_0}\hookrightarrow\cZ_{\lambda_0}$.
	Then by the universality of the exactification, the projection induces a morphism $g\colon\cZ_{\underline{\lambda}}\rightarrow\cZ'_{\lambda_0}$, which is log smooth by \cite[Prop.\,1.53(4)]{ErtlYamada}.
	As $i_{\underline{\lambda}}$ and $i'_{\lambda_0}$ are homeomorphic exact closed immersions, $g$ is strcit, and hence the morphism between underlying weak formal schemes is smooth by \cite[Lem.\,1.59]{ErtlYamada}.
	By the strong fibration theorem \cite[Prop.\,1.38]{ErtlYamada}, there locally exists an integer $d\geqslant 0$ and an isomorphism $\cZ_{\underline{\lambda}}\xrightarrow{\cong}\cZ'_{\lambda_0}\times\Spwf W\llbracket x_1,\ldots,x_d\rrbracket$.
	As $\cZ_m=\coprod_{\underline{\lambda}\in\Lambda^{m+1}}\cZ_{\underline{\lambda}}$ and $\Spwf W\llbracket x_1,\ldots,x_d\rrbracket$ is \'{e}tale over $\Spwf W[x_1,\ldots,x_d]^\dagger$, we conclude that, locally on $\cZ_m$ there exists a strict \'{e}tale morphism $\cZ_m\rightarrow\cS(n,n')\times_{W^\varnothing}W[\bbN^\ell]^{\dagger,\varnothing}$ over $\cS$, where we put $\ell:=r_{\lambda_0}+d$, and $\cD_m$ is defined by the image of $1$ under the composite 
	\[\bbN\rightarrow\bbN^{n'}\rightarrow\Gamma(\cS(n,n'),\cN_{\cS(n,n')})\rightarrow\Gamma(\cZ_m,\cN_{\cZ_m})\rightarrow\Gamma(\cZ_m,\cO_{\cZ_m}),\]
	where the first map is the diagonal map.
	Then by \cite[Prop.\,1.54(2)(a)]{ErtlYamada} the sheaves $\omega^k_{\cD_m^{(r),\flat}/W^\varnothing}$ coincide with the pull-back of the corresponding sheaves on $\cS(n,n')\times_{W^\varnothing}W[\bbN^\ell]^{\dagger,\varnothing}$.
	
	Thus we may assume that
		\[\cZ_m= \cS(n,n')\times_{W^\varnothing}W[\bbN^\ell]^{\dagger,\varnothing}= \cS(n)\times_{W^\varnothing}\cV_{n'} \times_{W^\varnothing}W[\bbN^\ell]^{\dagger,\varnothing},\]
	where $\cS(n):= \cS(n,0)$ and $\cV_{n'}$ is $\Spwf W[\bbN^{n'}]^\dagger$ endowed with the log structure associated to the natural map $\bbN^{n'}  \rightarrow  W[\bbN^{n'}]^\dagger$.
	Then $v:= (1,\ldots,1)\in\bbN^{n'}$ generates $i_{m,\star}\theta_m^*\cM_D$.
	Let $\cV \subset\cV_{n'}$ be the closed weak formal subscheme defined by $v$, and for $r\geqslant 0$ let $\cV^{(r)}$ be the disjoint union of all intersections of $r$ irreducible components of $\cV$ endowed with the trivial log structure.
	Then we have
		\[\cD^{(r),\flat}_m= \cS(n)\times_{W^\varnothing}\cV^{(r)}\times_{W^\varnothing}W[\bbN^\ell]^{\dagger,\varnothing}.\]
	If we set $\xi_{i,1},\ldots,\xi_{i,a_i}$ be a free basis of $\omega^i_{\cS(n)\times W[\bbN^\ell]^{\dagger,\varnothing}/W^\varnothing}$, then we have
		\begin{eqnarray*}
		\cO_{\cZ_{m,\bbQ}}(\cM_D)\otimes\omega^k_{\cZ_m/W^\varnothing,\bbQ}
		&= &\bigoplus_{i,j\in\bbN}\xi_{i,j}\cO_{\cZ_m,\bbQ}\otimes_{\cO_{\cV_{n',\bbQ}}}v\omega^{k-i}_{\cV_{n'}/W^\varnothing,\bbQ},\\
		\omega^k_{\cD^{(r),\flat}/W^\varnothing,\bbQ}
		&= &\bigoplus_{i,j\in\bbN}\xi_{i,j}\cO_{\cZ_m,\bbQ}\otimes_{\cO_{\cV_{n',\bbQ}}}\omega^{k-i}_{\cV^{(r)}/W^\varnothing,\bbQ}.
		\end{eqnarray*}
	Thus it suffices to show that
		\begin{equation}\label{eq: Koszul}
		0  \rightarrow v\cdot\omega^j_{\cV_{n'}/W^\varnothing,\bbQ}  \rightarrow \omega^j_{\cV^{(0)}/W^\varnothing,\bbQ}  \rightarrow \omega^j_{\cV^{(1)}/W^\varnothing,\bbQ}  \rightarrow \omega^j_{\cV^{(2)}/W^\varnothing,\bbQ}  \rightarrow \cdots
		\end{equation}
	is exact for any $j\in\bbN$.
For this set $\cV_{n'}^{\mathrm{int}}:=\Spec F[\bbN^{n'}]$ endowed with the log structure associated to the natural map $\bbN^{n'}\rightarrow F[\bbN^{n'}]$.  
Let $\cV^{\mathrm{int}}$ be the closed subscheme defined by $v=(1,\ldots,1)\in\bbN^{n'}$ 
and for any integer $r\geq 0$ let $\cV^{(r),\mathrm{int}}$ be the disjoint union of all intersections of $r$ irreducible components of $\cV^{\mathrm{int}}$ endowed with the trivial log structure. 
In this situation the exactness of the sequence
\[
0\rightarrow v\cdot \omega^j_{\cV_{n'}^{\mathrm{int}}/F}\rightarrow \omega^j_{\cV^{(0),\mathrm{int}}/F}\rightarrow\omega^j_{\cV^{(1),\mathrm{int}}/F}\rightarrow\omega^j_{\cV^{(2),\mathrm{int}}/F}\rightarrow\cdots
\]
is well-known (see for example \cite[Ex.\,7.23 (1)]{PetersSteenbrink2007}).
By tensoring this with $F[\bbN^{n'}]^\dagger$ over $F[\bbN^{n'}]$ we obtain the exact sequence \eqref{eq: Koszul} because $F[\bbN^{n'}]^\dagger$ is flat over $F[\bbN^{n'}]$. 
\end{proof}	

In particular there exists the following spectral sequences:
		\begin{align}
		\nonumber&E_1^{i,j}= H^j_\rig(D^{(i),\flat}/W^\varnothing) \Longrightarrow H^{i+j}_{\rig,c}(Y/W^\varnothing), \\
		\nonumber&E_1^{i,j}= H^{\rig,j}_\HK(D^{(i),\flat}) \Longrightarrow H^{\rig,i+j}_{\HK,c}(Y),\\
		\label{eq: spectral seq}& E_1^{i,j}= H^j_\rig(D^{(i),\flat}/\cS) \Longrightarrow  H^{i+j}_{\rig,c}(Y/\cS),\\
		\nonumber& E_1^{i,j}= H^j_\rig(D^{(i),\flat}/W^0) \Longrightarrow  H^{i+j}_{\rig,c}(Y/W^0),\\
		\nonumber& E_1^{i,j}= H^j_\rig(D^{(i),\flat}/V^\sharp)_\pi  \Longrightarrow H^{i+j}_{\rig,c}(Y/V^\sharp)_\pi.
		\end{align}

Before we state the consequences of Proposition \ref{prop: D}, we introduce the notion of unipotent isocrystals as a reasonable class of coefficients.
\begin{definition}
	Let $(T,\cT,\iota)$ be a widening and let $Y$ be a fine log scheme over $T$.
	A log overconvergent isocrystal $\sE\in\Isoc^\dagger(Y/\cT)$ is said to be \textit{unipotent} if it is written as an iterated extension of $\sO_{Y/\cT}$.
	When we consider a Frobenius lift on $\cT$, a log overconvergent $F$-isocrystal $(\sE,\Phi)\in F\Isoc^\dagger(Y/\cT)$ is said to be unipotent if the underlying isocrystal $\sE$ is so.
	We denote by $\Isoc^\dagger(Y/\cT)^\unip$ (resp.\,$F\Isoc^\dagger(Y/\cT)^\unip$) the category of unipotent log overconvergent isocrystals (resp.\,$F$-isocrystals).
\end{definition}

\begin{corollary}\label{cor: compare rigid}
	Let $Y$ be a proper strictly semistable log scheme over $k^0$ and $\sE\in\Isoc^\dagger(Y/W^\varnothing)^\unip$.
	By abuse of notation we denote by $\sE$ also the unipotent log overconvergent isocrystal obtained via any base change.
	\begin{enumerate}
	\item The map $R\Gamma^\rig_{\HK,c}(Y,\sE)  \rightarrow  R\Gamma_{\rig,c}(Y/W^0,\sE)$ is a quasi-isomorphism.
	\item The morphisms in (\ref{eq: diag crig}) induce quasi-isomorphisms
		\begin{eqnarray*}
		&&R\Gamma_{\rig,c}(Y/\cS,\sE)  \otimes_{F\{s\},s\mapsto 0}F \xrightarrow{\cong}R\Gamma_{\rig,c}(Y/W^0,\sE),\\
		&&R\Gamma_{\rig,c}(Y/\cS,\sE)  \otimes_{F\{s\},s\mapsto \pi}K \xrightarrow{\cong}R\Gamma_{\rig,c}(Y/V^\sharp,\sE)_\pi,\\
		&&R\Gamma^\rig_{\HK,c}(Y,\sE)  \otimes_FF\{s\} \xrightarrow{\cong}R\Gamma_{\rig,c}(Y/\cS,\sE).
		\end{eqnarray*}
	\item For any branch of the $p$-adic logarithm $\log$ over $K$, the map $\Psi_{\pi,\log}$ induces a quasi-isomorphism $\Psi_{\pi,\log,K}\colon R\Gamma^\rig_{\HK,c}(Y,\sE)  \otimes_FK  \xrightarrow{\cong}  R\Gamma_{\rig,c}(Y/V^\sharp,\sE)_\pi$.
	\item The cohomology groups $H^i_{\rig,c}(Y/W^\varnothing,\sE)$, $H^{\rig,i}_{\HK,c}(Y,\sE)$, $H^i_{\rig,c}(Y/W^0,\sE)$ are finite dimensional $F$-vector spaces, $H^i_{\rig,c}(Y/\cS,\sE)$ is a free $F\{s\}$-module of finite rank, and $H^i_{\rig,c}(Y/V^\sharp,\sE)_\pi$ is a finite dimensional $K$-vector space.
		The monodromy operator on $H^{\rig,i}_{\HK,c}(Y,\sE)$ is nilpotent.
	\item When we consider a Frobenius structure $\Phi$ on $\sE$, the induced Frobenius operator $\varphi$ on $H^{\rig,i}_{\HK,c}(Y,(\sE,\Phi))$ is bijective.
	\end{enumerate}
\end{corollary}

\begin{proof}
	Since $\sE$ is unipotent, (i)--(iv) can be reduced to the case $\sE=\sO_{Y/W^\varnothing}$.
	Moreover by \cite[Prop.\,3.16]{Yamada2020}, (v) is also reduced to the case $\sE=\sO_{Y/W^\varnothing}$.
	In this case the assertions follow from the corresponding ones for non-compact supported cohomology \cite[Cor. 3.35, Lem. 3.36, Prop. 4.9]{ErtlYamada} via the spectral sequences (\ref{eq: spectral seq}).
\end{proof}

\begin{remark}
	By Theorem \ref{thm: ext compare}, the derived category $D^b(\Mod_F^\fin(\varphi,N))$ is equivalent to the full subcategory of $D^b(\Mod_F(\varphi,N))$ consisting objects whose cohomology groups belong to $\Mod_F^\fin(\varphi,N)$.
	Therefore by the corollary we may regard $R\Gamma_{\HK,c}(Y,(\sE,\Phi))$ as an object of $D^b(\Mod_F^\fin(\varphi,N))$.
	Similarly, we may regard $R\Gamma_{\rig,c}(Y/W^0,(\sE,\Phi))$ and  $R\Gamma_{\rig,c}(Y/W^\varnothing,(\sE,\Phi))$ as objects of $D^b(\Mod_F^\fin(\varphi))$.
\end{remark}

The following proposition can be interpreted as the contravariant functoriality for proper morphisms and the covariant functoriality for open immersions.

\begin{proposition}\label{prop: functoriality}
	Let $(T,\cT,\iota)$ be a widening equipped with a morphism $k^0\rightarrow T$.
	Let $f\colon Y'\rightarrow Y$ be a morphism between proper strictly semistable log schemes over $k^0$ with horizontal divisors $D'$ and $D$.
	Let $\sE\in\Isoc^\dagger(Y/\cT)$.
	\begin{enumerate}
	\item\label{item: contravariant} If $D'=f^*D$, then $f$ induces a canonical morphism
		\begin{equation}\label{eq: contravariant}
		R\Gamma_{\rig,c}(Y/\cT,\sE)\rightarrow R\Gamma_{\rig,c}(Y'/\cT,f^*\sE).
		\end{equation}
		When $(T,\cT,\iota)=(k^\varnothing,W^\varnothing,\iota)$ we also have
		\begin{equation}\label{eq: contravariant HK}
		R\Gamma_{\HK,c}^\rig(Y,\sE)\rightarrow R\Gamma_{\HK,c}^\rig(Y',f^*\sE).
		\end{equation}
	\item\label{item: covariant} Suppose that $f$ is identity as a morphism of underlying schemes and that $D$ is a union of horizontal components of $D'$.
		Then $f$ induces a canonical morphism
		\begin{equation}\label{eq: covariant}
		R\Gamma_{\rig,c}(Y'/\cT,f^*\sE)\rightarrow R\Gamma_{\rig,c}(Y/\cT,\sE).
		\end{equation}
		When $(T,\cT,\iota)=(k^\varnothing,W^\varnothing,\iota)$ we also have			\begin{equation}\label{eq: covariant HK}
		R\Gamma_{\HK,c}^\rig(Y',f^*\sE)\rightarrow R\Gamma_{\HK,c}^\rig(Y,\sE).
		\end{equation}
	\end{enumerate}
\end{proposition}

\begin{proof}
	 The point \ref{item: contravariant} follows from $\sO_{Y'/\cT}(\cM_{D'})=f^*\sO_{Y/\cT}(\cM_D)$ by the functoriality of log rigid cohomology with coefficients \cite[(2.53)]{Yamada2020}.
	 
	 We will prove \ref{item: covariant}.
	 Let $(Z'_\lambda,\cZ'_\lambda,i'_\lambda,h'_\lambda,\theta'_\lambda)$ be a local embedding datum for $Y'$ over $\cT$, and $(Z'_\bullet,\cZ'_\bullet,i'_\bullet,h'_\bullet,\theta'_\bullet)$ the associated simplicial object.
	For $m\geqslant 0$ we denote by $\cZ_m$ the underlying weak formal scheme of $\cZ'_m$ equipped with the log structure $\cN_{\cZ_m}:=i'_{m,\star}f^*\cN_Y$.
	Then we may compute the log rigid cohomology of $Y$ by $\cZ_\bullet$.
	
	For any $\sE\in\Isoc^\dagger(Y/\cT)$, we have $(f^*\sE)_{\cZ'_m}=\sE_{\cZ_m}$ and there exists a natural morphism $(f^*\sE)_{\cZ'_m}\otimes\cO_{\cZ'_{m,\bbQ}}(\cM_{D'})\otimes\omega^\bullet_{\cZ'_m/\cT,\bbQ}\rightarrow \sE_{\cZ_m}\otimes\cO_{\cZ_{m,\bbQ}}(\cM_D)\otimes\omega^\bullet_{\cZ_m/\cT,\bbQ}$ on $\cZ'_{m,\bbQ}=\cZ_{m,\bbQ}$.
	This induces the desired morphism.
\end{proof}

Next we summarize properties of the cup product on log rigid cohomology and Hyodo--Kato cohomology.
In all the cases mentioned below, the cup product is naturally induced from the wedge product of log differential forms, however the precise description is quite complicated, because we need to take into account simplicial construction and homotopy limit at the same time.
The details of the definition of the cup product and necessary proofs of its properties are given in Appendix \ref{app: cup product}.

For a fine log scheme $Y$ over $k^0$, let $(\cI(Y),\cR(Y,-),\cD)$ be one of the following:
\begin{itemize}
\item $\cI(Y)=\Isoc^\dagger(Y/\cT)$, $\cR(Y,\sE)=R\Gamma_\rig(Y/\cT,\sE)$ for $\sE\in\cI(Y)$, and $\cD=D^+(\Mod_F)$ where $\cT=W^\varnothing$ or $W^0$,
\item $\cI(Y)=F\Isoc^\dagger(Y/\cT)$, $\cR(Y,\sE)=R\Gamma_\rig(Y/\cT,(\sE,\Phi))$ for $\sE=(\sE,\Phi)\in\cI(Y)$, and $\cD=D^+(\Mod_F(\varphi))$ where $\cT=W^\varnothing$ or $W^0$,
\item $\cI(Y)=\Isoc^\dagger(Y/V^\sharp)_\pi$, $\cR(Y,\sE)=R\Gamma_\rig(Y/V^\sharp,\sE)_\pi$ for $\sE\in\cI(Y)$, and $\cD=D^+(\Mod_K)$,
\item $\cI(Y)=\Isoc^\dagger(Y/W^\varnothing)$, $\cR(Y,\sE)=R\Gamma^\rig_\HK(Y,\sE)$ for $\sE\in\cI(Y)$, and $\cD=D^+(\Mod_F(N))$,
\item $\cI(Y)=F\Isoc^\dagger(Y/W^\varnothing)$, $\cR(Y,\sE)=R\Gamma^\rig_\HK(Y,(\sE,\Phi))$ for $\sE=(\sE,\Phi)\in\cI(Y)$, and $\cD=D^+(\Mod_F(\varphi,N))$.
\end{itemize}

\begin{proposition}\label{prop: cup product HK}
	For any fine log scheme $Y$ over $k^0$ and objects $\sE,\sE'\in \cI(Y)$, there exists a morphism which we will call \emph{the cup product}
	\begin{equation}\label{eq: pairing general form}
	\cup\colon \cR(Y,\sE)\otimes^L\cR(Y,\sE')\rightarrow \cR(Y,\sE\otimes\sE')
	\end{equation}
	in $\cD$ satisfying the following properties:
	\begin{enumerate}
	\item\label{item: comm ass} (Anti)commutativity and associativity: The following diagrams commute:
			\begin{align*}
			&\xymatrix{\cR(Y,\sE)\otimes^L \cR(Y,\sE')\ar[r]^-\cup\ar[d]_-\theta^-\cong&\cR(Y,\sE\otimes\sE')\\
			\cR(Y,\sE')\otimes^L \cR(Y,\sE)\ar[ur]_-\cup&}\\
			&\xymatrix{
			\cR(Y,\sE)\otimes^L \cR(Y,\sE')\otimes \cR(Y,\sE'')\ar[r]^-{\cup\otimes\id}\ar[d]_-{\id\otimes\cup}&
			\cR(Y,\sE\otimes\sE')\otimes^L \cR(Y,\sE'')\ar[d]^-\cup\\
			\cR(Y,\sE)\otimes^L \cR(Y,\sE'\otimes\sE'')\ar[r]^-\cup&\cR(Y,\sE\otimes\sE'\otimes\sE'')
			}\end{align*}
		where $\theta$ is the quasi-isomorphism mapping $\alpha\otimes\beta$ to $(-1)^{ji}\beta\otimes\alpha$ if $\alpha$ and $\beta$ are cochains of degree $i$ and $j$, respectively,
	\item\label{item: funct} Functoriality: For a finite extension $k'$ of $k$, 
	a fine log scheme $Y'$ over $k'^0$, and a morphism $f\colon Y'\rightarrow Y$ over $k^0$
	we have a commutative diagram
		\[\xymatrix{
		\cR(Y,\sE)\otimes^L \cR(Y,\sE')\ar[r]^-\cup\ar[d]_-{f^\ast\otimes f^\ast}&\cR(Y,\sE\otimes\sE')\ar[d]^-{f^\ast}\\
		\cR(Y',f^\ast\sE)\otimes^L \cR(Y',f^\ast\sE')\ar[r]^-\cup&\cR(Y',f^\ast\sE\otimes f^\ast\sE').
		}\]
		\item\label{item: HK cup} The cup product $\cup$ commutes with the morphisms displayed in the following commutative diagram:
		\begin{equation}\label{eq: diag rig HK}\xymatrix{
		&R\Gamma_\rig(Y/W^\varnothing,\sE).\ar[ld]\ar[d]\ar[rd]&\\
		R\Gamma_\rig(Y/W^0,\sE)&R\Gamma_\HK^\rig(Y,\sE)\ar[l]\ar[r]^-{\Psi_{\pi,\log}}&R\Gamma_\rig(Y/V^\sharp,\sE)_\pi
		}\end{equation}
	\end{enumerate} 
\end{proposition}

\begin{proof}
	The cup product is defined by the formula \eqref{eq: def cup} via the identification of Proposition \ref{prop: M iso R}.
	Its well-definedness as a morphism in the derived category of vector spaces follows from the equation \eqref{eq: iota cup}.
	The compatibility with monodromy and Frobenius operators are discussed in Remark \ref{rem: monodromy cup} and Remark \ref{rem: Frobenius cup}, respectively.
	The properties of \ref{item: comm ass} is proved in Proposition \ref{prop: comm cup} and \ref{prop: ass cup}, the property \ref{item: funct} is given in Proposition \ref{prop: cup functorial}, and \ref{item: HK cup} is discussed in Remark \ref{rem: HK map cup}.
\end{proof}

When $Y$ is a proper strictly semistable log scheme over $k^0$ with horizontal divisor $D$, let $\cR_c(Y,\sE):=\cR(Y,\sE(\cM_D))$.
For $i\in\bbN$, we denote by $\cH^i(Y,\sE)$ and $\cH^i_c(Y,\sE)$ the $i$-th cohomology groups of $\cR(Y,\sE)$ and $\cR_c(Y,\sE)$, respectively.

\begin{corollary}\label{cor: functorial cup}
Let $Y$ be a proper strictly semistable log scheme over $k^0$ and $\sE\in\cI(Y)$.
Then the natural morphism $\sE\otimes\sE^\vee(\cM_D)\rightarrow\sO_{Y/W^\varnothing}(\cM_D)$ induces a morphism 
	\begin{equation}\label{eq: pairing dual}
	\cup\colon \cR(Y,\sE) \otimes^L \cR_c(Y,\sE^\vee)\rightarrow \cR_c(Y)
	\end{equation}
in $\cD$.
	This is compatible with the morphisms in the commutative diagrams \eqref{eq: diag rig HK} and the upper half of \eqref{eq: diag crig}: 
	In particular the following diagram commutes
	\[\xymatrix{
	R\Gamma_\HK(Y,\sE)\otimes^L R\Gamma_{\HK,c}(Y,\sE^\vee)\ar[r]^-\cup\ar[d]_-{\Psi_{\pi,\log}\otimes\Psi_{\pi,\log}}
	&R\Gamma_{\HK,c}(Y)\ar[d]^-{\Psi_{\pi,\log}}\\
	R\Gamma_\rig(Y/V^\sharp,\sE)\otimes^L R\Gamma_{\rig,c}(Y/V^\sharp,\sE^\vee)\ar[r]^-\cup&R\Gamma_{\rig,c}(Y/V^\sharp).}\]
	Moreover, let $f\colon Y'\rightarrow Y$ be a morphism of proper strictly semistable log schemes over $k^0$ with horizontal divisors $D'$ and $D$.
	\begin{enumerate}
	\item\label{item: pull cup} If $D'=f^\ast D$, then we have
		\[f^\ast\alpha\cup f^\ast\beta=f^\ast(\alpha\cup\beta)\]
		for any $\alpha\in \cH^i(Y,\sE)$ and $\beta\in \cH^j_c(Y,\sE^\vee)$, where $f^\ast$ for compact support cohomology is defined in Proposition \ref{prop: functoriality} \ref{item: contravariant}.
	\item\label{item: push cup} If $f$ is the identity as a morphism of the underlying schemes and $D$ is a union of horizontal components of $Y'$, 
	then we have
		\[f_\ast(f^\ast\alpha\cup\beta)=\alpha\cup f_\ast\beta\]
		for any $\alpha\in \cH^i(Y,\sE)$ and $\beta\in \cH^j_c(Y',f^\ast\sE')$, where $f_\ast$ is defined in Proposition \ref{prop: functoriality} \ref{item: covariant}.
	\end{enumerate}
\end{corollary}

\begin{proof}
	The assertions except for \ref{item: push cup} are a special case of Proposition \ref{prop: cup product HK}.
	For \ref{item: push cup}, let $\cZ'_m=\coprod_{\underline{\lambda}\in\Lambda^{m+1}}\cZ'_{\underline{\lambda}}$ and $\cZ_m=\coprod_{\underline{\lambda}\in\Lambda^{m+1}}\cZ_{\underline{\lambda}}$ as in the proof of Proposition \ref{prop: functoriality}.
	Note that $\cZ_{\underline{\lambda},\bbQ}=\cZ'_{\underline{\lambda},\bbQ}$ and $\sE_{\cZ_{\underline{\lambda}}}=f^\ast\sE_{\cZ'_{\underline{\lambda}}}$.
	Then $f^\ast$ and $f_\ast$ are induced by the natural injections
	\begin{align*}
	&\sE_{\cZ_{\underline{\lambda}}}\otimes\omega^\bullet_{\cZ_{\underline{\lambda}}/W^\varnothing,\bbQ}[u]\hookrightarrow\sE_{\cZ_{\underline{\lambda}}}\otimes\omega^\bullet_{\cZ'_{\underline{\lambda}}/W^\varnothing,\bbQ}[u],\\
	&\sE_{\cZ'_{\underline{\lambda}}}\otimes\cO_{\cZ_{\underline{\lambda},\bbQ}}(\cM_{D'})\otimes\omega^\bullet_{\cZ'_{\underline{\lambda}}/W^\varnothing,\bbQ}[u]\hookrightarrow \sE_{\cZ_{\underline{\lambda}}}\otimes\cO_{\cZ_{\underline{\lambda},\bbQ}}(\cM_D)\otimes\omega^\bullet_{\cZ_{\underline{\lambda}}/W^\varnothing,\bbQ}[u],
	\end{align*}
	respectively.
	Thus \ref{item: push cup} easily follows from the definition of the cup product.
\end{proof}

An indication that the theory of compact supports we propose in this paper is reasonable, 
is that we can recover Berthelot's rigid cohomology with compact support in appropriate cases.
To begin with, we consider the case of a smooth scheme with a simple normal crossing divisor.
In the context of compactly supported cohomology this situation was investigated in detail in \cite{NakkajimaShiho2008}.
We translate it into the technical framework developed in this paper. 

\begin{definition}\label{def: compact support for smooth scheme with sncd}
Let $Y$ be a smooth proper scheme over $k$, with a simple normal crossing divisor $D\subset Y$ in the sense that the irreducible components of $D$ are smooth over the base $k$.
Denote by $Y^\infty$ the log scheme over $k^\varnothing$  whose underlying scheme is $Y$ and whose log structure is associated to the divisor $D$.
Let $\cM_D$ be the monogenic log substructure generated by the generator of the ideal of $D$ in $Y$.

For the widening $(k^\varnothing, W^\varnothing,\iota)$ and any $\sE\in\Isoc^\dagger(Y^\infty/W^\varnothing)$ (resp.\,$(\sE,\Phi)\in F\Isoc^\dagger(Y^\infty/W^\varnothing)$), we define the \textit{log rigid cohomology with compact support} of $Y^\infty$ over $W^\varnothing$ with coefficients in $\sE$ (resp.\,$(\sE,\Phi)$) to be
		\begin{align*}
		R\Gamma_{\rig,c}(Y^\infty/W^\varnothing,\sE)&=R\Gamma_{\rig,c}(Y^\infty/(k^\varnothing,W^\varnothing,\iota),\sE):= R\Gamma_{\rig,\cM_D}(Y^\infty/(k^\varnothing,W^\varnothing,\iota),\sE)\\
		R\Gamma_{\rig,c}(Y^\infty/W^\varnothing,(\sE,\Phi)) &= R\Gamma_{\rig,c}(Y^\infty/(k^\varnothing,W^\varnothing,\iota),(\sE,\Phi)) := R\Gamma_{\rig,\cM_D}(Y^\infty/(k^\varnothing,W^\varnothing,\iota),(\sE,\Phi)).
		\end{align*}
\end{definition}

Let $U:=Y\backslash D$ be the open dense subscheme of $Y$ which is the complement of $D$.
Assume that $Y$ is projective. 
We want to show that $R\Gamma_{\rig,c}(Y^\infty/W^\varnothing)$ computes the compactly supported rigid cohomology of $U$ introduced by Berthelot in \cite[\S3]{Berthelot1986} which we denote here by $R\Gamma_{\mathrm{Ber},c}(U^\varnothing/W^\varnothing)$.

\begin{remark}
The notation $(-)^\varnothing$ indicates that we may regard any scheme over $k$ as a log scheme over $k^\varnothing$ endowed with the trivial log structure.  
For $U:=Y\backslash D$ the trivial log structure on $U^\varnothing$ coincides with the pull-back log structure coming from $Y^\infty$. 
\end{remark}

\begin{proposition}\label{prop: comparison rig Ber c}
Let $Y$ be a projective smooth scheme over $k$.
Let $D\subset Y$ be a  simple normal crossing divisor and set $U:=Y\setminus D$.
Then there exists a quasi-isomorphism 
	$$
	R\Gamma_{\rig,c}(Y^\infty/W^\varnothing) \xrightarrow{\cong} R\Gamma_{\mathrm{Ber},c}(U^\varnothing/W^\varnothing).
	$$
\end{proposition}

\begin{proof}
According to \cite[(3.1)(iii)]{Berthelot1986}
	 there is a distinguished triangle
	\[R\Gamma_{\Ber,c}(U^\varnothing/W^\varnothing)\rightarrow R\Gamma_{\Ber,c}(Y^\varnothing/W^\varnothing)\rightarrow R\Gamma_{\Ber,c}(D^\varnothing/W^\varnothing)\rightarrow\]
	and  as $Y$ and $D$ are proper, we have \cite[(3.1)(i)]{Berthelot1986}
	\begin{align*}
	R\Gamma_{\Ber,c}(Y^\varnothing/W^\varnothing)=R\Gamma_\rig(Y^\varnothing/W^\varnothing),&&R\Gamma_{\Ber,c}(D^\varnothing/W^\varnothing)=R\Gamma_\rig(D^\varnothing/W^\varnothing).
	\end{align*}
	Thus it suffices to show that
	\[R\Gamma_{\rig,c}(Y^\infty/W^\varnothing)\rightarrow R\Gamma_{\rig}(Y^\varnothing/W^\varnothing)\rightarrow R\Gamma_{\rig}(D^\varnothing/W^\varnothing)\rightarrow\]
	is a distinguished triangle.
	
	Since $Y$ is projective, there exists an integer $r\geqslant 0$ and a closed immersion $Y\hookrightarrow\bbP^r_W$.
    Let $Y^\infty\hookrightarrow\cZ$ be the weak completion of this immersion.
    We define a simple normal crossing divisor on $\cZ$ by lifting equations of $D$.
    Then $\cZ$ endowed with the log structure associated to $\cD$ is strongly log smooth over $W^\varnothing$.

	With Definition \ref{def: compact support for smooth scheme with sncd} we have
	\begin{align*}
	R\Gamma_{\rig,c}(Y^\infty/W^\varnothing) = R\Gamma(\cZ_\bbQ,\cO_{\cZ_\bbQ}(\cM_D)\otimes \omega^\star_{\cZ/W^\varnothing,\bbQ}),&&
	R\Gamma_\rig(Y^\varnothing/W^\varnothing)=R\Gamma(\cZ_\bbQ,\Omega^\star_{\cZ_\bbQ}).
	\end{align*}
	Therefore it suffices to show that
	\[R\Gamma(\cZ_\bbQ,\Omega^\star_{\cZ_\bbQ}/(\cO_{\cZ_\bbQ}(\cM_D)\otimes \omega^\star_{\cZ/W^\varnothing,\bbQ}))\cong R\Gamma_\rig(D^\varnothing/W^\varnothing),\]
	and this follows in the same way as \cite[Ex.\,7.23]{PetersSteenbrink2007}.
	Indeed, let $\cD^{(m)}$ (resp.\,$D^{(m)}$) be the disjoint union of $(m+1)$-fold intersections of the irreducible components of $\cD$ (resp.\,$D$), and $\iota_m\colon\cD^{(m)}\hookrightarrow\cZ$ be the closed immersion.
	Then one can see by  a similar local computation as in the proof of Proposition \ref{prop: D} that the sequence
	\[0\rightarrow \cO_{\cZ_\bbQ}(\cM_D)\otimes \omega^n_{\cZ/W^\varnothing,\bbQ} \rightarrow \Omega^n_{\cZ/W^\varnothing,\bbQ} \rightarrow \iota_{0,*}\Omega^n_{\cD^{(0)}_{\bbQ}} \rightarrow \iota_{1,*}\Omega^n_{\cD^{(1)}_{\bbQ}} \rightarrow \iota_{2,*}\Omega^n_{\cD^{(2)}_{\bbQ}}\rightarrow\cdots\]
	is exact for each $n\geqslant 0$.
	Thus we see that
	\[R\Gamma(\cZ_\bbQ,\Omega^\star_{\cZ_\bbQ} / (\cO_{\cZ_\bbQ}(\cM_D)\otimes \omega^\star_{\cZ/W^\varnothing,\bbQ}))\cong R\Gamma(\cD^{(\bullet)}_{\bbQ},\Omega^\star_{\cD^{(\bullet)}_{\bbQ}}) = R\Gamma_\rig(D^{(\bullet),\varnothing}/W^\varnothing) \cong R\Gamma_\rig(D^\varnothing/W^\varnothing),\]
	and this finishes the proof.
\end{proof}

%
\section{Crystalline cohomology with compact support revisited}\label{sec: crys}
%

From now on, a log scheme is considered with respect to \'{e}tale topology. 
In this section we will  provide an interpretation of Yamashita--Yasuda's construction of crystalline cohomology with compact support  \cite{YamashitaYasuda2014},
which is closer to the rigid construction above. 

For the benefit of the reader and to lay the groundwork for the statements concerning coefficients necessary in the context of this paper, we recall the basics from \cite[\S6]{ErtlYamada} which is a slight generalisation of the definitions in \cite{Shiho2008} and as such is very useful for comparison with log convergent cohomology.

Let $T\hookrightarrow\cT$ be an exact closed immersion (not necessarily homeomorphic) of a fine log scheme $T$ over $k$ into a Noetherian fine formal log scheme $\cT$ over $W$.
Assume that $\cT=\Spf A$ is affine.
Let $\cI\subset A$ be the ideal of $T$ in $\cT$ and take an ideal of definition $\cJ$ of $\cT$.
Let $A_{\PD}$ be the $\cJ$-adic completion of the $\PD$-envelope of $A$ with respect to $\cI$ over $W$ (which is endowed with the canonical $\PD$-structure on $pW$).
We denote by $\gamma$ the $\PD$-structure on $A_{\PD}$.
Set $\cT_{\PD}:=\Spf A_{\PD}$ equipped with the pull-back log structure of that on $\cT$.
Denote by $\cT_{\PD,n}$ the closed log subscheme of $\cT_{\PD}$ defined by the ideal $\cJ^nA_{\PD}$.
For a fine log scheme $Y$ over $T$ the log crystalline site $(Y/\cT_{\PD,n})_{\cris}$ is defined as usual. 
To define its ``limit'' we follow the definitions in \cite[Def.\,1.2 and 1.4]{Shiho2008}.

\begin{definition}\label{def: crystalline site}
Let $Y$ be a fine log scheme over $T$.
	\begin{enumerate}
	\item The log crystalline site $(Y/\cT_{\PD})_\cris$ has as objects quadruples $(U,L,i,\delta)$ where $U$ is a strictly \'{e}tale fine log scheme  over $Y$, $L$ is a fine log scheme over $\cT_{\PD,n}$ for some $n$, $i\colon U\hookrightarrow L$ is an exact closed immersion over $\cT_{\PD}$, and $\delta$ is a $\PD$-structure on the ideal of $U$ in $L$ which is compatible with $\gamma$.
	Morphisms in $(Y/\cT_{\PD})_\cris$ are defined in the usual way.
	Covering families are induced by the \'{e}tale topology on $L$.
	
	\item A sheaf $\sF$ on $(Y/\cT_{\PD})_\cris$ is equivalent to a datum of sheaves (with respect to the \'etale topology) $\sF_L$ on $L$ for each  $(U,L,i,\delta)\in(Y/\cT_{\PD})_\cris$ and appropriate transition maps. 
	Denote by $\sO_{Y/\cT_{\PD}}$ the structure sheaf of $(Y/\cT_{\PD})_\cris$. 
		An $\sO_{Y/\cT_{\PD}}$-module $\sE$ is a \textit{crystal} on $Y$ over $\cT_{\PD}$ if for any $f\colon (U',L',i',\delta')\rightarrow(U,L,i,\delta)$ in $(Y/\cT_{\PD})_\cris$ the natural homomorphism $f^*\sE_L\rightarrow \sE_{L'}$ is an isomorphism.
		
	\item	The category $\mathrm{Crys}(Y/\cT_{\PD})_\bbQ$ of isocrystals on $Y$ over $\cT_{\PD}$ is defined as follows:
		An object of $\mathrm{Crys}(Y/\cT_{\PD})_\bbQ$ is a crystal $\sE$ on $Y$ over $\cT_{\PD}$ such that $\sE_L$ is of finite presentation for any $(U,L,i,\delta)\in(Y/\cT_{\PD})_\cris$, and morphisms in $\mathrm{Crys}(Y/\cT_{\PD})_{\bbQ}$ are defined by
		\[\Hom_{\mathrm{Crys}(Y/\cT_{\PD})_\bbQ}(\sE,\sE'):=\Hom_{\cO_{Y/\cT_{\mathrm{PD}}}}(\sE,\sE')\otimes_\bbZ\bbQ.\]
		
	\item Assume that $\cT_{\PD}$ is endowed with a lift $\sigma: \cT_{\PD}\rightarrow \cT_{\PD}$ of the absolute Frobenius on $\cT_{\PD}$ modulo $p$. 
	Together with the absolute Frobenius $F_Y$ of $Y$ it induces a morphism of topoi $\sigma^\ast_\cris: \widetilde{(Y/\cT_{\PD})}_\cris \rightarrow \widetilde{(Y/\cT_{\PD})}_\cris$.
	In this situation, an $F$-crystal on $(Y/\cT_{\PD})$ is a crystal $\sE$ together with  a $\sO_{Y/\cT_{\PD}}$-linear map
	$$\Phi: \sigma^\ast_\cris\sE\rightarrow \sE$$
	such that there exists an integer $i\geqslant 0$ and an $\sO_{Y/\cT_{\PD}}$-linear map $V:\sE\rightarrow \sigma^\ast_\cris\sE$ with $V\circ \Phi=p^i$.
	\end{enumerate}
\end{definition}

\begin{remark}
	Let $\cT_n$ be the closed subscheme of $\cT$ defined by $\cJ^n$.
	According to \cite[Rem.\,6.2]{ErtlYamada} $\cT_{\PD,n}$ coincides with the log $\PD$-envelope of $T\hookrightarrow\cT_n$ over $W$ if $n$ is large enough.
\end{remark}

Let $Y$ be a fine log scheme over $T$ and $\sE$ a crystal on $(Y/\cT_{\PD})_{\cris}$.
The crystalline cohomology $R\Gamma_{\cris}(Y/\cT_{\PD},\sE)$ is the cohomology of $\sE$ on the crystalline site $(Y/\cT_{\PD})_{\cris}$ and similarly for the finite versions $R\Gamma_{\cris}(Y/\cT_{\PD,n},\sE)$.
Analogous to the non-logarithmic version in \cite[7.19 Prop.]{BerthelotOgus} we have
	\begin{equation}\label{equ: holim}
	R\Gamma_{\cris}(Y/\cT_{\PD},\sE) \cong \holim_n R\Gamma_{\cris}(Y/\cT_{\PD,n},\sE)
	\end{equation}
if $\sE$ is locally 
quasi-coherent in the sense that for every $(U,L,i,\delta)\in(Y/\cT_{\PD})_\cris$ the $\sO_L$-module $\sE_L$ is quasi-coherent \cite[Tag \href{https://stacks.math.columbia.edu/tag/07IS}{07IS}]{stacks}.

\begin{definition}\label{def: crys compact}
Let $Y$ be a fine  log scheme over $T$. 
Assume a monogenic log substructure $\cM\subset \cN_Y$ is given.
Let $(U,L,i,\delta)$ be  an object  of the  log crystalline site $(Y/ \cT_{\PD})_{\cris}$. 
One has $\cN_U/\cO_U^\times \cong \cN_L/ \cO_L^\times$ on $U_{\et}\cong L_{\et}$.
Hence the local generator $a_U$ of $\cM$ on $U$ lifts to a local section $a_L$ of $\cN_L$, 
which generates a monogenic log substructure of $\cN_L$. 
Let $\cO_L(\cM)$ be the locally free $\cO_L$-module generated by $a_L$. 
More precisely, to $a_L$ we associate an $\cO_L$-module, $e_{a_L}\cO_L$, 
where $e_{a_L}$ is a free generator. 
By what we said above, this module is independent of the choice of $a_L$ up to canonical isomorphisms.

Define the sheaf $\sO_{Y/\cT_{\PD}}(\cM)$ on the log crystalline site $(Y/ \cT_{\PD})_{\cris}$ by setting 
	$$
	\Gamma\left((U,L,i,\delta), \sO_{Y/\cT_{\PD}}(\cM)\right) :=  \Gamma(L,\cO_L(\cM)).
	$$
Similar to the rigid case, this is a crystal on $(Y/ \cT_{\PD})_{\cris}$ by definition. 
Moreover, let $\alpha_L:\cN_L\rightarrow \cO_L$ be the structure morphism. 
It induces a canonical $\cO_L$-linear homomorphism $\cO_L(\cM)\rightarrow \cO_L$ by mapping $e_{a_L}\mapsto \alpha_L(a_L)$. 
By the crystalline property, we obtain a canonical $\sO_{Y/\cT_{\PD}}$-linear morphism 
$$
\sO_{Y/\cT_{\PD}}(\cM) \rightarrow \sO_{Y/\cT_{\PD}}
$$
of crystals.
\end{definition}

The following definition will lay the groundwork to define compactly supported log crystalline cohomology. We loosely call it ``log crystalline cohomology supported towards $\cM$''.

\begin{definition}
Let $Y$ be a fine log scheme over $T$.
Assume a monogenic log substructure $\cM\subset \cN_Y$ is given.
For a crystal $\sE$  on $(Y/\cT_{\PD})_{\cris}$ set $\sE(\cM):= \sE\otimes \sO_{Y/\cT_{\PD}}(\cM)$ 
and define
$$R\Gamma_{\cris,\cM}(Y/\cT_{\PD}, \sE) :=  R\Gamma_{\cris}(Y/\cT_{\PD},\sE(\cM)).$$
If $\sE= \sO_{Y/\cT_{\PD}}$, $\sE$ is omitted in the notation, and we simply write $R\Gamma_{\cris,\cM}(Y/\cT_{\PD})$. 
\end{definition}

As a first application, we discuss log crystalline cohomology with compact support for a strictly semistable log scheme over $k^0$ \cite[\S4.2.2]{YamashitaYasuda2014} via the log crystalline site $(Y/W^0)_\cris$, which is obtained from Definition \ref{def: crystalline site} by setting $(T\hookrightarrow \cT)=(k^0\rightarrow W^0)$, with the ideals $(p)=\cJ=\cI$.
Because $W$ is already $p$-adically complete and the ideal  $(p)$ has divided powers we have $\cT_{\PD}=W^0$ as well.
Note that in this case the log scheme $Y$ is as nice as one can wish for: 
it is given over $k^0$, fine and of Cartier type over $T=k^0$ (or equivalently universally saturated \cite[(4.8)]{Kato1989}).

\begin{definition}
Let $Y$ be a proper strictly semistable log scheme over $k^0$ with horizontal divisor $D$ as in Definition \ref{def: st over k}, and $\sE$ a crystal on $(Y/W^0)_{\cris}$. 
Denote by $\cM_D\subset \cN_Y$ the monogenic log substructure associated to $D$ as defined in Definition \ref{def: st over k}.
The \textit{compactly supported log crystalline cohomology} of $Y$ over $W^0$ with coefficients in $\sE$ is
	$$
	R\Gamma_{\cris,c}(Y/W^0,\sE) :=   R\Gamma_{\cris,\cM_D}(Y/W^0,\sE).
	$$
If $\sE=\sO_{Y/W^0}$, we simply write $R\Gamma_{\cris,c}(Y/W^0) :=   R\Gamma_{\cris,\cM_D}(Y/W^0)$.
\end{definition}

The definition above is in fact a variant of a special case  of the compactly supported crystalline cohomology introduced in \cite[Def.\,5.4]{Tsuji1999}. 
As a consequence it satisfies Poincaré duality:
\begin{proposition}\label{rem: Poincare cris}
Let $Y$ be a proper strictly semistable log scheme over $k^0$ of pure dimension $n$.  For  a finite locally free crystal $\sE$ on $(Y/W^0)_{\cris}$ denote by  $\sE^\vee:=\sheafhom(\sE,\sO_{Y/W^0})$ its dual.
Then there is a canonical homomorphism
$$\Tr_{Y/k^0}: H^{2n}_{\cris,c}(Y/W^0) \rightarrow W$$
called trace, 
such that the pairing induced by the cup product and the trace morphism
$$H^i_{\cris}(Y/W^0,\sE)\times H^{2n-i}_{\cris,c}(Y/W^0,\sE^\vee) \rightarrow W$$
is perfect.
\end{proposition}
\begin{proof}
Let $D$ be the horizontal divisor of $Y$, and $\cM_D\subset \cN_Y$ the 
monogenic log substructure introduced in Definition \ref{def: monogenic D}. 
For an object $(U,L,i,\delta)$ of the crystalline site $(Y/W^0)_{\cris}$, let $a_U$ be a local generator of $\cM_D$ and $a_L$ a lift to $\cN_L$. 
We observe that in the case at hand, the image of $a_L$ under the structure morphism $\cN_L \rightarrow \cO_L$ is not a zero-divisor. 

Let $\sK_{Y/\cT_{\PD},\cM_D}$ be the ideal sheaf of the crystalline structure sheaf $\sO_{Y/\cT_{\PD}}$ given on a $\PD$-thickening $(U,L,i,\delta)$ as the ideal generated  by the image of a lift $a_L$ of the  local generator $a_U$ of $\cM_D$ under the structure morphism $\cN_L\rightarrow \cO_L$. 
Then for any locally quasi-coherent crystal $\sE$ on the restricted crystalline site \cite[\S6]{Tsuji1999}, one can prove that the sheaf $\sK_{Y/\cT_{\PD},\cM_D}\sE$ is a crystal as in \cite[Lem.\,5.3]{Tsuji1999}.
By the equivalence between the categories of crystals on the usual and restricted crystalline site  \cite[VI.\,Rem.\,2.1.4, Prop.\,1.6.3]{Berthelot1974} one obtains a crystal in the usual sense.
In particular, $\sK_{Y/\cT_{\PD},\cM_D}$ is a crystal. 
By definition it is locally generated by a non-zero-divisor, and therefore free. 
Hence $\sK_{Y/\cT_{\PD},\cM_D}$ can be identified with $\sO_{Y/\cT_{\PD}}(\cM)$. 
More generally, $\sK_{Y/\cT_{\PD},\cM_D}\sE$ for $\sE$ as above coincides with $\sE(\cM_D)$. 

In particular, for a finite locally free crystal $\sE$, 
the cohomology of $R\Gamma_{\cris,c}(Y/W^0,\sE)$ coincides 
with the compactly supported crystalline cohomology of \cite[Def.\,5.4]{Tsuji1999}. 
Denote by  $\sE^\vee:=\sheafhom(\sE,\sO_{Y/W^0})$ its dual.
Then the existence of the trace map which induces a perfect pairing 
as in the statement follows from \cite[Prop.\,5.5 and Thm.\,5.6]{Tsuji1999}.
\end{proof}

Let $D^{(\bullet),\flat}$ be defined as in the previous section:
Denote by $Y^\flat$ the log scheme whose underlying scheme is that of $Y$ and whose log structure is generated by its irreducible components $\Upsilon_Y$.
For a subset  $J\subset \Upsilon_D$ of irreducible components of $D$ let $D_J^\flat$ be the intersection $\bigcap_{j\in J}D_j$ endowed with the pull back log structure from $Y^\flat$. 
For $i\in\bbN$, respectively $0$, set
\begin{align}\label{equ : D^flat}
D^{(i),\flat}:=  \coprod_{J\subset \Upsilon_D, \vert J\vert = i}D_J^\flat,&&D^{(0),\flat}:=  Y^\flat.
\end{align}

\begin{lemma}\label{lem: cris D}
There is a quasi-isomorphism
	\begin{equation}\label{equ : iso cris}
	\eta: R\Gamma_{\cris,c}(Y/W^0)   \xrightarrow{\sim}     R\Gamma_{\cris}(D^{(\bullet),\flat}/W^0)
	\end{equation}
which induces a spectral sequence
$$
E_1^{i,j}=  H^j_{\cris}(D^{(i),\flat}/ W^0)  \Longrightarrow  H^{i+j}_{\cris,c}(Y/W^0).
$$
\end{lemma}
\begin{proof}
This is \cite[Thm.\,9.17]{YamashitaYasuda2014}.
It is the crystalline version of the fourth quasi-isomorphism of Proposition \ref{prop: D}. 
Both can be shown by the same formal argument  only taking care of 
the technical differences between the rigid and crystalline setup.
\end{proof}

Next we discuss the action of Frobenius on the cohomology $R\Gamma_{\cris,c}(Y/W^0)$ for a proper strictly semistable log scheme $Y$ over $k^0$.

\begin{lemma}\label{lem: frob}
There is an $\sO_{Y/W^0}$-linear homomorphism 
	$$
	\sigma^\ast_\cris\sO_{Y/W^0}(\cM_D) \rightarrow \sO_{Y/W^0}(\cM_D)
	$$
which is compatible with the canonical morphism $\sigma^\ast_\cris\sO_{Y/W^0} \rightarrow \sO_{Y/W^0}$. 
\end{lemma}
\begin{proof}
Similar to the rigid case, we have
	$$
	\sigma^\ast_\cris\sO_{Y/W^0}(\cM_D) = \sO_{Y/W^0}(F_Y^\ast \cM_D).
	$$
Since $F_Y^\ast \cM_D$ is locally generated by the $p$-th power of a local generator of $\cM_D$, we may write 
\begin{align*}
\cO_L(F^\ast_Y\cM_D)=e_{a_L^p}\cO_L,&& \cO_L(\cM_D)=e_{a_L}\cO_L
\end{align*}
for $(U,L,i,\delta)\in (Y/W^0)_{\cris}$, by taking a generator as in Definition \ref{def: crys compact}.
Then the morphism
$$
\cO_L(F^\ast_Y\cM_D) \rightarrow \cO_L(\cM_D); e_{a_L^p} \mapsto \alpha(a_L)^{p-1}e_{a_L}
$$
is independent of the generator $a_L$. 
In other words, by uniqueness up to units of the local generators of $\cM_D$ and $F_Y^\ast \cM_D$ respectively, we obtain the desired $\sO_{Y/W^0}$-linear homomorphism 
$\sigma^\ast_\cris \sO_{Y/W^0}(\cM_D) \rightarrow \sO_{Y/W^0}(\cM_D)$
which is by construction compatible with the canonical morphism $\sigma^\ast_\cris \sO_{Y/W^0} \rightarrow \sO_{Y/W^0}$.
\end{proof}

\begin{corollary}
Let $(\sE,\Phi)$ be an $F$-crystal on $(Y/W^0)_{\cris}$.
There is an $\sO_{Y/W^0}$-linear homomorphism
 	\begin{equation}\label{equ: frob}
 	\Phi: \sigma^\ast_\cris \sE(\cM_D)=\sigma^\ast_\cris\sE\otimes \sigma^\ast_\cris \sO_{Y/W^0}(\cM_D) \rightarrow \sE\otimes \sO_{Y/W^0}(\cM_D) = \sE(\cM_D)
 	\end{equation}
compatible with the canonical morphism $\sigma^\ast_\cris \sO_{Y/W^0} \rightarrow \sO_{Y/W^0}$. 
\end{corollary}

\begin{definition}
Let $Y$ be a proper strictly semistable log scheme over $k^0$ and $(\sE,\Phi)$ an $F$-crystal on $(Y/W^0)_{\cris}$. 
The Frobenius action 
$$\varphi: R\Gamma_{\cris,c}(Y/W^0,\sE)\rightarrow R\Gamma_{\cris,c}(Y/W^0,\sE)$$
 is induced by the absolute Frobenius $F_Y$ on $Y$ and the Frobenius $\sigma$ on $W^0$ and the morphism (\ref{equ: frob}).
 More precisely, $\varphi$ is given as the composition
 \begin{align*}
R\Gamma_{\cris}(Y/W^0,\sE(\cM_D)) &\xrightarrow{\sigma} R\Gamma_{\cris}(Y^{(\sigma)}/W^0,\sE(\cM_D)^{(\sigma)}) \xrightarrow{F_{Y/k^0}^\ast} \\
&\rightarrow  R\Gamma_{\cris}(Y/W^0,\sigma^\ast_\cris\sE(\cM_D)) \xrightarrow{\Phi}  R\Gamma_{\cris}(Y/W^0,\sE(\cM_D)).
\end{align*}
\end{definition}

\begin{lemma}\label{lemma Frob W}
Let $Y$ be a proper strictly semistable log scheme over $k^0$. 
Then the Frobenius action induces a rational quasi-isomorphism
$$
\varphi: R\Gamma_{\cris,c}(Y/W^0)_{\bbQ} \xrightarrow{\sim} R\Gamma_{\cris,c}(Y/W^0)_{\bbQ}.
$$
\end{lemma}

\begin{proof}
This was stated and proved in \cite[Lem.\,9.20]{YamashitaYasuda2014}. 
It can be deduced easily from the bijectivity in the non-compactly supported case \cite[Prop.\,2.24]{HyodoKato1994}  via the quasi-isomorphism (\ref{equ : iso cris}).
\end{proof}

In order to provide the Hyodo--Kato map we are led to study in more detail the situation of a strictly semistable log scheme over the complete discrete valuation ring $V$.

Consider the  exact closed immersion $(T\hookrightarrow \cT)= (V_1^\sharp \hookrightarrow V^\sharp)$ of log schemes induced by reduction modulo $p$ with ideals $\cI=\cJ=(p)$, 
and its PD-envelope $\cT_{\PD}$, which coincides with $V^\sharp$. 
For a fine log scheme $X$ over $V^\sharp$, we will consider the log crystalline site $(X_1/V^\sharp)_{\cris}$ obtained from Definition \ref{def: crystalline site}.

As in \cite[Def.\,6.3]{ErtlYamada}, we will consider a crystalline site over an affine formal log scheme on which intermediate objects for the Hyodo--Kato map are defined. 

\begin{definition}
Consider the exact closed immersion of fine formal log schemes 
$(T\hookrightarrow \cT) = (V_1^\sharp \hookrightarrow \cS)$  
given by the ideal $\cI=(p,f)$ where $f=s^e+ \sum_{i=0}^{e-1}a_is^i\in W[s]$  is the Eisenstein polynomial of a uniformiser $\pi$ of $V$, and $e$ the ramification index of $V$. 
Note that this immersion depends on $\pi$.
By the formalism described at the beginning of the section, 
we obtain a ring $R_{\PD}$   by completing with respect to the ideal $\cJ=\cI=(p,f)$ the $\PD$-envelope of $(V_1^\sharp \hookrightarrow \cS)$
for the ideal $\cI=(p,f)$. 
Since $(p,s)^e\subset (p,f)\subset (p,s)$, it is clear that $R_{\PD}$ is also $(p,s)$-adically complete. 
Its  $\PD$-ideal is given as the closure of $(p, \frac{f^n}{n!}\;\vert\;  n\geqslant 1)  =( p,\frac{s^{en}}{n!}\;\vert\; n\geqslant 1)$.
There is a lifting of Frobenius $\sigma$ to $R_{\PD}$ induced by $s\mapsto s^p$ extending the canonical Frobenius on $W$ which is a $\PD$-morphism.

Denote by $\cS_{\PD}$  the formal scheme $\Spf R_{\PD}$ with the log structure generated by $s$, and by  $\cS_{\PD,n}$ its reduction modulo $p^n$.
\end{definition}

\begin{remark}
According to \cite[Rem.\,6.4 and Lem.\,6.5]{ErtlYamada} the log scheme $\cS_{\PD}$ coincides with the log scheme $\cS'_{\PD}$ 
where the underlying scheme scheme is the $p$-adic completion of the PD-envelope (over $W$ with the canonical PD-structure) of the exact closed immersion of fine formal log schemes 
$(T\hookrightarrow \cT)=(V_1^\sharp \hookrightarrow \Spf \widehat{W[s]})$, 
with the log structure generated by $s$.
This means in particular that  $\cS_{\PD}$ is $p$-adic. 
While the construction via $(V_1^\sharp \hookrightarrow \Spf \widehat{W[s]})$ is more classical, we chose the approach via $(V_1^\sharp \hookrightarrow \cS)$ to emphasis the relation with the rigid approach where $\cS$ is used instead of $\Spf \widehat{W[s]}$. 
\end{remark}

\begin{remark}\label{rem: breve{S}}
We will be interested in the crystalline cohomology for log schemes over $V^\sharp$, that is we consider the $\PD$-thickening $V^\sharp_1\hookrightarrow \cS_{\PD}$. 
Later we will also consider the unramified case of the above construction as a special case. 
More precisely, we consider the exact closed immersion of fine formal log schemes 
$(T\hookrightarrow \cT)=(k^0\hookrightarrow \cS)$ given by the ideal $\cI=(p,s)$, 
and denote by $\breve{\cS}_{\PD}$ the log scheme 
obtained by completing 
the PD-envelope of $(T\hookrightarrow \cT)=(k^0\hookrightarrow \cS)$ 
with respect to the ideal $\cJ=(p)$. 
As before, we denote by $\breve{R}_{\PD}$ the associated PD-rings and by $\sigma$ the Frobenius lifting induced by $s\mapsto s^p$. 
There is a natural morphism of fine formal log schemes 
$$\breve{\cS}_{\PD} \rightarrow \cS_{\PD}$$
given by inclusion on the level of rings.
\end{remark}

Let $\pi$ be a uniformiser of $V$. 
There is a comutative diagram of formal log schemes
\begin{equation}\label{diag: comm}
\xymatrix{
&k^0 \ar[dl]_{i_0} \ar[dr]^{i_\pi} \ar[d]^\tau&\\
W^0  \ar[r]_{j_0}& \cS_{\PD}& V^\sharp  \ar[l]^{j_\pi}.
}
\end{equation}
where the horizontal maps are defined via  $s \mapsto  0$ and $s \mapsto  \pi$, respectively,  
$i_0$ is the canonical embedding, 
$\tau$ is the composition $j_0\circ i_0$, 
and $i_\pi$ be the unique morphism such that $\tau=j_\pi\circ i_\pi$.  

\begin{definition}
\begin{enumerate}
\item 
For a uniformiser $\pi\in V$ and integers $n \geqslant 1$ and $m \geqslant 0$, let $V_\pi^\sharp(n,m)  \rightarrow  V^\sharp$ be the morphism of fine log schemes induced by the diagram
	\[
	\xymatrix{
	\bbN^n\oplus\bbN^m\ar[d]_-\alpha & \bbN\ar[d]^-{\alpha_\pi}\ar[l]_-\beta\\
	V[\bbN^n\oplus\bbN^m]/(\pi-\beta(1)) & V\ar[l],
	}
	\]
where $\alpha_\pi$ is defined by $1\mapsto\pi$, 
$\beta$ is the composition of the diagonal map $\bbN  \rightarrow \bbN^n$ and the canonical injection $\bbN^n  \rightarrow \bbN^n\oplus\bbN^m$, 
and $\alpha$ is the composition of the natural inclusion $\bbN^n\oplus\bbN^m \rightarrow V[\bbN^n\oplus\bbN^m]$ with the quotient map $V[\bbN^n\oplus\bbN^m] \rightarrow V[\bbN^n\oplus\bbN^m]/(\pi-\beta(1)) $.
			
\item A log scheme $X$ over $V^\sharp$ is called {\it strictly semistable}, 
if Zariski locally on $X$ there exists a strict log smooth morphism $X \rightarrow V^\sharp_\pi(n,m)$ over $V^\sharp$.
Note that this definition is independent of $\pi$.
Using the local description above, we call the closed subset of $X$
\[
C:=\{x\in X\mid \text{$\exists a\in\cN_{X,x}$ $\forall b\in\cN_{X,x}$  $ab$ is not contained in the image of $\cN_{V^\sharp,f(x)} \xrightarrow{f^\sharp}\cN_{X,x}$ }\},
\]
where $f^\sharp$ is induced by the structure morphism $f:X\rightarrow V^\sharp$,
with the reduced structure the \textit{horizontal divisor} of $X$.
The horizontal divisor can be empty as we allow the case $m= 0$.
Let $\cM_C \subset \cN_X$ be the monogenic log substructure locally generated by 
the equation of $C$. 
In other words, it is locally generated by the image of $(1,\ldots,1)\in\bbN^m$ under the homomorphism $\bbN^m_X\rightarrow\cN_X$.
\end{enumerate}
\end{definition}

Let $X$ be a proper  strictly semistable log scheme over $V^\sharp$ with horizontal divisor $C$.
For a choice of uniformiser, $Y_\pi:= X\times_{V^\sharp,i_\pi}k^0$ is  a proper strictly semistable log scheme over $k^0$ with horizontal divisor $D_\pi:= C\times_{V^\sharp,i_\pi}k^0$.
In particular, $X$ is a fine log smooth log scheme over $V^\sharp$ such that $Y_\pi$ is of Cartier type over $k^0$.
We consider the log crystalline sites $(Y_\pi/W^0)_{\cris}$, $(X_1/V^\sharp)_{\cris}$,  
and $(X_1/\cS_{\PD})_{\cris}$.

\begin{definition}\label{def: crys cs}
The compactly supported crystalline Hyodo--Kato cohomology of $X$, as well as the compactly supported log crystalline cohomology  of $X$ over $V^\sharp$ and over $\cS_{\PD}$ is defined as
\begin{align*}
&R\Gamma_{\HK,c}^{\cris}(X)_\pi := R\Gamma_{\cris,c}(Y_\pi/W^0) =   R\Gamma_{\cris,\cM_{D_\pi}}(Y_\pi/W^0),&\\
& R\Gamma_{\cris,c}(X/V^\sharp) :=   R\Gamma_{\cris,c}(X_1/V^\sharp):=  R\Gamma_{\cris,\cM_C}(X_1/V^\sharp),&\\
&R\Gamma_{\cris,c}(X/\cS_{\PD})_\pi  := R\Gamma_{\cris,c}(X_1/\cS_{\PD})_\pi:=    R\Gamma_{\cris,\cM_C}(X_1/\cS_{\PD}).&
\end{align*}
As explained above, the first and the last two complexes are defined via the morphisms $j_\pi$,
which is the reason they depend on $\pi$. 
\end{definition}

The following statement is very useful when carrying over properties from the non-compactly supported to the compactly supported versions of log crystalline cohomology.

\begin{lemma}
Let $D_\pi^{(\bullet),\flat}$ be defined as above (\ref{equ : D^flat}), and $C^{(\bullet),\flat}$ similarly.
There are quasi-isomorphisms
\begin{eqnarray}\label{equ : iso cryst}
\nonumber \eta: R\Gamma_{\HK,c}^{\cris}(X)_\pi  & \xrightarrow{\sim} &    R\Gamma^{\cris}_{\HK}(C^{(\bullet),\flat})_\pi,\\
\eta: R\Gamma_{\cris,c}(X/V^\sharp)  & \xrightarrow{\sim} &     R\Gamma_{\cris}(C^{(\bullet),\flat}/V^\sharp) ,\\
\nonumber \eta: R\Gamma_{\cris,c}(X/\cS_{\PD})_\pi  & 
\xrightarrow{\sim} &   R\Gamma_{\cris}(C^{(\bullet),\flat}/\cS_{\PD})_\pi,
\end{eqnarray}
which induce spectral sequences
\begin{eqnarray}\label{equ : spec cryst}
\nonumber E_1^{i,j}=  H^{\cris,j}_{\HK}(C^{(i),\flat})_\pi & \Longrightarrow&  H^{i+j,\cris}_{\HK,c}(X)_\pi,\\
E_1^{i,j}=  H^j_{\cris}(C^{(i),\flat}/ V^\sharp) & \Longrightarrow&  H^{i+j}_{\cris,c}(X/V^\sharp),\\
\nonumber E_1^{i,j}=  H^j_{\cris}(C^{(i),\flat}/ \cS_{\PD})_\pi &
  \Longrightarrow &  H^{i+j}_{\cris,c}(X/\cS_{PD})_\pi.
\end{eqnarray}
\end{lemma}
\begin{proof}
This is proved in \cite[Thm.\,9.17]{YamashitaYasuda2014}  
(see also the explanation in the proof of 
\cite[Thm.\,10.2]{YamashitaYasuda2014}). 
Similar to Lemma \ref{lem: cris D}, it is the crystalline analogue of Proposition \ref{prop: D} and follows essentially from the same argument by replacing rigid techniques with crystalline ones.
\end{proof}

For the definition of the crystalline Hyodo--Kato map, or more precisely for its uniqueness, the Frobenius  on $R\Gamma_{\HK,c}^{\cris}(X)_\pi$ and on 
$R\Gamma_{\cris,c}(X/\cS_{\PD})_\pi$ plays a crucial role. 
We already discussed the former earlier in this section.

\begin{lemma}
There is an $\sO_{X_1/\cS_{\PD}}$-linear homomorphism
	\begin{align}\label{equ: frob S}
	 \Phi:\sigma^\ast_\cris\sO_{X_1/\cS_{\PD}}(\cM_C) \rightarrow \sO_{X_1/\cS_{\PD}}(\cM_C) 
	\end{align}
compatible with the canonical morphism 
$\sigma^\ast_\cris\sO_{X_1/\cS_{\PD}} \rightarrow \sO_{X_1/\cS_{\PD}}$. 
\end{lemma}
\begin{proof}
The same proof as for Lemma \ref{lem: frob} also works in this situation.
\end{proof}

\begin{definition}\label{def: frobenius crys S}
Let $X$ be a proper strictly semistable log scheme over $V^\sharp$. 
The Frobenius action  $\varphi$ on 
$R\Gamma_{\cris,c}(X/\cS_{\PD})_\pi$
is induced by the absolute Frobenius $F_{X_1}$ on $X_1$, the Frobenius $\sigma$ on  
$\cS_{\PD}$, and  
(\ref{equ: frob S}).
 More precisely, $\varphi$ is  given as  composition
 \begin{align*}
R\Gamma_{\cris}(X_1/\cS_{\PD},\sO_{X_1/\cS_{\PD}}(\cM_C)) &\xrightarrow{\sigma} R\Gamma_{\cris}(X_1^{(\sigma)}/\cS_{\PD},\sO_{X_1/\cS_{\PD}}(\cM_C)^{(\sigma)}) \xrightarrow{F_{X_1/V_1^\sharp}^\ast} \\
&\rightarrow  R\Gamma_{\cris}(X_1/\cS_{\PD},\sigma^\ast_\cris\sO_{X_1/\cS_{\PD}}(\cM_C)) \xrightarrow{\Phi}  R\Gamma_{\cris}(X_1/\cS_{\PD},\sO_{X_1/\cS_{\PD}}(\cM_C)).
\end{align*}
\end{definition}

\begin{lemma}\label{lemma: frob S}
Let $X$ be a proper strictly semistable log scheme over $V^\sharp$. 
Then the linearisation
\begin{align*}
\varphi:R_{\PD} {}_{\;\sigma} \otimes_{R_{\PD}}^L R\Gamma_{\cris,c}(X/\cS_{\PD})_\pi  \xrightarrow{\sim} R\Gamma_{\cris,c}(X/\cS_{\PD})_\pi 
\end{align*}
is a quasi-isomorphism.
\end{lemma}

\begin{proof}
This is well-known in the non-compactly supported case (see for example the proof of \cite[Lem.\,5.3]{HyodoKato1994}), 
and hence (\ref{equ : iso cryst}) implies that the first morphism in the statement is a quasi-isomorphism (compare the proof of \cite[Thm.\,10.2]{YamashitaYasuda2014}).
\end{proof}

\begin{proposition}\label{prop: cris sections}
Let $X$ be a proper strictly semistable log scheme over $V^\sharp$. 
Consider the morphisms
$$
R\Gamma^{\cris}_{\HK,c}(X)_{\pi,\bbQ} \xleftarrow{j_0^\ast} 
R\Gamma_{\cris,c}(X/\cS_{\PD})_{\pi,\bbQ} \xrightarrow{j^\ast_\pi}
R\Gamma_{\cris,c}(X/V^\sharp)_{\bbQ}
$$
induced by the horizontal morphisms in the diagram (\ref{diag: comm}). 
The left-pointing morphism $j_0^\ast$ has in the derived category
a unique functorial $F$-linear section
\begin{align*}
s_\pi: R\Gamma^{\cris}_{\HK,c}(X)_{\pi,\bbQ} \rightarrow R\Gamma_{\cris,c}(X/\cS_{\PD})_{\pi,\bbQ}
\end{align*}
which is compatible with Frobenius. 
\end{proposition}
\begin{proof}
To show the existence of the section, we follow \cite[Lem.\,5.2]{HyodoKato1994}.  
Let $Y_\pi:=X\times_{V^\sharp,j_\pi}k^0$, 
and for the horizontal divisor $C$ of $X$
set $D_\pi= C\times_{V^\sharp,j_\pi}k^0$.
In the following we will consider the compactly supported crystalline cohomology
\begin{align*}
R\Gamma_{\cris,c}(Y_\pi/ \breve{\cS}_{\PD}):= R\Gamma_{\cris,\cM_{D_\pi}}(Y_\pi/\breve{\cS}_{\PD}), 
\end{align*}
with $\breve{\cS}_{\PD}$ as defined in Remark \ref{rem: breve{S}}.
Take now $r\geqslant 0$, such that $\pi^{p^r}\in pV$. 
There is a commutative diagram of log schemes
$$
\xymatrix{
X_1 \ar[r] \ar[d]& Y_\pi\ar@{^{(}->}[r] \ar[d] \ar@{}[rd]|\bigbox&X_1\ar[d]\\
V_1^\sharp \ar[r] \ar[d] & k^0 \ar@{^{(}->}[r] \ar[d] & V^\sharp_1\ar[d] \\
\cS_{\PD} \ar[r]^g & \breve{\cS}_{\PD} \ar[r]&\cS_{\PD} 
}
$$
where the morphism $g$ is given by $s\mapsto s^{p^r}$ and $a\mapsto \sigma^r(a)$ for $a\in W$, 
the composites of the horizontal arrows are the $r$\textsuperscript{th} 
iteration of the Frobenius,
and the upper right square is Cartesian. 
Thus by \cite[Prop.\,2.23]{HyodoKato1994} and \cite[1.16 Thm.\,(i)]{Beilinson2013} one obtains a quasi-isomorphism
\begin{equation}\label{equ HK2}
R_{\PD}{}_{\;\sigma^r}\otimes_{R_{\PD}}^L R\Gamma_{\cris,c}(X_1/\cS_{\PD})_{\pi,\bbQ} \xrightarrow{\sim} R_{\PD}{}_{\;g}\otimes_{\breve{R}_{\PD}}^L R\Gamma_{\cris,c}(Y_\pi/ \breve{\cS}_{\PD})_{\bbQ}.
\end{equation}
Furthermore, by the Lemmas \ref{lemma: frob S} and \ref{lemma Frob W} there are quasi-isomorphisms
\begin{align}
\label{equ HK3}
R_{\PD}{}_{\;\sigma^r}\otimes^L_{R_{\PD}}  R\Gamma_{\cris,c}(X_1/\cS_{\PD})_{\pi,\bbQ} 
&\xrightarrow[\sim]{1\otimes\varphi^r} R\Gamma_{\cris,c}(X_1/\cS_{\PD})_{\pi,\bbQ},
\\
\label{equ HK4} 
R_{\PD} {}_{\;\sigma^r}\otimes^L_{W} R\Gamma_{\cris,c}(Y_\pi/ W^0)_{\bbQ}
&\xrightarrow[\sim]{1\otimes\varphi^r}
R_{\PD}\otimes^L_WR\Gamma_{\cris,c}(Y_\pi/ W^0)_{\bbQ}.
\end{align}
Finally, by Lemma \ref{lem: HK 4.23 generalisation} below, 
there is a quasi-isomorphism
\begin{align}\label{equ HK1}
R_{\PD}{}_{\;\sigma^r}\otimes^L_{W} R\Gamma_{\cris,c}(Y_\pi/ W^0)_{\bbQ}  
\xrightarrow{\sim} R_{\PD}{}_{\;g}\otimes^L_{\breve{R}_{\PD}} R\Gamma_{\cris,c}(Y_\pi/ \breve{\cS}_{\PD})_{\bbQ} 
\end{align}
Clearly all of the morphisms considered above are compatible with Frobenius. 
Putting all of this together, 
we obtain a zigzags of natural morphism
\begin{align*}
&R_{\PD}\otimes^L_W R\Gamma_{\cris,c}(Y_\pi/ W^0)_{\bbQ} 
\xleftarrow[\sim]{\text{(\ref{equ HK4})}}
R_{\PD}{}_{\;\sigma^r}\otimes^L_{W} R\Gamma_{\cris,c}(Y_\pi/ W^0)_{\bbQ}\rightarrow\\
&\xrightarrow[\sim]{\text{(\ref{equ HK1})}}R_{\PD} {}_{\;g}\otimes^L_{\breve{R}_{\PD}} R\Gamma_{\cris,c}(Y_\pi/ \breve{\cS}_{\PD})_{\bbQ} \xleftarrow[\sim]{\text{(\ref{equ HK2})}}
R_{\PD}{}_{\;\sigma^r}\otimes^L_{R_{\PD}} R\Gamma_{\cris,c}(X_1/\cS_{\PD})_{\pi,\bbQ}\rightarrow\\
&\xrightarrow[\sim]{\text{(\ref{equ HK3})}}R\Gamma_{\cris,c}(X_1/\cS_{\PD})_{\pi,\bbQ}.
\end{align*}
Thus we can define a functorial $F$-linear section of $j_0^\ast$ 
by
$$
R\Gamma^{\cris}_{\HK,c}(X)_{\pi,\bbQ} \xrightarrow{s_\pi}  
R_{\PD}\otimes^L_W R\Gamma^{\cris}_{\HK,c}(X)_{\pi,\bbQ} \cong
R\Gamma_{\cris,c}(X/\cS_{\PD})_{\pi,\bbQ}.
$$
By construction, it is compatible with Frobenius.

We now address the uniqueness of the section $s_\pi$. 
According to \cite[Lem.\,4.4.10]{Tsuji1999-a} the $E_2$-sheet of the third spectral sequence in (\ref{equ : spec cryst}) provides finitely generated free
$\bbQ\otimes R_{\PD}$-modules.  
Hence,
the $H^n_{\cris,c}(X/\cS_{\PD})_{\bbQ}$, $n\in\bbN$ are finitely generated free $\bbQ\otimes R_{\PD}$-modules. 
It follows from \cite[Lem.\,4.4.11]{Tsuji1999-a} 
that the section  $s_\pi$ are unique on the level of cohomology groups, which implies uniqueness in the derived category by the same argument as in the proof of \cite[Prop.-Def.\,6.8]{ErtlYamada}.
\end{proof}

\begin{lemma}\label{lem: HK 4.23 generalisation}
Let $Y$ be a proper strictly semistable log scheme over $k^0$, 
and consider the canonical morphism 
$j_0^\ast:R\Gamma_{\cris,c}(Y/\breve{\cS}_{\PD}) \rightarrow R\Gamma_{\cris,c}(Y/W^0)$ 
given by $s\mapsto 0$. 
Then there exists a unique quasi-isomorphism 
$$
\breve{R}_{\PD}\otimes^L_W R\Gamma_{\cris,c}(Y/W^0)_{\bbQ} \xrightarrow[\sim]{s_0}R\Gamma_{\cris,c}(Y/\breve{\cS}_{\PD})_{\bbQ}
$$
such that the composition  $j_0^\ast\circ s_0$ coincides with the morphism induced  by $\breve{R}_{\PD} \rightarrow W$, 
which compatible with the Frobenius.
\end{lemma}
\begin{proof}
We modify the arguments in the proof of \cite[Prop.\,4.13]{HyodoKato1994} to adapt them to our situation. 
We have seen in Lemma \ref{lemma Frob W} that $\varphi$ is a quasi-isomorphism 
on $R\Gamma_{\cris,c}(Y/W^0)_{\bbQ}$.
More precisely, on $R\Gamma_{\cris,c}(Y/W^0)$ there exists a morphism $\psi$ 
and $r\in\bbN$, such that $\varphi\psi=\psi\varphi=p^r$. 
Indeed, according to \cite[Prop.\,2.24]{HyodoKato1994} 
for each simmplicial index $i$ of $D^{(\bullet),\flat}$ as defined 
in (\ref{equ : D^flat}), the natural number  
$r_i=\max\{\mathrm{rank}_x\omega^1_{D^{(i),\flat}/k^0}\,|\,x\in D^{(i),\flat}\}$ 
satisfies $\psi\varphi=\varphi\psi=p^{r_i}$ on 
$R\Gamma_{\cris,c}(D^{(i),\flat}/W^0)_{\bbQ}$. 
But the $r_i$ are bounded by 
$r=\max\{\mathrm{rank}_x\omega^1_{Y/k^0}\,|\,x\in Y\}$. 

Let $I$ be the kernel of $\breve{R}_{\PD} \rightarrow W$.
It is generated by divided powers $\frac{s^j}{j!}$ of $s$, 
and satisfies $\sigma^c(I) \subset p^c!\cdot I$ for any $c\geqslant 1$. 
Indeed, we have
$$
\sigma^c\left(\frac{s^j}{j!}\right) = \frac{s^{p^cj}}{j!} 
= \frac{(p^cj)!}{j!}\cdot \frac{s^{p^cj}}{(p^cj)!}
$$
for any $j\geqslant 1$ and $\frac{(p^cj)!}{j!}\in p^c!\cdot \bbZ$. 
We denote by 
\begin{align*}
&\varphi: \Phi(I\otimes^L_{\breve{R}_{\PD}}R\Gamma_{\cris,c}(Y/\breve{\cS}_{\PD})) \rightarrow I\otimes^L_{\breve{R}_{\PD}}R\Gamma_{\cris,c}(Y/\breve{\cS}_{\PD})
\end{align*}
where $\Phi(-)$ is the exact functor of tensoring with $W{}_{\;\sigma}\otimes^L_W -$, 
the morphisms given by $\sigma$ on  $I$, 
and by $\varphi$ on  $R\Gamma_{\cris,c}(Y/\breve{\cS}_{\PD})$ (defined as in Definition \ref{def: frobenius crys S}). 
Similarly, for finite levels, that is modulo $p^m$, we have
\begin{align*}
&\varphi: \Phi_m(I_m\otimes^L_{\breve{R}_{\PD,m}}R\Gamma_{\cris,c}(Y/\breve{\cS}_{\PD,m})) \rightarrow I\otimes^L_{\breve{R}_{\PD,m}}R\Gamma_{\cris,c}(Y/\breve{\cS}_{\PD,m}), 
\end{align*}
with the exact functor $\Phi_m(-)=W_m{}_{\;\sigma}\otimes^L_{W_m} -$.
Working with finite coefficients, 
it is clear that the set of morphisms of the form
\begin{align*}
R\Gamma_{\cris,c}(Y/W^0_m ) \rightarrow \cK_m,
\end{align*}
where $\cK$ is either one of
\begin{align*}
I\otimes^L_{\breve{R}_{\PD}}R\Gamma_{\cris,c}(Y/\breve{\cS}_{\PD}) &&\text{or}&&
I\otimes^L_{\breve{R}_{\PD}}R\Gamma_{\cris,c}(Y/\breve{\cS}_{\PD})[1],
\end{align*} 
and $\cK_m$ its reduction modulo $p^m$
is annihilated by a high enough power of $p$ (depending on $m$). 

Take now  $c\in\bbN$ such that $p^{rc+1}| p^c!$ and let 
$h: R\Gamma_{\cris,c}(Y/W^0 ) \rightarrow \cK$ be a morphism compatible with Frobenius in the sense that
$h\varphi=\varphi\Phi(h)$. 
For a fixed $m$ consider the reduction of this morphism modulo $p^m$, that is
$h_m: R\Gamma_{\cris,c}(Y/W_m^0 ) \rightarrow \cK_m$.
Then we have
$$
p^{rc}h_m= h_m\varphi^c\psi^c= \varphi^c\Phi^c(h_m)\psi^c= 
(p^c)!\frac{\varphi^c}{(p^c)!}\Phi^c(h_m)\psi^c =
p^{rc} a \frac{\varphi^c}{(p^c)!}\Phi^c(h_m)\psi^c
$$
for some $a\in p\bbZ$, 
where we have used that $\varphi^c$ on 
$\Phi^c_m(\cK_m)$ 
is divisible by $p^c!$. 
Iterating this process, we obtain equalities
$$
p^{rc}h_m = p^{rc} a \frac{\varphi^c}{(p^c)!}\Phi^c(h_m)\psi^c
= p^{rc}a^2 \frac{\varphi^c}{(p^c)!} \Phi^c\left(\frac{\varphi^c}{(p^c)!}\right)\varphi^{2c}(h_m) \psi^{2c} = \cdots = p^{rc} a^m \tilde{h}_m=0
$$
producing at each step a new morphism $\tilde{h}_m:R\Gamma_{\cris,c}(Y/W^0_m ) \rightarrow \cK_m$ compatible with Frobenius, 
until the power of $a$ is high enough so that the morphism becomes trivial. 
This shows that  $h_m$ is annihilated by $p^{rc}$. 
Since $p^{rc}$ is independent of $m$,  it follows, that 
$h$ itself is annihilated by $p^{rc}$. 
Inverting $p$, we conclude that any morphism 
$h: R\Gamma_{\cris,c}(Y/W^0 )_{\bbQ} \rightarrow \cK_{\bbQ}$ which is compatible with the Frobenius
is necessarily trivial. 

Consider now the exact triangle in the derived category of $\varphi$-modules
$$
I\otimes^L_{\breve{\cS}_{\PD}}R\Gamma_{\cris,c}(Y/\breve{\cS}_{\PD})_{\bbQ} 
\rightarrow
R\Gamma_{\cris,c}(Y/\breve{\cS}_{\PD})_{\bbQ} 
\xrightarrow{j_0^\ast} 
R\Gamma_{\cris,c}(Y/W^0)_{\bbQ} \rightarrow+
$$
By what we have seen above, its boundary morphism is trivial, 
and hence it splits.
In particular, we obtain a section of the morphisms $j_0^\ast$
$$
\xymatrix{
R\Gamma_{\cris,c}(Y/W^0)_{\bbQ}  \ar@<-.5ex>[r]_{s_0} &
R\Gamma_{\cris,c}(Y/\breve{\cS}_{\PD})_{\bbQ} \ar@<-.5ex>[l]_{j_0^\ast}
}
$$
which is compatible with the Frobenius. 
To see that this section induces a quasi-isomorphism
after tensoring with $\breve{R}_{\PD}\otimes^L_W-$ respectively, 
we observe, that the construction explained above is compatible with 
the one in \cite[Prop.\,4.13]{HyodoKato1994} in the non-compactly supported case. 
Hence using the fact that for the base $\breve{\cS}_{\PD}$ we also have a quasi-isomorphism of the form 
(\ref{equ : iso cryst}) because $\breve{\cS}_{\PD}$ constitutes just the unramified case of $\cS_{\PD}$, 
we obtain a commutative diagram 
$$
\xymatrix{
\breve{R}_{\PD}\otimes^L_W R\Gamma_{\cris,c}(Y/W^0)_{\bbQ} \ar[r]^-{s_0} \ar[d]^-\eta_-{\sim} & 
R\Gamma_{\cris,c}(Y/\breve{\cS}_{\PD})_{\bbQ} \ar[d]^-\eta_-{\sim}\\
\breve{R}_{\PD}\otimes^L_W R\Gamma_{\cris}(D^{(\bullet),\flat}/W^0)_{\bbQ} \ar[r]^-{s_0}_-{\sim} &
R\Gamma_{\cris}(D^{(\bullet),\flat}/\breve{\cS}_{\PD})_{\bbQ}
}
$$
where the bottom morphism is a quasi-isomorphism, because for every simplicial index $i$, 
$s_0:\breve{R}_{\PD}\otimes^L_W R\Gamma_{\cris}(D^{(i),\flat}/W^0)_{\bbQ}
\xrightarrow{\sim} R\Gamma_{\cris}(D^{(i),\flat}/\breve{\cS}'_{\PD})_{\bbQ}$ 
is a quasi-isomorphism by \cite[Prop.\,4.13]{HyodoKato1994}. 
Therefore the upper morphism in the diagram is a quasi-isomorphism as desired.
\end{proof}

\begin{proposition-definition}\label{prop-def: cris HK}
Let $X$ be a proper strictly semistable log scheme over $V^\sharp$ with horizontal divisor $C$ and $\pi$ a uniformiser. 
We set 
$$
\Psi_\pi^\cris: =  j_\pi^\ast\circ s_\pi \colon  R\Gamma_{\HK,c}^\cris(X)_{\pi,\bbQ} \rightarrow   R\Gamma_{\cris,c}( X/V^\sharp)_{\bbQ}.
$$
It induces a $K$-linear functorial quasi-isomorphism $\Psi_{\pi,K}^\cris :=   \Psi_{\pi}^\cris \otimes 1\colon   R\Gamma_{\HK,c}^\cris(X)_\pi \otimes_{W}K  \xrightarrow {\sim}  R\Gamma_{\cris,c}( X/V^\sharp)_{\bbQ}$.
\end{proposition-definition}
\begin{proof}
As the constructions in Proposition \ref{prop: cris sections} followed 
the methods of \cite[\S5]{HyodoKato1994} in the non-compactily supported case, 
the obtained section $s_\pi$ and hence $\Psi_\pi^\cris$ is compatible with 
the quasi-isomorphisms (\ref{equ : iso cryst}). 
Therefore, we obtain a commutative diagram
$$
\xymatrix{
R\Gamma_{\HK,c}^{\cris}(X)_{\pi,\bbQ} \ar[d]_{\sim}^\eta \ar@<-.5ex>[r]_{s_\pi}
& R\Gamma_{\cris,c}(X/\cS_{\PD})_{\pi,\bbQ}
\ar[d]_{\sim}^\eta \ar[r]^{j_{\pi}^\ast} \ar@<-.5ex>[l]_{j_0^\ast}
&R\Gamma_{\cris,c}(X/V^\sharp)_{\bbQ}\ar[d]_{\sim}^\eta\\
R\Gamma_{\HK}^\cris(C^{(\bullet),\flat})_{\pi,\bbQ}\ar@<-.5ex>[r]_{s_\pi}
&R\Gamma_{\cris}(C^{(\bullet),\flat}/\cS_{\PD})_{\pi,\bbQ}
\ar[r]^{j_{\pi}^\ast} \ar@<-.5ex>[l]_{j_0^\ast} 
&R\Gamma_{\cris}(C^{(\bullet),\flat}/V^\sharp)_{\bbQ}
}
$$

For every $i\geqslant 0$, 
the composition 
$\Psi_\pi^{\cris} =j_\pi^\ast\circ s_\pi: R\Gamma_{\HK}^\cris(C^{(i),\flat})_{\pi,\bbQ} \rightarrow R\Gamma_{\cris}(C^{(i),\flat}/V^\sharp)_\bbQ$
induces a $K$-linear functorial quasi-isomorphism 
$\Psi_{\pi,K}^\cris :=   \Psi_{\pi}^\cris \otimes 1\colon   R\Gamma_{\HK}^\cris(C^{(i),\flat})_\pi \otimes^L_{W}K  \xrightarrow {\sim}  R\Gamma_{\cris}(C^{(i),\flat}/V^\sharp)_{\bbQ}$
according to \cite[(5.4)]{HyodoKato1994}.
This implies the statement via the above commutative diagram.
\end{proof}

%
\section{Comparison of compactly supported log crystalline and log rigid cohomology}\label{sec: comparison}
%

Continuing the conventions of the previous section, a log structure on a (formal) log scheme is a sheaf on the \'{e}tale topology.
We say that a fine (formal) log scheme is \textit{of Zariski type} if Zariski locally it admits a chart.
Note that by \cite[Cor.\,1.1.11]{Shiho2002}, giving a fine log scheme (with respect to the \'{e}tale topology) of Zariski type is equiavalent to giving a fine log scheme with respect to the Zariski topology.
Through this identification, we apply the constructions and results in Section \ref{sec: isoc}-\ref{sec: c rig} to a fine (formal) log scheme of Zariski type.

The goal of this section is to compare the rigid construction of cohomology with compact support with the crystalline construction.
The comparison passes through log convergent cohomology. 
Thus we first study the relation between log overconvergent isocrystals in \cite{Yamada2020} and Shiho's log convergent isocrystals in \cite{Shiho2000}, \cite{Shiho2008}.
To any weak formal log scheme $\cT$, by taking the completion with respect to the ideal of definition we may associate a formal log scheme $\widehat{\cT}$. 

Let $(T,\cT,\iota)$ be a widening.
For a fine log scheme $Y$ over $T$ of Zariski type, we define a site
	\[\Conv(Y/\wh\cT)=\Conv(Y/(T,\wh\cT,\wh\iota))\]
and a category
	\[\Isoc(Y/\wh\cT)=\Isoc(Y/(T,\wh\cT,\wh\iota))\]
by the same procedure as $\OC(Y/\cT)$ and $\Isoc^\dagger(Y/\cT)$ but using formal log schemes and rigid spaces instead of weak formal log schemes and dagger spaces.

Our first task is to understand in the case that $\wh\cT$ is $p$-adic the relation between $\Isoc(Y/\widehat{\cT})$ and the category of log convergent isocrystals defined by Shiho.
When $\wh\cT$ is $p$-adic, we denote by $(Y/\wh\cT)_{\conv,\zar}$ the log convergent site with respect to the Zariski topology in \cite[Def.\,2.4]{Shiho2008}.
To any sheaf $E$ on $(Y/\wh\cT)_{\conv,\zar}$ and any object $(Z,\cZ,i,h,\theta)\in (Y/\wh\cT)_{\conv,\zar}$, we may associate a sheaf $E_\cZ$ on $\cZ$.
Recall that $E$ is said to be a \textit{locally free isocrystal} \cite[Def.\,2.6]{Shiho2008} if each $E_\cZ$ is a locally free isocoherent sheaf \cite[Def.\,1.2]{Shiho2008} and for any morphism $f\colon (Z',\cZ',i',h',\theta')\rightarrow (Z,\cZ,i,h,\theta)$ the induced morphism $f^*E_\cZ\rightarrow E_{\cZ'}$ is an isomorphism.
We denote by $I_{\conv,\zar}^{\mathrm{lf}}(Y/\wh\cT)$ the category of locally free isocrystals.

The site $(Y/\wh\cT)_{\conv,\zar}$ consists of quintuples $(Z,\cZ,i,h,\theta)$ where $\cZ$ are $p$-adic and flat over $W$, and its topology is given by the Zariski topology on $\cZ$.
In contrast, the site $\Conv(Y/\wh\cT)$ consists of quintuples $(Z,\cZ,i,h,\theta)$ where $\cZ$ is not necessarily $p$-adic, and its topology is induced from the topology of the rigid spaces $\cZ_\bbQ$.
Clearly there exists a canonical continuous functor
	\begin{equation}\label{eq: continuous functor}
	\chi\colon(Y/\wh\cT)_{\conv,\zar}\rightarrow\Conv(Y/\wh\cT).
	\end{equation}

\begin{proposition}\label{prop: isocrystals}
	The functor \eqref{eq: continuous functor} induces an equivalence of categories
	\begin{equation}\label{eq: conv isoc}
	\chi_*\colon \Isoc(Y/\wh\cT)\rightarrow I^{\mathrm{lf}}_{\conv,\zar}(Y/\wh\cT).
	\end{equation}
\end{proposition}

\begin{proof}
	For an object $\sE\in\Isoc(Y/\wh\cT)$, we define a sheaf $\chi_*\sE$ on $(Y/\wh\cT)_{\conv,\zar}$ by
	\[\Gamma((Z,\cZ,i,h,\theta),\chi_*\sE):=\Gamma((Z,\cZ,i,h,\theta),\sE).\]
	For $(Z,\cZ,i,h,\theta)\in (Y/\wh\cT)_{\conv,\zar}$, let $(\chi_*\sE)_\cZ$ be the associated sheaf on $\cZ$.
	If $\cZ$ is affine, $(\chi_*\sE)_\cZ$ corresponds to a finite projective module $M:=\Gamma(\chi(Z,\cZ,i,h,\theta),\sE)$ over $A:=\Gamma(\cZ_\bbQ,\cO_{\cZ_\bbQ})$ in the sense that we have $\Gamma(\cU,(\chi_*\sE)_\cZ)\cong M\otimes_A\Gamma(\cU_\bbQ,\cO_{\cU_\bbQ})$ for any open subset $\cU\subset\cZ$.
	Thus $(\chi_*\sE)_\cZ$ is isocoherent for general $(Z,\cZ,i,h,\theta)$ by \cite[(1.2.1) and (1.2.5)]{Ogus1984}, and we obtain the functor \eqref{eq: conv isoc}.
	
	To see that $\chi_*$ is an equivalence, 
	we construct a quasi-inverse 
	by associating an object $\sE\in\Isoc(Y/\wh\cT)$ to an object $E\in I^{\mathrm{lf}}_{\conv,\zar}(Y/\wh\cT)$ as follows:
	
	For an object $(Z,\cZ,i,h,\theta)\in\Conv(Y/\wh\cT)$, let $\{(Z[n],\cZ[n],i[n],h[n],\theta[n])\}_n$ be the family of universal enlargements defined similarly to Definition \ref{def: univ enl}.
	Then $\cZ[n]$ are $p$-adic and flat over $W$, and we have $\cZ_\bbQ=\bigcup_n\cZ[n]_{\bbQ}$.
	We define a presheaf $\sE$ on $\Conv(Y/\wh\cT)$ by
		\[\Gamma((Z,\cZ,i,h,\theta),\sE):=\varprojlim_n\Gamma((Z[n],\cZ[n],i[n],h[n],\theta[n]),E).\]
	We prove that $\sE$ is in fact a sheaf.
	Let $\{f_\lambda\colon(Z_\lambda,\cZ_\lambda,i_\lambda,h_\lambda,\theta_\lambda)\rightarrow(Z,\cZ,i,h,\theta)\}_\lambda$ be a covering family in $\Conv(Y/\wh\cT)$.
	Then the induced morphisms between universal enlargements give covering families
		\[\{f_{\lambda}[n]\colon (Z_{\lambda}[n],\cZ_{\lambda}[n],i_{\lambda}[n],h_{\lambda}[n],\theta_{\lambda}[n])\rightarrow(Z[n],\cZ[n],i[n],h[n],\theta[n])\}_\lambda.\]
	It suffices to check the sheaf condition of $\sE$ for the covering family $\{f_{\lambda}[n]\}_\lambda$ for each $n$.
	Fix $n\geqslant 1$.
	There exists an admissible blow-up $\wt\cZ\rightarrow\cZ[n]$ and open subsets $\cU_\lambda\subset\wt\cZ$ for all $\lambda$ such that each $f_{\lambda}[n]$ factors through a morphism $\cZ_{\lambda}[n]\rightarrow\cU_\lambda$ which induces an isomorphism $\cZ_\lambda[n]_\bbQ\xrightarrow{\cong}\cU_{\lambda,\bbQ}$.
	Note that the $\cU_\lambda$'s cover $\wt\cZ$ since $\wt\cZ$ is $p$-adic and flat over $W$. 
	Now we obtain a commutative diagram
		\begin{equation}\label{eq: blowup}
		\xymatrix{
		(Z_{\lambda}[n],\cZ_{\lambda}[n],i_{\lambda}[n],h_{\lambda}[n],\theta_{\lambda}[n])\ar[d]\ar[r]^-{f_{\lambda}[n]}&
		(Z[n],\cZ[n],i[n],h[n],\theta[n])\\
		(U_\lambda,\cU_\lambda,\wt i_\lambda,\wt h_\lambda,\wt\theta_\lambda)\ar[r]&
		(\wt Z,\wt\cZ,\wt i,\wt h,\wt\theta)\ar[u]
		}\end{equation}
	by setting
	$\wt Z:=\wt\cZ\times_{\cZ[n]}Z[n]$ and $U_\lambda:=\cU_\lambda\times_{\wt\cZ}\wt Z$ with obvious morphisms $\wt i,\wt h,\wt\theta,\wt i_\lambda,\wt h_\lambda,\wt\theta_\lambda$.
	Since the vertical morphisms in \eqref{eq: blowup} induce isomorphisms
		\begin{align*}
		\Gamma((U_\lambda,\cU_\lambda,\wt i_\lambda,\wt h_\lambda,\wt\theta_\lambda),E)&\xrightarrow{\cong}\Gamma((Z_{\lambda}[n],\cZ_{\lambda}[n],i_{\lambda}[n],h_{\lambda}[n],\theta_{\lambda}[n]),E)\\
		&=\Gamma((Z_{\lambda}[n],\cZ_{\lambda}[n],i_{\lambda}[n],h_{\lambda}[n],\theta_{\lambda}[n]),\sE),\end{align*}
		\begin{align*}
		\Gamma((Z[n],\cZ[n],i[n],h[n],\theta[n]),E)&\xrightarrow{\cong}\Gamma((\wt Z,\wt\cZ,\wt i,\wt h,\wt\theta),E)\\
		&=\Gamma((\wt Z,\wt\cZ,\wt i,\wt h,\wt\theta),\sE),
		\end{align*}
	the sheaf condition of $\sE$ for $\{f_{\lambda}[n]\}_\lambda$ follows from the sheaf condition of $E$ for $\{(U_\lambda,\cU_\lambda,\wt i_\lambda,\wt h_\lambda,\wt\theta_\lambda)\rightarrow(\wt Z,\wt\cZ,\wt i,\wt h,\wt\theta)\}_\lambda$.
	
	Moreover for any morphism $g\colon (Z',\cZ',i',h',\theta')\rightarrow(Z,\cZ,i,h,\theta)$, by taking universal enlargements we obtain a family of morphisms
		\[g[n]\colon(Z'[n],\cZ'[n],i'[n],h'[n],\theta'[n])\rightarrow(Z[n],\cZ[n],i[n],h[n],\theta[n]).\]
	The isomorphisms $E_{\cZ'[n]}\cong g[n]^*E_{\cZ[n]}$ induce an isomorphism $\sE_{\cZ'}\cong g^*\sE_\cZ$.
	Thus we see that $\sE$ is an object of $\Isoc(Y/\wh\cT)$.
\end{proof}

Completion induces a canonical functor
	\begin{equation}\label{eq: cpln}
	\Isoc^\dagger(Y/\cT)\rightarrow\Isoc(Y/\wh\cT).
	\end{equation}
For $\sE\in\Isoc^\dagger(Y/\cT)$, we denote by $\wh\sE\in \Isoc(Y/\wh\cT)$ the image of $\sE$ under this functor.
By applying the same construction as in the case of $R\Gamma_\rig(Y/\cT,\sE)$ for $\sE\in\Isoc^\dagger(Y/\cT)$ to $\wh\sE\in\Isoc(Y/\wh\cT)$, we may define a complex of $\Gamma(\wh\cT,\cO_{\wh\cT})$-modules, which we denote by $R\Gamma_\conv(Y/\wh\cT,\sE)$ and call the \textit{log convergent cohomology}.
The log convergent cohomology was originally defined by Shiho (\cite[Def.\,2.2.12]{Shiho2002}, \cite[Def.\,4.1]{Shiho2008}).
Moreover, we may define $R\Gamma_{\conv,\cM}(Y/\wh\cT,\wh\sE)$ for a monogenic log substructure $\cM\subset\cN_Y$ and $R\Gamma_{\conv,c}(Y/\wh\cT,\wh\sE)$ if $Y$ is proper strictly semistable, in the same way as Definition \ref{def: coh for monogenic} and Definition \ref{def: coh with compact supp}. 

By definition there exists a canonical morphism
	\begin{equation}\label{eq: rig an}
	R\Gamma_\rig(Y/\cT,\sE)\rightarrow R\Gamma_\conv(Y/\wh\cT,\wh\sE).
	\end{equation}
When $Y$ is a proper strictly semistable log scheme over $k^0$, then we have a canonical morphism
	\begin{equation}\label{eq: rig an c}
	R\Gamma_{\rig,c}(Y/\cT,\sE)\rightarrow R\Gamma_{\conv,c}(Y/\wh\cT,\wh\sE).
	\end{equation}

\begin{proposition}\label{prop : rig con iso}
	Let $Y$ be a proper strictly semistable log scheme over $k^0$.
	Let $(T,\cT,\iota)$ be $(k^0,W^0,i_0)$ or $(k^0,V^\sharp,i_\pi)$, and $\sE\in \Isoc^\dagger(Y/\cT)^\unip$.
	Then the morphisms 
	\eqref{eq: rig an} and \eqref{eq: rig an c} defined above are quasi-isomorphisms.
\end{proposition}

\begin{proof}
	Since $\sE$ is unipotent, we may assume that $\sE=\sO_{Y/\cT}$.
	The statement for 
	\eqref{eq: rig an} has been proved in \cite[Thm.\,5.3]{GrosseKlonne2005} and \cite[Lem.\,5.3]{ErtlYamada1}.
	The statement for 
	\eqref{eq: rig an c} is reduced to that for 
	\eqref{eq: rig an} via the quasi-isomorphisms in Proposition \ref{prop: D} and their convergent analogues.
\end{proof}

\begin{proposition}\label{prop: an vs cris}
	Let $T$ be an affine fine log scheme over $k$, $\cT$ an affine formal log scheme (not necessarily $p$-adic) over $W$, and $T\hookrightarrow\cT$ a homeomorphic exact closed immersion with ideal $\cI$. 
	Denote by $\cT_{\mathrm{PD}}$ the $\cI$-adic completion of the PD-envelope of $T\hookrightarrow\cT$ over $\Spf W$ with the canonical PD-structure.
	Let $Y$ be a fine log scheme over $T$ of Zariski type, and let $\sE\in\Isoc(Y/\cT)$.

	Suppose that there exists a family of coherent sheaves $E_\cZ$ on $\cZ$ for all $(Z,\cZ,i,h,\theta)\in\Conv(Y/\cT)$ together with isomorphisms $\tau_f\colon f^*E_\cZ\cong E_{\cZ'}$ for all $f\colon(Z',\cZ',i',h',\theta')\rightarrow(Z,\cZ,i,h,\theta)$ in $\Conv(Y/\cT)$ satisfying $\tau_{f\circ g}=\tau_g\circ g^*\tau_f$ which induces $\sE$ in the following sense:
	\begin{enumerate}
	\item\label{item: an vs crys} For any object $(Z,\cZ,i,h,\theta)\in\Conv(Y/\cT)$ where $\cZ$ is $p$-adic, there exists an isomorphism
		\[\Gamma(\cZ_\bbQ,\sE_\cZ)\cong\Gamma(\cZ,E_{\cZ})\otimes_\bbZ\bbQ.\]
	\item For any morphism $f\colon(Z',\cZ',i',h',\theta')\rightarrow(Z,\cZ,i,h,\theta)$ where $\cZ$ and $\cZ'$ are $p$-adic, the isomorphism $f^*\sE_\cZ\cong \sE_{\cZ'}$ is induced by $\tau_f$ and the isomorphisms in \ref{item: an vs crys}.
	\end{enumerate}
	Then $\sE$ naturally induces an object $\sE_\cris\in \mathrm{Crys}(Y/\cT_{\PD})_\bbQ$, and there exists a canonical morphism
	\begin{equation}\label{eq: morp}
	R\Gamma_\conv(Y/\cT,\sE)\rightarrow R\Gamma_\cris(Y/\cT_{\PD},\sE_\cris).
	\end{equation}
\end{proposition}

\begin{proof}
	We first consider the case that there exists an adic log smooth morphism $\cZ\rightarrow\cT$ admitting a chart and a closed immersion $Y\hookrightarrow\cZ$.
	For $i\geqslant 0$, let $\cZ(i)$ be the $(i+1)$-fold fibre product of $\cZ$ over $\cT$.
	Let $D(i)$ be the $\cI$-adic completion of the log PD-envelope of $Y\hookrightarrow\cZ(i)$ over $\Spf W$ with the canonical PD-structure.
	(Note that this is the same as the log PD-envelope of $Y\hookrightarrow\cZ(i)\times_\cT\cT_{\PD}$ over $\cT_{\PD}$.)
	Let $Y\hookrightarrow\cZ(i)^{\mathrm{ex}}$ be the exactification of $Y\hookrightarrow\cZ(i)$ (cf.\,\cite[Prop.-Def.\,2.10]{Shiho2008}).
	
	To relate $D(i)$ and $\cZ(i)^{\mathrm{ex}}$ for any fixed $i\geqslant 0$, we choose a factorization $Y\hookrightarrow\cX_i\rightarrow \cZ(i)$ into an exact closed immersion and an adic log \'{e}tale morphism.
	Then $D(i)$ is the (non-logarithmic) PD-envelope of $Y\hookrightarrow\cX_i$ over $W$ equipped with the pull-back log structure, hence given by adding to $\cO_{\cX_i}$ divided powers of local generators of the ideal of $Y\hookrightarrow\cX_i$. 
	On the other hand, $\cZ(i)^{\mathrm{ex}}$ is the completion of $\cX_i$ along $Y$. 
	Thus there exists a natural morphism $g(i)\colon D(i)\rightarrow\cZ(i)^{\mathrm{ex}}$, which is independent of the choice of the factorization $\cX_i$ and clearly compatible with the different natural projections 
	$p_{j_1\ldots j_{i+1}}: D(i+1)\rightarrow D(i)$ and 
	$q_{j_1\ldots j_{i+1}}: \cZ(i+1)^{\mathrm{ex}}\rightarrow \cZ(i)^{\mathrm{ex}}$, 
	$1\leqslant j_1<\cdots<j_{i+1}\leqslant i+2$.
     
	Using the datum of the family of isomorphisms on coherent sheaves in the hypothesis one obtains for $E:=g(0)^*E_{\cZ(0)^{\mathrm{ex}}}$ the isomorphism 
    \[\epsilon\colon p_2^*E=g(1)^*q_2^*E_{\cZ(0)^{\mathrm{ex}}}\xrightarrow[\tau_{q_2}]{\cong} g(1)^*E_{\cZ(1)^{\mathrm{ex}}}\xleftarrow[\tau_{q_1}^{-1}]{\cong} g(1)^*q_1^*E_{\cZ(0)^{\mathrm{ex}}}=p_1^*E.\]
    As it satisfies and $p_{12}^*(\epsilon)\circ p_{23}^*(\epsilon)=p_{13}^*(\epsilon)$ and $\Delta^*(\epsilon)=\id$ with the diagonal morphism $\Delta\colon D(0)\rightarrow D(1)$,
     it gives an HPD-stratification on $D(0)$, that is a projective system $\{E_n\}_n$ of coherent sheaves on $D(0)_n$ (here the subscript means modulo $\cI^n$) together with an isomorphism $\epsilon$ as above (compare \cite[Def.\,1.21]{Shiho2008}).
     Moreover, by \cite[\S 6]{Kato1989} it induces a projective system of log crystals on $Y$ over $\cT_{\PD,n}$, and hence equivalently a log crystal $E_\cris$ on $Y$ over $\cT_{\PD}$.

	The above construction of $E_\cris$ depends (at least integrally) upon the family $\{E_\cZ,\tau_f\}$. 
	To see that rationally it is independent of such a choice, note that, for each $i\geqslant 0$, $g(i)$ factors through $\cZ(i)^{\mathrm{ex}}[n]$ for large enough $n$, where $[n]$ denotes the convergent analogue of the construction in Definition \ref{def: univ enl}.
	If we take another family $\{E'_\cZ,\tau'_f\}$ satisfying the same conditions, then we have a canonical isomorphism $E'_{\cZ(i)^{\mathrm{ex}}[n]}\otimes\bbQ\cong E_{\cZ(i)^{\mathrm{ex}}[n]}\otimes\bbQ$ as isocoherent sheaves, since $\cZ(i)^{\mathrm{ex}}[n]$ is $p$-adic.
	As a consequence, $\sE_\cris:=E_\cris\otimes\bbQ$ is independent of the choice of the family $\{E_\cZ,\tau_f\}$ as an object of $\mathrm{Crys}(Y/\cT_\PD)_\bbQ$.
	
	Finally the morphism \eqref{eq: morp} is given as the pull-back via $g(0)$ between the de~Rham complexes of log connections corresponding to $\sE$ and $\sE_\cris$.
	
	In the general case, take an open covering $\{Z_\lambda\}_{\lambda\in\Lambda}$ of $Y$ with closed immersions $Z_\lambda\hookrightarrow\cZ_\lambda$ into formal log schemes which are adic and log smooth over $\cT$.
	For $m\geqslant 0$ let $Z^{(m)}$ and $\cZ^{(m)}$ be the $(m+1)$-fold fibre product of $\coprod_{\lambda\in\Lambda}Z_\lambda$ over $Y$ and $\coprod_{\lambda\in\Lambda}\cZ_\lambda$ over $\cT$, respectively.
	Applying the arguments as above to $Z^{(m)}\hookrightarrow\cZ^{(m)}$, we obtain a crystal $\sE_{\cris}^{(m)}$ on $Z^{(m)}$ over $\cT_\PD$ for each $m$.
	By construction there exists a canonical isomorphism $\mathrm{pr}_2^*\sE_\cris^{(0)}\cong \sE_\cris^{(1)}\cong \mathrm{pr}_1^*\sE_\cris^{(0)}$ where $\mathrm{pr}_i\colon Z^{(1)}\rightarrow Z^{(0)}$ is the natural projection, and it satisfies the cocycle condition.
	This descent datum defines an object $\sE_\cris\in\mathrm{Crys}(Y/\cT_\PD)_\bbQ$.
	
	The morphism  \eqref{eq: morp} in the general case is immediately induced from the above case by the \v{C}ech construction.
\end{proof}	
	
This construction was given in \cite[Prop.\,2.35, Thm.\,2.36]{Shiho2008} if $\cT$ is $p$-adic, under the identification of Proposition \ref{prop: isocrystals}.
In that case the morphism \eqref{eq: morp} is a quasi-isomorphism if $Y$ is log smooth over $T$.
However the case when $\cT$ is not $p$-adic is not covered in \cite{Shiho2008}.
More precisely, the authors do not know whether 
the equivalence between isocrystals and HPD-isostratifications 
\cite[Prop.\,1.22]{Shiho2008} which is crucially used in the proof of \cite[Prop.\,2.35]{Shiho2008} can be extended to the case that the base formal log scheme is not $p$-adic.
Composing with \eqref{eq: rig an} and \eqref{eq: rig an c} for the homeomorphic exact closed immersions 
$k^0\hookrightarrow W^0$, $V_1^{\sharp}\hookrightarrow V^\sharp$, and $V_1^\sharp\hookrightarrow \widehat{\cS}$ we obtain the following corollary.

\begin{corollary}\label{cor: rig cris}
	 If $Y$ is a proper strictly semistable log scheme over $k^0$, then there exist canonical quasi-isomorphisms
		\begin{align*}
		&R\Gamma_\rig(Y/W^0)\xrightarrow{\cong}R\Gamma_\cris(Y/W^0)_\bbQ,\\
		&R\Gamma_{\rig,c}(Y/W^0)\xrightarrow{\cong}R\Gamma_{\cris,c}(Y/W^0)_\bbQ.
		\end{align*}
	
	 If $X$ is a proper strictly semistable log scheme over $V^\sharp$, then there exist canonical quasi-isomorphisms
		\begin{align*}
		&R\Gamma_\rig(X_1/V^\sharp)\xrightarrow{\cong}R\Gamma_\cris(X_1/V^\sharp)_\bbQ,\\
		&R\Gamma_{\rig,c}(X_1/V^\sharp)\xrightarrow{\cong}R\Gamma_{\cris,c}(X_1/V^\sharp)_\bbQ,
		\end{align*}
	and canonical morphisms
		\begin{align*}
		&R\Gamma_\rig(X_1/\cS)\rightarrow R\Gamma_\cris(X_1/\cS_{\PD})_\bbQ,\\
		&R\Gamma_{\rig,c}(X_1/\cS)\rightarrow R\Gamma_{\cris,c}(X_1/\cS_{\PD})_\bbQ,
		\end{align*}
	where $X_1$ is $X$ modulo $p$.
	
	All of the above comparison morphisms are compatible with cup products. 
\end{corollary}

\begin{remark}
Let $Y$ be a proper strictly semistable log scheme over $k^0$ with horizontal divisor $D$. 
The second quasi-isomorphism in Corollary \ref{cor: rig cris} is compatible with the fourth quasi-isomorphism of Proposition \ref{prop: D}, its convergent analogue, and its crystalline analogue (\ref{equ : iso cris}). 
In particular, we obtain a commutative diagram
	$$
	\xymatrix{R\Gamma_{\rig,c}(Y/W^0) \ar[d]^\sim \ar[r] & R\Gamma_{\conv,c}(Y/W^0) \ar[d]^\sim\ar[r] & R\Gamma_{\cris,c}(Y/W^0)_{\bbQ} \ar[d]^\sim\\
	 R\Gamma_{\rig}(D^{(\bullet),\flat/W^0}) \ar[r]^\sim & R\Gamma_{\conv}(D^{(\bullet),\flat/W^0}) \ar[r]^\sim & R\Gamma_{\cris}(D^{(\bullet),\flat/W^0})_{\bbQ} }
	 $$
\end{remark}

\begin{remark}
The statement of Corollary \ref{cor: rig cris} easily extends to log overconvergent isocrystals $\sE\in \Isoc^\dagger(Y/W^0)$, $\sE\in\Isoc^\dagger(X_1/V^\sharp)$ and $\sE\in\Isoc^\dagger(X_1/\cS)$ respectively, if in the third case we assume the existence of a family $\{E_\cZ\}$ satisfying the conditions of Proposition \ref{prop: an vs cris}. 
\end{remark}

Let $X$ be a proper strictly semistable log scheme over $V^\sharp$ with horizontal divisor $C$.
For a choice of uniformiser $\pi$ of $V$, let $Y_\pi:=X\times_{V^\sharp,j_\pi}k^0$.
We will compare
$\Psi_\pi^\cris$ and $\Psi_{\pi,\log_\pi}$, where $\log_\pi$ is the branch of the $p$-adic logarithm such that $\log_\pi(\pi)= 0$.

\begin{proposition}\label{prop: HK cris -rig}
Let $X$ be a proper  strictly semistable log scheme over $V^\sharp$.
For a  uniformiser $\pi$ of $V$ set $Y_\pi:= X\times_{V^\sharp,i_\pi}k^0$.
There is a commutative diagram
	$$
	\xymatrix{
	R\Gamma_{\HK,c}^{\rig}(Y_\pi) \ar[d]^-\sim \ar[r]^-{\Psi_{\pi,\log_\pi}} & R\Gamma_{\rig,c}(Y_\pi/V^\sharp)_\pi \ar[d]^-\sim\\
	R\Gamma_{\HK,c}^{\cris}(X)_{\pi,\bbQ} \ar[r]^-{\Psi_\pi^{\cris}} & R\Gamma_{\cris,c}(X/V^\sharp)_{\bbQ}}
	$$
\end{proposition}

\begin{proof}
Recall that by definition $R\Gamma_{\HK,c}^{\cris}(X)_{\pi}= R\Gamma_{\cris,c}(Y_\pi/W^0)$ and consider the diagram
	$$
	\xymatrix{
	R\Gamma_{\HK,c}^{\rig}(Y_\pi) \ar[d]^\sim \ar[r]^{\psi} & R\Gamma_{\rig,c}(Y_\pi/\cS) \ar[r]^{j^\ast_\pi}\ar[ld]_-{j_0^{\ast}}& R\Gamma_{\rig,c}(Y_\pi/V^\sharp)_\pi  \\
	R\Gamma_{\rig,c}(Y_\pi/W^0)\ar[d]^\sim &R\Gamma_{\rig,c}(X_1/\cS)_\pi \ar[u]_-\sim\ar[d] \ar[r]^{j^\ast_\pi} \ar[l]_{j_0^\ast} & R\Gamma_{\rig,c}(X_1/V^\sharp)\ar[u]_-\sim\ar[d]^\sim \\
	R\Gamma_{\cris,c}(Y_\pi/W^0)_{\bbQ} \ar@<-0.5mm>[r]_{s_\pi}& R\Gamma_{\cris,c}(X_1/\cS_{\PD})_{\pi,\bbQ} \ar[r]^-{j^\ast_\pi} \ar@<-0.5mm>[l]_{j_0^\ast} & R\Gamma_{\cris,c}(X_1/V^\sharp)_\bbQ
	}$$
Note that, since $Y_\pi\hookrightarrow X_1$ is a homeomorphic exact closed immersion, their log rigid cohomologies coincide by definition.
The vertical morphisms from rigid to crystalline pass through log convergent cohomology as explained in Corollary \ref{cor: rig cris}.
It follows immediately from the definitions that all triangles and squares in the above diagram commute.

Our choice of branch of the $p$-adic logarithm  implies that the lower right triangle in (\ref{eq: diag crig}) commutes, and hence the composition of the upper horizontal maps gives $\Psi_{\pi,\log_\pi}: R\Gamma_{\HK,c}^{\rig}(Y_\pi) \rightarrow R\Gamma_{\rig,c}(Y_\pi/V^\sharp)_\pi$.
The composition of the bottom map is $\Psi_\pi^\cris: R\Gamma_{\HK,c}^{\cris}(X)_{\pi,\bbQ} \rightarrow R\Gamma_{\cris,c}(X/V^\sharp)_{\bbQ}$ by construction. 

By commutativity of the diagram, the composition
	\begin{align*}
	&R\Gamma_{\cris,c}(Y_\pi/W^0)_\bbQ\xleftarrow{\cong}R\Gamma^\rig_{\HK,c}(Y_\pi)\xrightarrow{\psi}R\Gamma_{\rig,c}(Y_\pi/\cS)\xleftarrow{\cong}R\Gamma_{\rig,c}(X_1/\cS)_\pi\rightarrow R\Gamma_{\cris,c}(X_1/\cS_{\PD})_{\pi,\bbQ}
	\end{align*}
coincides with the section
$$
R\Gamma_{\cris,c}(Y_\pi/W^0)_\bbQ\xrightarrow{s_\pi} R\Gamma_{\cris,c}(X_1/\cS_{\PD})_{\pi,\bbQ}
$$
of $j_0^{\ast}\colon R\Gamma_{\cris,c}(X_1/\cS_{\PD})_{\pi,\bbQ}\rightarrow R\Gamma_{\cris,c}(Y_\pi/W^0)_\bbQ$ and compatible with the Frobenius actions.
Hence the statement follows by uniqueness of such a section 
as explained in the proof of Proposition \ref{prop: cris sections}.
\end{proof}

%
\section{Poincar\'{e} duality}\label{sec: dual}
%

Let $Y$ be a proper strictly semistable scheme over $k^0$ with horizontal divisor $D$ and $(\sE,\Phi)\in F\Isoc^\dagger(Y/W^\varnothing)$.
As in \eqref{eq: pairing dual} we have a pairing for the rigid Hyodo--Kato cohomology which is compatible with Frobenius and monodromy, and similarly pairings for log rigid cohomology which over $W^0$ and $W^\varnothing$ are compatible with Frobenius:
	\begin{eqnarray}
	\label{eq: pairing}\nonumber && R\Gamma^\rig_\HK(Y,(\sE,\Phi))  \otimes R\Gamma^\rig_{\HK,c}(Y,(\sE^\vee,\Phi^\vee))  \rightarrow  R\Gamma^\rig_{\HK,c}(Y),\\
	\nonumber && R\Gamma_\rig(Y/W^0,(\sE,\Phi))  \otimes R\Gamma_{\rig,c}(Y/W^0,(\sE^\vee,\Phi^\vee))  \rightarrow  R\Gamma_{\rig,c}(Y/W^0),\\
	\nonumber &&R\Gamma_\rig(Y/W^\varnothing,(\sE,\Phi))  \otimes R\Gamma_{\rig,c}(Y/W^\varnothing,(\sE^\vee,\Phi^\vee))  \rightarrow  R\Gamma_{\rig,c}(Y/W^\varnothing),\\
	\nonumber && R\Gamma_\rig(Y/V^\sharp,\sE)_\pi  \otimes R\Gamma_{\rig,c}(Y/V^\sharp,\sE^\vee)_\pi  \rightarrow  R\Gamma_{\rig,c}(Y/V^\sharp)_\pi.	
	\end{eqnarray}

Because of  the comparison isomorphisms in Corollary \ref{cor: rig cris}, Poincar\'{e} duality for log crystalline cohomology over $W^0$ due to Tsuji explained in Proposition \ref{rem: Poincare cris} carries over to log rigid cohomology over $W^0$.

The morphisms in the diagram \eqref{eq: diag crig} (and its non-compactly supported version) are by construction compatible with the pairings.
Note that the compatibility of the Hyodo--Kato map $\Psi_{\pi,\log}$ with the pairing can be seen by
\[\Psi_{\pi,\log}(u^{[i]}\wedge u^{[j]})=\frac{(i+j)!}{i!j!}\frac{(-\log(\pi))^{i+j}}{(i+j)!}=\frac{(-\log(\pi))^i}{i!}\frac{(-\log(\pi))^j}{j!}=\Psi_{\pi,\log}(u^{[i]})\wedge\Psi_{\pi,\log}(u^{[j]}).\]
Therefore when $\sE$ is unipotent Poincar\'{e} duality for log rigid cohomology over $W^0$ implies those for rigid Hyodo--Kato cohomology and log rigid cohomology over $V^\sharp$, by Corollary \ref{cor: compare rigid}.

Furthermore, by definition of the usual and the compactly supported rigid Hyodo--Kato cohomology, $R\Gamma_\rig(Y/W^\varnothing,(\sE,\Phi))$ and $R\Gamma_{\rig,c}(Y/W^\varnothing,(\sE^\vee,\Phi^\vee))$ compute the cone of the respective monodromy operator on $R\Gamma^\rig_\HK(Y,(\sE,\Phi))$ and $R\Gamma^\rig_{\HK,c}(Y,(\sE^\vee,\Phi^\vee))$. 
Consequently, we obtain Poincar\'{e} duality for log rigid cohomology over $W^\varnothing$, but the degree shifts by one.

To summarise, we obtain the following theorem.

\begin{theorem}\label{thm: dual}
	Let $d:= \dim Y$.
	Let $(\sE,\Phi)\in F\Isoc^\dagger(Y/W^\varnothing)^\unip$, and denote the dual $F$-isocrystal by $(\sE^\vee,\Phi^\vee):=\sheafhom((\sE,\Phi),\sO_{Y/W^\varnothing})$.
	Then there exist canonical isomorphisms
	\begin{align}
	\nonumber & R\Gamma_\rig(Y/W^\varnothing,(\sE,\Phi)) \xrightarrow{\cong}R\Gamma_{\rig,c}(Y/W^\varnothing,(\sE^\vee,\Phi^\vee))^*[-2d-1](-d-1)&&\text{in }D^b(\Mod_F^\fin(\varphi)),\\
	\nonumber\label{eq: duality derived}&R\Gamma^\rig_\HK(Y,(\sE,\Phi)) \xrightarrow{\cong}R\Gamma^\rig_{\HK,c}(Y,(\sE^\vee,\Phi^\vee))^*[-2d](-d)&&\text{in }D^b(\Mod_F^\fin(\varphi,N)),\\
	\nonumber &R\Gamma_\rig(Y/W^0,(\sE,\Phi)) \xrightarrow{\cong}R\Gamma_{\rig,c}(Y/W^0,(\sE^\vee,\Phi^\vee))^*[-2d](-d)&&\text{in }D^b(\Mod_F^\fin(\varphi)),\\
	\nonumber &R\Gamma_\rig(Y/V^\sharp,\sE)_\pi \xrightarrow{\cong}R\Gamma_{\rig,c}(Y/V^\sharp,\sE^\vee)_\pi^*[-2d]&&\text{in }D^b(\Mod^\fin_K),
	\end{align}
	where $*$ denotes the (derived) internal Hom $R\underline{\Hom}(-,F)$ (or $R\underline{\Hom}(-,K)$).
	On the level of cohomology groups, we have isomorphisms
	\begin{align}
	\nonumber &H^i_\rig(Y/W^\varnothing,(\sE,\Phi)) \xrightarrow{\cong}H^{2d-i+1}_{\rig,c}(Y/W^\varnothing,(\sE^\vee,\Phi^\vee))^*(-d-1)&&\text{in }\Mod_F^\fin(\varphi),\\
	\nonumber&H^{\rig,i}_\HK(Y,(\sE,\Phi)) \xrightarrow{\cong}H^{\rig,2d-i}_{\HK,c}(Y,(\sE^\vee,\Phi^\vee))^*(-d)&&\text{in }\Mod_F^\fin(\varphi,N),\\
	\nonumber & H^i_\rig(Y/W^0,(\sE,\Phi)) \xrightarrow{\cong}H^{2d-i}_{\rig,c}(Y/W^0,(\sE^\vee,\Phi^\vee))^*(-d)&&\text{in }\Mod_F^\fin(\varphi),\\
	\nonumber & H^i_\rig(Y/V^\sharp,\sE)_\pi \xrightarrow{\cong}H^{2d-i}_{\rig,c}(Y/V^\sharp,\sE^\vee)_\pi^*&&\text{in }\Mod_K^\fin,
	\end{align}
	where for an integer $r\in\bbZ$ we denote by $(r)$ the Tate twist defined by multiplying $\varphi$ by $p^{-r}$.
	The Hyodo--Kato maps are compatible with Poincar\'{e} duality in the sense that $\Psi_{\pi,\log,K}^*=\Psi_{\pi,\log,K}^{-1}$.
\end{theorem}

\begin{remark}
Poincar\'e duality as stated above implies that the compactly supported log rigid cohomology defined in this paper is independent of the compactification (as long as it is a simple normal crossing compactification).

More precisely, let $Y$ be a proper strictly semistable  log scheme with horizontal divisor $D$ and $U=Y\backslash D$.
The statement of \cite[Prop.\,3.30]{ErtlYamada} together with Poincar\'e duality, implies the existence of isomorphisms of cohomology groups
\begin{align*}
	H^i_\rig(U/W^\varnothing) &\cong H^{2d-i+1}_{\rig,c}(Y/W^\varnothing)^*(-d-1),\\
	H^{\rig,i}_\HK(U) &\cong H^{\rig,2d-i}_{\HK,c}(Y)^*(-d),\\
	 H^i_\rig(U/W^0)& \cong H^{2d-i}_{\rig,c}(Y/W^0)^*(-d),\\
	H^i_\rig(U/V^\sharp)_\pi &\cong H^{2d-i}_{\rig,c}(Y/V^\sharp)_\pi^*.
	\end{align*}
As this is true for any simple normal crossing compactification $Y$ of $U$, the independence follows.
\end{remark}

 Finally we remark that the push-forward of compactly supported cohomology with respect to an open immersion given in Proposition \ref{prop: functoriality} is the dual of the pull-back of non-compactly supported cohomology.
We thank David Loeffler for pointing out that such a result might hold. 
 
\begin{proposition}
	Let $f\colon Y'\rightarrow Y$ be a morphism between proper strictly semistable log schemes over $k^0$ as in Proposition \ref{prop: functoriality}, i.e.\,$f$ is the identity as a morphism of underlying schemes, and the horizontal divisor $D$ of $Y$ is a union of the horizontal components of the horizontal divisor $D'$ of $Y'$.
	Let $\sE\in\Isoc^\dagger(Y/W^\varnothing)$.
	Then
		\[f^*\colon R\Gamma_\HK^\rig(Y,\sE)\rightarrow R\Gamma_\HK^\rig(Y',f^*\sE)\]
	and
		\[f_*\colon R\Gamma^\rig_{\HK,c}(Y',f^*\sE^\vee)\rightarrow R\Gamma^\rig_{\HK,c}(Y,\sE^\vee)\]
	are dual to each other via Poincar\'{e} duality in Theorem \ref{thm: dual}.
	
	We also have similar statements for log rigid cohomology over $W^\varnothing$, $W^0$, and $V^\sharp$.
\end{proposition}

\begin{proof}
	This follows immediately from Corollary \ref{cor: functorial cup}\ref{item: push cup}.
\end{proof}

\appendix
\section{Extension groups of $(\varphi,N)$-modules}
Let $F$ be a finite unramified extension of $\bbQ_p$.
As defined in Definition \ref{def: modules}, we denote by $\Mod_F(\varphi,N)$ the category of $(\varphi,N)$-modules over $F$ and by $\Mod_F^\fin(\varphi,N)$ the full subcategory of finite $(\varphi,N)$-modules.
The purpose of this section is to prove the following theorem:

\begin{theorem}\label{thm: ext compare}
		Let $D^b_{\fin}(\Mod_F(\varphi,N))$ be the full subcategory of $D^b(\Mod_F(\varphi,N))$ consisting of objects $M^\bullet$ such that $H^i(M^\bullet)$ is finite for all $i\in\bbZ$.
	Then the natural functor $D^b(\Mod_F^\fin(\varphi,N))\rightarrow D^b_{\fin}(\Mod_F(\varphi,N))$ is an equivalence of categories.
\end{theorem}

\begin{remark}
The theorem implies that the functor 
$D^b(\Mod_F^\fin(\varphi,N))\rightarrow D^b(\Mod_F(\varphi,N))$ is fully faithful.
This fact was used in the proof of \cite[Lem.\,2.5]{DegliseNiziol} \textit{without proof} in order to compute the extension groups in 
$\Mod_F^\fin(\varphi,N)$ by those in $\Mod_F(\varphi,N)$.
Note that an object of $\Mod_F(\varphi,N)$ (e.g.\,the object $M_{\varphi,N}$ used in the proof of \cite[Lem.\,2.5]{DegliseNiziol}) is in general \textit{not} the union of its finite subobjects.
For this reason, this gap seems not to be trivial to the authors.
\end{remark}

Before starting the proof of Theorem \ref{thm: ext compare}, we prepare some terminologies concerning $(\varphi,N)$-modules.
Let $R:=F[\varphi]$ be the set of polynomials in a variable $\varphi$ with the usual addition and the (non-commutative) product given by $\varphi a=\sigma(a)\varphi$ for $a\in F$. 
Clearly, the category of left $R$-modules is equivalent to the category of $\varphi$-modules. 
Thus by forgetting the monodromy operator, a $(\varphi,N)$-module can be seen as a left $R$-module.
We say that an object $M\in\Mod_F(\varphi,N)$ is $R$-\textit{torsion} (resp.\,$R$-\textit{torsion free}) if it is torsion (resp.\,torsion free) as a left $R$-module.
Similarly, we define $R$-torsion and non-$R$-torsion elements of $M$ by regarding $M$ as a left $R$-module.
To deal with phenomena related to the monodromy operator, it will also be useful to consider the isomorphism $\tau\colon R\xrightarrow{\cong}R$ defined by $\tau(a)=a$ for any $a\in F$ and $\tau(\varphi)=p\varphi$.

For $M\in\Mod_F(\varphi,N)$ and $x\in M$, we define
	\begin{align*}
	&\ord_N(x):=\begin{cases}\min\{i\in\bbN\mid N^ix=0\}&\text{if $N^ix=0$ for some $i$,}\\ \infty&\text{otherwise,}\end{cases}\\
	&\ord_{\varphi,N}(x):=\min_{f\in R\setminus\{0\}}\ord_N(fx).
	\end{align*}
	
\begin{lemma}\label{lem: order}
\begin{enumerate}
\item\label{item: order 1} We have $\ord_{\varphi,N}(x)\leq\ord_N(x)$ for any $x\in M$.
\item\label{item: order 2} Let $i\geqslant 0$. We have $i<\ord_{N,\varphi}(x)$ if and only if $N^ix$ is $R$-torsion free. If these conditions hold, $N^ifx$ is $R$-torsion free for any $f\in R\setminus\{0\}$.
\end{enumerate}
\end{lemma}

\begin{proof}
Noting that $\tau$ is bijective on $R$ and that $Nf=\tau(f)N$ for $f\in R$, the statements follow immediately from the definition.
\end{proof}

We will compare step by step the derived categories of the following full subcategories:
\begin{align*}
\Mod_F(\varphi,N)&\\
\cup\hspace{24pt}&\\
\Mod_F^\nil(\varphi,N)&:=\{M\in\Mod_F(\varphi,N)\mid \text{$\ord_N(x)<\infty$ for any $x\in M$}\}\\
\cup\hspace{24pt}&\\
\Mod_F^\tor(\varphi,N)&:=\{M\in\Mod_F^\nil(\varphi,N)\mid \text{$M$ is $R$-torsion}\}\\
\cup\hspace{24pt}&\\
\Mod_F^\fin(\varphi,N)&=\{M\in\Mod_F^\fin(\varphi,N)\mid \text{$\mathrm{dim}M<\infty$ and $\varphi$ is bijective}\}.
\end{align*}

As before, for $!,?\in\{\emptyset,\nil,\tor,\fin\}$, we denote by $D^b_{!}(\Mod_F^?(\varphi,N))$ the full subcategory of $D^b(\Mod_F^?(\varphi,N))$ consisting of objects $M^\bullet$ such that $H^i(M^\bullet)\in\Mod_F^!(\varphi,N)$ for all $i\in\bbZ$. 

\begin{lemma}\label{lem: usual vs nil}
	The natural functor $D^b(\Mod_F^\nil(\varphi,N))\rightarrow D^b_\nil(\Mod_F(\varphi,N))$ is an equivalence of categories.
\end{lemma}

\begin{proof}
	We first note that the proof of \cite[Lem.\,2.5]{DegliseNiziol} actually shows that the extension groups in $\Mod_F(\varphi,N)$ are computed as the cohomology groups of certain double complexes given in the statement of \textit{loc.cit.}
     	
	For an $F$-vector space $M$, we set $M'_{\varphi,N}:=\bigoplus_{i\geqslant 0}\prod_{j\geqslant 0}M_{i,j}$ with $M_{i,j}:=M$, endowed with the $F$-action and the $(\varphi,N)$-module structure induced by
	\begin{align*}
	&a(m_{i,j})_{i,j}:=(\sigma^j(a)m_{i,j})_{i,j} \text{ for $a\in F$,}&
	&\varphi(m_{i,j})_{i,j}:=(m_{i,j+1})_{i,j},&&N(m_{i,j})_{i,j}:=(p^{-j}m_{i+1,j})_{i,j}.
	\end{align*}
			
	Then $M'_{\varphi,N}$ is an object of $\Mod_F^\nil(\varphi,N)$, and the association $M\mapsto M'_{\varphi,N}$ gives a right adjoint of the forgetful functor.
	Now we can show that the extension groups in $\Mod_F^\nil(\varphi,N)$ are also computed by double complexes of the same form as in \cite[Lem.\,2.5]{DegliseNiziol} by replacing $M_{\varphi,N}$ in the proof of \textit{loc.cit.} by $M'_{\varphi,N}$.
	
	Consequently we have canonical isomorphisms 
	$$\Ext^n_{\Mod_F^\nil(\varphi,N)}(L,M)\xrightarrow{\cong}\Ext^n_{\Mod_F(\varphi,N)}(L,M)$$ for all $L,M\in\Mod_F^\nil(\varphi,N)$ and $n\in\bbZ$ and the lemma follows from \cite[Lem.\,1.4]{Beilinson1987}.
\end{proof}

For the other comparisons, we use the following theorem.

\begin{theorem}\label{thm: thick}
	Let $\cA$ be an abelian category, $\cA'\subset\cA$ a thick full subcategory, and $D_{\cA'}^b(\cA)\subset D^b(\cA)$ the full subcategory consisting of objects $X$ such that $H^i(X)\in\cA'$ for all $i\in\bbZ$.
	The natural functor $D^b(\cA')\subset D^b_{\cA'}(\cA)$ is an equivalence of categories if one of the following conditions holds:
	\begin{enumerate}
	\item\label{item: thick 1}for any monomorphism $A\rightarrow B$ in $\cA$ with $A\in\cA'$, there exists a morphism $B\rightarrow C$ in $\cA$ with $C\in\cA'$ such that the composite $A\rightarrow C$ is a monomorphism,
	\item\label{item: thick 2}for any epimorphism $B\rightarrow A$ in $\cA$ with $A\in\cA'$, there exists a morphism $C\rightarrow B$ in $\cA$ with $C\in\cA'$ such that the composite $C\rightarrow A$ is an epimorphism.
	\end{enumerate}
\end{theorem}

\begin{proof}
	This follows immediately from \cite[Thm.\,13.2.8]{KashiwaraSchapira} and its dual.
\end{proof}

To compare the derived categories of $\Mod_F^\nil(\varphi,N)$ and $\Mod_F^\tor(\varphi,N)$, we will need the following lemma:

\begin{lemma}\label{lem: linearly independent}
	Let $M\in\Mod_F^\nil(\varphi,N)$, $x\in M$, and $f\in R\setminus\{0\}$.
	Take $f\in R\setminus\{0\}$ such that $\alpha:=\ord_{\varphi,N}(x)=\ord_N(fx)$.
	Then the set $\{N^ifx\}_{i=0}^{\alpha-1}$ is linearly independent over $R$.
\end{lemma}

\begin{proof}
	If $\alpha=0$, the statement is trivial because the set $\{fx,Nfx,\ldots,N^{\alpha-1}fx\}$ is empty.
	Suppose that $\alpha>0$ and $\sum_{i=0}^{\alpha-1}g_iN^ifx=0$ for $g_i\in R$.
	Assume that $g_i\neq 0$ for some $i$, and let $i_0$ be the smallest one in such $i$'s.
	Then we have
	\[0=N^{\alpha-i_0-1}\sum_{i=0}^{\alpha-1}g_iN^ifx=\sum_{i=0}^{\alpha-1}N^{\alpha+i-i_0-1}\tau^{-i}(g_i)fx=N^{\alpha-1}\tau^{-i_0}(g_{i_0})fx.\]
	But this contradicts the minimality of $\alpha$.
	Thus we obtain $g_i=0$ for any $i$.
\end{proof}

\begin{lemma}\label{lem: nil vs tor}
	The natural functor $D^b(\Mod_F^\tor(\varphi,N))\rightarrow D^b_\tor(\Mod_F^\nil(\varphi,N))$ is an equivalence of categories.
\end{lemma}

\begin{proof}
	Let $A\hookrightarrow B$ be an injection in $\Mod_F^\nil(\varphi,N)$ and suppose that $A$ is $R$-torsion.
	Let $\Theta$ be the set of all quotients $C$ of $B$ in $\Mod_F^\nil(\varphi,N)$ such that the induced map $A\rightarrow C$ is injective.
	We define a partial order on $\Theta$ by setting $C_1<C_2$ if $C_2$ is a quotient of $C_1$.
	For any totally ordered subset $X\subset\Theta$, the map $A\rightarrow\varinjlim_{C\in X}C$ is still injective.
	Thus by Zorn's lemma there exists a maximal element $C_0\in\Theta$ with respect to this partial order.
	By Theorem \ref{thm: thick}\ref{item: thick 1}, it suffices to show that $C_0$ is $R$-torsion.
	
	Assume that there exists a non-$R$-torsion element $x\in C_0$. 
	Take $f\in R\setminus\{0\}$ such that $\alpha:=\ord_{\varphi,N}(x)=\ord_N(fx)$ and endow $\bigoplus_{i=0}^{\alpha-1}R$ with the monodromy operator defined by $N(g_i)_i:=(\tau(g_{i-1}))_i$.
	It is easy to see that
	\[\bigoplus_{i=0}^{\alpha-1}R\rightarrow C_0;\ (g_i)_i\mapsto\sum_{i=0}^{\alpha-1}g_iN^ifx\]
	defines a morphism in $\Mod_F^\nil(\varphi,N)$.
	Moreover this is injective by Lemma \ref{lem: linearly independent}.
	By Lemma \ref{lem: order} $N^ifx$ is $R$-torsion free for each $0\leq i\leq \alpha-1$.
	As $A$ is $R$-torsion, it follows that $A\rightarrow C_0/\bigoplus_{i=0}^{\alpha-1}R$ is injective.
	Since $\alpha>0$ again by Lemma \ref{lem: order}, this contradicts the maximality of $C_0$.
\end{proof}

\begin{lemma}\label{lem: tor vs fp}
	The natural functor $D^b(\Mod_F^\fin(\varphi,N))\rightarrow D^b_\fin(\Mod_F^\tor(\varphi,N))$ is an equivalence of categories.
\end{lemma}

\begin{proof}
	Let $B\rightarrow A$ be a surjection in $\Mod_F^\tor(\varphi,N)$ and suppose that $A$ is finite.
	Choose elements $x_1,\ldots,x_d\in B$ such that their images in $A$ form an $F$-basis, and let $C'\subset B$ be the $F$-subspace spanned by elements of the form $N^i\varphi^jx_k$ for all integers $i,j\geqslant 0$  and $k=1,\ldots,d$.
	Then $C'$ is stable under $\varphi$ and $N$, and moreover finite dimensional, since $B\in\Mod_F^\tor(\varphi,N)$.
	By construction, the induced map $C'\rightarrow A$ is surjective.
	
	Note that $\varphi$ on $C'$ is not necessarily bijective.
	However, as it is finite dimensional, there exists an integer $r\geqslant 0$ such that $\varphi^{r+m}(C')=\varphi^r(C')$ for any $m\geqslant 0$.
	Then $C:=\varphi^r(C')\subset C'$ is a finite subobject.
	For $k=1,\ldots,d$, let $y_k:=\varphi^{-r}\circ\varphi^r(x_k)$, 
	where $\varphi^{-r}$ denotes the inverse of $\varphi^r$ on $C$.
	Then we see that the composite $C\hookrightarrow C'\rightarrow A$ is surjective, since $y_k$ maps to $x_k$.
	Now our assertion follows by Theorem \ref{thm: thick}\ref{item: thick 2}.
\end{proof}
\begin{proof}[Proof of Theorem \ref{thm: ext compare}]
We obtain Theorem \ref{thm: ext compare} by combining Lemmas \ref{lem: usual vs nil}, \ref{lem: nil vs tor}, and \ref{lem: tor vs fp}.
\end{proof}

\section{The cup product on Hyodo--Kato cohomology}\label{app: cup product}

The purpose of this section is to give a precise definition of the cup product on the log rigid cohomology and on the Hyodo--Kato cohomology. 
Note that because of the use of simplicial objects and homotopy limits, some technical combinatorial arguments are required.

We first introduce complexes representing log rigid cohomology and Hyodo--Kato cohomology which allow a definition of the cup product.
Let $(T,\cT,\iota)$ be a widening, $Y$ a fine log scheme over $T$, and $\sE$ an object of $\Isoc^\dagger(Y/\cT)$.
To describe the log rigid cohomology in terms of log differential forms, take a local embedding datum $(Z_\lambda,\cZ_\lambda,i_\lambda,h_\lambda,\theta_\lambda)_{\lambda\in\Lambda}$ for $Y$ over $(T,\cT,\iota)$ such that $Z_\lambda$ and $\cZ_\lambda$ are affine for any $\lambda\in\Lambda$.
We define objects $(Z_{\underline{\lambda}},\cZ_{\underline{\lambda}},i_{\underline{\lambda}},h_{\underline{\lambda}},\theta_{\underline{\lambda}})$ of the log overconvergent site $\OC(Y/\cT)$ for $m\in\bbN$ and $\underline{\lambda}\in\Lambda^{m+1}$ as in \eqref{eq: product in OC}.
For $k\geq 1$ and $\mathbf{n}=(n_1,\ldots,n_k)\in\bbN^k$, we define $\cZ_{\underline{\lambda}}[\mathbf{n}]:=\cZ_{\underline{\lambda}}[\displaystyle\min_{1\leq r\leq k} n_r]$ (see Definition \ref{def: univ enl}).
Note that each $\cZ_{\underline{\lambda}}[\mathbf{n}]_\bbQ$ is an affinoid dagger space, 
and that we have $\cZ_{\underline{\lambda},\bbQ}=\bigcup_{n\in\bbN}\cZ_{\underline{\lambda}}[n]_\bbQ$ and $\cZ_{\underline{\lambda}}[\mathbf{n}]_\bbQ=\bigcap_{r=1}^k\cZ_{\underline{\lambda}}[n_r]_\bbQ$.

For $k\geq 1$ and $i\geq 0$, define
\[M_{\rig,k}^i(Y/\cT,\sE):=\bigoplus_{I\subset\{1,\ldots,k\}}\bigoplus_{m\in\bbN}\prod_{\underline{\lambda}\in\Lambda^{m+1}}\prod_{\mathbf{n}\in\bbN^k}\Gamma(\cZ_{\underline{\lambda}}[\mathbf{n}]_\bbQ,\sE_{\cZ_{\underline{\lambda}}[\mathbf{n}]}\otimes\omega^{i-m-\lvert I\rvert}_{\cZ_{\underline{\lambda}}[\mathbf{n}]/\cT,\bbQ}).\]
For $k\geq 1$, $I\subset\{1,\ldots,k\}$ and $r\in I$, we set
\begin{align*}
I_r^-:=\{s\mid s\in I, \ s\leq r\},&&I_r^+:=\{s-r\mid s\in I,\ s>r\},
\end{align*}
and let $1_r\in\bbN^k$ be the element whose $r$-th entry is $1$ and whose other entries are $0$.
For $m\in\bbN$, $\underline{\lambda}=(\lambda_0,\ldots,\lambda_m)\in\Lambda^{m+1}$ and $\nu=0,\ldots,m$, we let
\[\underline{\lambda}(\nu):=(\lambda_0,\ldots,\lambda_{\nu-1},\lambda_{\nu+1},\ldots,\lambda_m).\]

For an element $\alpha\in M^i_{\rig,k}(Y/\cT,\sE)$, we denote by $\alpha_{I,\underline{\lambda},\mathbf{n}}$ the entry of $\alpha$ at $I\subset\{1,\ldots,k\}$, $\underline{\lambda}\in\Lambda^{m+1}$, and $\mathbf{n}\in\bbN^k$, and define an element $D\alpha\in M_{k,\rig}^{i+1}(Y/\cT,\sE)$ by
\[(D\alpha)_{I,\underline{\lambda},\mathbf{n}}:=\nabla\alpha_{I,\underline{\lambda},\mathbf{n}}+\sum_{\nu=0}^m(-1)^{i-m-\lvert I\rvert+\nu+1}p_\nu^\ast\alpha_{I,\underline{\lambda}(\nu),\mathbf{n}}+\sum_{r\in I}(-1)^{i+\lvert I_r^+\rvert}(\alpha_{I\setminus\{r\},\underline{\lambda},\mathbf{n}}-\alpha_{I\setminus\{r\},\underline{\lambda},\mathbf{n}+1_r}|_{\cZ_{\underline{\lambda}}[\mathbf{n}]_\bbQ}),\]
where $\nabla$ is the log connection associated to $\sE$ and $p_\nu\colon \cZ_{\underline{\lambda}}[\mathbf{n}]_\bbQ\rightarrow\cZ_{\underline{\lambda}(\nu)}[\mathbf{n}]_\bbQ$ denotes the natural projection.
Then one can see that $D\circ D=0$ by direct computation or by the interpretation as the total complex explained in Remark \ref{rem: tot homotopy limit} below, namely $M_{\rig,k}^\bullet(Y/\cT,\sE)$ forms a complex.

Similarly, for a fine log scheme $Y$ over $k^0$ and $\sE\in\Isoc^\dagger(Y/W^\varnothing)$, we take a local embedding datum $(Z_\lambda,\cZ_\lambda,i_\lambda,h_\lambda,\theta_\lambda)_{\lambda\in\Lambda}$ for $Y$ over $(k^0,\cS,\tau)$ such that $Z_\lambda$ and $\cZ_\lambda$ are affine for any $\lambda\in\Lambda$, and define objects $(Z_{\underline{\lambda}},\cZ_{\underline{\lambda}},i_{\underline{\lambda}},h_{\underline{\lambda}},\theta_{\underline{\lambda}})$ of $\OC(Y/\cS)$ for $m\in\bbN$ and $\underline{\lambda}\in\Lambda^{m+1}$ as before.
Then we define a complex $M_{\HK,k}^\bullet(Y,\sE)$ by
\[M_{\HK,k}^i(Y,\sE):=\bigoplus_{I\subset\{1,\ldots,k\}}\bigoplus_{m\in\bbN}\prod_{\underline{\lambda}\in\Lambda^{m+1}}\prod_{\mathbf{n}\in\bbN^k}\Gamma(\cZ_{\underline{\lambda}}[\mathbf{n}]_\bbQ,\sE_{\cZ_{\underline{\lambda}}[\mathbf{n}]}\otimes\omega^{i-m-\lvert I\rvert}_{\cZ_{\underline{\lambda}}[\mathbf{n}]/W^\varnothing,\bbQ}[u])\]
with the differential maps $D$ defined in the same way as above.

\begin{remark}\label{rem: tot homotopy limit}
	The complex $M_{\rig,k}^\bullet(Y/\cT,\sE)$ can be interpreted as the total complex of the $(k+2)$-ple complex in the shape of the $k$-dimensional cube  whose vertices are the double complex $\prod_{\mathbf{n}\in\bbN^k}\Gamma(\cZ_\bullet[\mathbf{n}]_\bbQ,\sE_{\cZ_\bullet[\mathbf{n}]}\otimes\omega^{\star}_{\cZ_\bullet[\mathbf{n}]/\cT,\bbQ})$ with differentials defined as follows:
	The first and second differentials are those on $\prod_{\mathbf{n}\in\bbN^k}\Gamma(\cZ_\bullet[\mathbf{n}]_\bbQ,\sE_{\cZ_\bullet[\mathbf{n}]}\otimes\omega^{\star}_{\cZ_\bullet[\mathbf{n}]/\cT,\bbQ})$ induced by $\nabla$ and the face maps of $\cZ_\bullet[\mathbf{n}]$, respectively.
	The $(r+2)$-th differentials $\eth_r$ are induced by the maps
	\[\prod_{\mathbf{n}\in\bbN^k}\Gamma(\cZ_m[\mathbf{n}]_\bbQ,\sE_{\cZ_m[\mathbf{n}]}\otimes\omega^i_{\cZ_m[\mathbf{n}]/\cT,\bbQ})
	\rightarrow
	\prod_{\mathbf{n}\in\bbN^k}\Gamma(\cZ_m[\mathbf{n}]_\bbQ,\sE_{\cZ_m[\mathbf{n}]}\otimes\omega^i_{\cZ_m[\mathbf{n}]/\cT,\bbQ})\]
	defined by
	\begin{align*}
	\eth_r((\alpha_{\mathbf{n}})_{\mathbf{n}}):=(\alpha_{\mathbf{n}}-\alpha_{\mathbf{n}+1_r}|_{\cZ_m[\mathbf{n}]_\bbQ})_{\mathbf{n}}.
	\end{align*}
	For example, the cases for $k=1,2$ are described as follows:
\begin{align*}
&M_{\rig,1}^\bullet=\Tot\left[\prod_{n\in\bbN}\Gamma(\cZ_\bullet[n]_\bbQ,\sE_{\cZ_{\bullet}[n]}\otimes\omega^{\star}_{\cZ_\bullet[n]/\cT,\bbQ})
	\xrightarrow{\eth_1}\prod_{n\in\bbN}\Gamma(\cZ_\bullet[n]_\bbQ,\sE_{\cZ_\bullet[n]}\otimes\omega^{\star}_{\cZ_\bullet[n]/\cT,\bbQ})\right],\\
&M_{\rig,2}^\bullet=\Tot\left[
\begin{aligned}
\xymatrix{
\displaystyle\prod_{\mathbf{n}\in\bbN^2}\Gamma(\cZ_\bullet[\mathbf{n}]_\bbQ,\sE_{\cZ_\bullet[\mathbf{n}]}\otimes\omega^{\star}_{\cZ_\bullet[\mathbf{n}]/\cT,\bbQ})\ar[r]^-{\eth_1}
&\displaystyle\prod_{\mathbf{n}\in\bbN^2}\Gamma(\cZ_\bullet[\mathbf{n}]_\bbQ,\sE_{\cZ_\bullet[\mathbf{n}]}\otimes\omega^{\star}_{\cZ_\bullet[\mathbf{n}]/\cT,\bbQ})\\
\displaystyle\prod_{\mathbf{n}\in\bbN^2}\Gamma(\cZ_\bullet[\mathbf{n}]_\bbQ,\sE_{\cZ_\bullet[\mathbf{n}]}\otimes\omega^{\star}_{\cZ_\bullet[\mathbf{n}]/\cT,\bbQ})\ar[r]^-{\eth_1}\ar[u]^-{\eth_2}
&\displaystyle\prod_{\mathbf{n}\in\bbN^2}\Gamma(\cZ_\bullet[\mathbf{n}]_\bbQ,\sE_{\cZ_\bullet[\mathbf{n}]}\otimes\omega^{\star}_{\cZ_\bullet[\mathbf{n}]/\cT,\bbQ})\ar[u]^-{\eth_2}
}
\end{aligned}
\right].
\end{align*}

Similarly, $M_{k,\HK}^\bullet(Y,\sE)$ is interpreted as the total complex of the $(k+2)$-ple complex in the shape of the $k$-dimensional cube whose vertices are the double complex $\prod_{\mathbf{n}\in\bbN^k}\Gamma(\cZ_\bullet[\mathbf{n}]_\bbQ,\sE_{\cZ_\bullet[\mathbf{n}]}\otimes\omega^{\star}_{\cZ_\bullet[\mathbf{n}]/W^\varnothing,\bbQ}[u])$ with differentials defined by the same rule as above.
\end{remark}

\begin{remark}
We point out that for the definition of the cup product on the various rigid complexes, only the cases $k=1,2,3$ are required. 
However, we deal with general $k$ in order to express them in a unified notation.
\end{remark}

For a fine log scheme $Y$ over $k^0$, let $(\cI(Y),\cR(Y,-),M_k^\bullet(Y,-),\cD)$ be one of the following:
\begin{itemize}
\item $\cI(Y)=\Isoc^\dagger(Y/\cT)$, $\cR(Y,\sE)=R\Gamma_\rig(Y/\cT)$, $M_k^\bullet(Y,\sE)=M_{k,\rig}(Y/\cT,\sE)$ for $\sE\in\cI(Y)$, and $\cD=D^+(\Mod_F)$ where $\cT=W^\varnothing$ or $W^0$,
\item $\cI(Y)=\Isoc^\dagger(Y/V^\sharp)_\pi$, $\cR(Y,\sE)=R\Gamma_\rig(Y/V^\sharp)_\pi$, $M_k^\bullet(Y,\sE)=M_{k,\rig}(Y/V^\sharp,\sE)$ for $\sE\in\cI(Y)$, and $\cD=D^+(\Mod_K)$, 
\item $\cI(Y)=\Isoc^\dagger(Y/W^\varnothing)$, $\cR(Y,\sE)=R\Gamma^\rig_\HK(Y,\sE)$, $M_k^\bullet(Y,\sE)=M_{\HK,k}^\bullet(Y,\sE)$ for $\sE\in\cI(Y)$, and $\cD=D^+(\Mod_F(N))$.
\end{itemize}

For an order preserving injection $\iota\colon\{1,\ldots,k\}\hookrightarrow\{1,\ldots,\ell\}$ and $i\in\bbN$, we define a map $f_\iota\colon M_k^i(Y,\sE)\rightarrow M_\ell^i(Y,\sE)$ by
\[(f_\iota\alpha)_{I,\underline{\lambda},\mathbf{n}}:=\begin{cases}\alpha_{\iota^{-1}(I),\underline{\lambda},\iota^\ast\mathbf{n}}|_{\cZ_{\underline{\lambda}}[\mathbf{n}]_\bbQ}
&\text{if $I\subset\iota(\{1,\ldots,k\})$},\\
0&\text{otherwise},
\end{cases}
\]
where we set $\iota^\ast\mathbf{n}:=(n_{\iota(1)},\ldots,n_{\iota(k)})$ for any $\mathbf{n}\in\bbN^\ell$.
In the rest of the paper, we will omit the notation $|_{\cZ_{\underline{\lambda}}[\mathbf{n}]_\bbQ}$ when there is no afraid of confusion.
Then for $I\subset\iota(\{1,\ldots,k\})$ we have
\begin{align*}
(f_\iota D\alpha)_{I,\underline{\lambda},\mathbf{n}}=&
\nabla\alpha_{\iota^{-1}(I),\underline{\lambda},\iota^\ast\mathbf{n}}
+\sum_{\nu=0}^m(-1)^{i-m-\lvert I\rvert+\nu+1}p_\nu^\ast\alpha_{\iota^{-1}(I),\underline{\lambda}(\nu),\iota^\ast\mathbf{n}}\\
&+\sum_{s\in\iota^{-1}(I)}(-1)^{i+\lvert\iota^{-1}(I)_s^+\rvert}(\alpha_{\iota^{-1}(I)\setminus\{s\},\underline{\lambda},\iota^\ast\mathbf{n}}-\alpha_{\iota^{-1}(I)\setminus\{s\},\underline{\lambda},\iota^\ast\mathbf{n}+1_s})\\
=&
\nabla\alpha_{\iota^{-1}(I),\underline{\lambda},\iota^\ast\mathbf{n}}
+\sum_{\nu=0}^m(-1)^{i-m-\lvert I\rvert+\nu+1}p_\nu^\ast\alpha_{\iota^{-1}(I),\underline{\lambda}(\nu),\iota^\ast\mathbf{n}}\\
&+\sum_{r\in I}(-1)^{i+\lvert I_r^+\rvert}(\alpha_{\iota^{-1}(I\setminus\{r\}),\underline{\lambda},\iota^\ast\mathbf{n}}-\alpha_{\iota^{-1}(I\setminus\{r\}),\underline{\lambda},\iota^\ast(\mathbf{n}+1_r)})\\
=&(D f_\iota\alpha)_{I,\underline{\lambda},\mathbf{n}}.
\end{align*}

Thus $f_\iota$ gives indeed a morphism of complexes $M_k^\bullet(Y,\sE)\rightarrow M_\ell^\bullet(Y,\sE)$.

\begin{lemma}\label{lem: inclusion}
	Let $\iota\colon\{1,\ldots,k-1\}\hookrightarrow\{1,\ldots,k\}$ be an order preserving injection.
	Then the morphism $f_\iota\colon M_{k-1}^\bullet(\sE)\rightarrow M_k^\bullet(\sE)$ is a quasi-isomorphism.
\end{lemma}

\begin{proof}
	We only prove the assertion for the case $M_k^\bullet(Y,\sE)=M_{\HK,k}^\bullet(Y,\sE)$, since the same proof works for the case $M_k^\bullet(Y,\sE)=M_{\rig,k}^\bullet(Y/\cT,\sE)$.
	Let $s$ be the unique element of $\{1,\ldots,k\}\setminus \iota(\{1,\ldots,k-1\})$.
	Recall that $M_k^\bullet(Y,\sE)$ can be interpreted as the total complex of a $k$-dimensional cube consisting of double complexes, as in Remark \ref{rem: tot homotopy limit}.
	In particular, $M_k^\bullet(Y,\sE)$ is written as the mapping cone of the morphism induced by the $(s+2)$-th differential maps between the total complexes of $(k-1)$-dimensional cubes.
	Namely we have an isomorphism $M_k^\bullet(Y,\sE)\cong\Cone(A^\bullet\xrightarrow{\eth_s} B^\bullet)[-1]$ with
	\begin{align*}
	&A^i=\bigoplus_{\substack{I\subset\{1,\ldots,k\}\\ I\not\ni s}}\bigoplus_{m\in\bbN}\prod_{\mathbf{n}\in\bbN^k}\prod_{\underline{\lambda}\in\Lambda^{m+1}}\Gamma(\cZ_{\underline{\lambda}}[\mathbf{n}]_\bbQ,\sE_{\cZ_{\underline{\lambda}}[\mathbf{n}]}\otimes\omega^{i-m-\lvert I\rvert}_{\cZ_{\underline{\lambda}}[\mathbf{n}]/W^\varnothing,\bbQ}[u]),\\
	&B^i=\bigoplus_{\substack{I\subset\{1,\ldots,k\}\\ I\ni s}}\bigoplus_{m\in\bbN}\prod_{\mathbf{n}\in\bbN^k}\prod_{\underline{\lambda}\in\Lambda^{m+1}}\Gamma(\cZ_{\underline{\lambda}}[\mathbf{n}]_\bbQ,\sE_{\cZ_{\underline{\lambda}}[\mathbf{n}]}\otimes\omega^{i-m-\lvert I\rvert}_{\cZ_{\underline{\lambda}}[\mathbf{n}]/W^\varnothing,\bbQ}[u]),\\
	&\eth_s\alpha_{I,\underline{\lambda},\mathbf{n}}=\alpha_{I\setminus\{s\},\underline{\lambda},\mathbf{n}}-\alpha_{I\setminus\{s\},\underline{\lambda},\mathbf{n}+1_s},
	\end{align*}
	and the morphism $M_{k-1}^\bullet(Y,\sE)\rightarrow\Cone(A^\bullet\xrightarrow{\eth_s} B^\bullet)[-1]$ induced by the morphism $g_\iota\colon M_{k-1}^\bullet(Y,\sE)\rightarrow A^\bullet$ with $(g_\iota\alpha)_{I,\underline{\lambda},\mathbf{n}}= \alpha_{\iota^{-1}(I),\underline{\lambda},\iota^*\mathbf{n}}$
	corresponds to $f_\iota$.
	Thus it suffices to show that
	\[0\rightarrow M_{k-1}^i(Y,\sE)\xrightarrow{g_\iota}A^i\xrightarrow{\eth_s}B^i\rightarrow 0\]
	is exact for any $i\in\bbN$.
	Without loss of generality, we may suppose that $s=k$. 
	Then the exactness of the above sequence follows if we show that
	\begin{align*}
	0\rightarrow \Gamma(\cZ_{\underline{\lambda}}[\mathbf{n}]_\bbQ,\sE_{\cZ_{\underline{\lambda}}[\mathbf{n}]}\otimes\omega^j_{\cZ_{\underline{\lambda}}[\mathbf{n}]/W^\varnothing,\bbQ}[u])\xrightarrow{g}
	\prod_{n\in\bbN}\Gamma(\cZ_{\underline{\lambda}}[(\mathbf{n},n)]_\bbQ,\sE_{\cZ_{\underline{\lambda}}[(\mathbf{n},n)]}\otimes\omega^j_{\cZ_{\underline{\lambda}}[(\mathbf{n},n)]/W^\varnothing,\bbQ}[u])\\
	\xrightarrow{\eth}\prod_{n\in\bbN}\Gamma(\cZ_{\underline{\lambda}}[(\mathbf{n},n)]_\bbQ,\sE_{\cZ_{\underline{\lambda}}[(\mathbf{n},n)]}\otimes\omega^j_{\cZ_{\underline{\lambda}}[(\mathbf{n},n)]/W^\varnothing,\bbQ}[u])\rightarrow 0
	\end{align*}
	is exact for any $j\in\bbN$, $\underline{\lambda}\in\Lambda^{m+1}$ and $\mathbf{n}=(n_1,\ldots,n_{k-1})\in\bbN^{k-1}$, where we set $(\mathbf{n},n_k):=(n_1,\ldots,n_{k-1},n)$, and the maps $g$ and $\eth$ are defined by
	\begin{align*}
	g(\alpha)=(\alpha|_{\cZ_{\underline{\lambda}}[(\mathbf{n},n)]_\bbQ})_n,&&\eth((\beta_n)_n)=(\beta_n-\beta_{n+1}|_{\cZ_{\underline{\lambda}}[(\mathbf{n},n)]_\bbQ})_n.
	\end{align*}
	However this is clear since we have $\cZ_{\underline{\lambda}}[(\mathbf{n},n)]_\bbQ=\cZ_{\underline{\lambda}}[\mathbf{n}]_\bbQ$ for any $n\geq \displaystyle\min_{1\leq r\leq k-1}n_r$.
\end{proof}

\begin{lemma}\label{lem: homotopy of inclusions}
	For any $r,s\in\{1,\ldots,k\}$, let $\iota_r$ and $\iota_s$ be the inclusions $\{1\}\hookrightarrow\{1,\ldots,k\}$ sending $1$ to $r$ and $s$, respectively.
	Then the morphisms $f_{\iota_r}$ and $f_{\iota_s}$ are homotopic to each other as morphisms $M_1^\bullet(Y,\sE)\rightarrow M_k^\bullet(Y,\sE)$.
\end{lemma}

\begin{proof}
	For $i\in\bbN$, we define a map $h\colon M_1^i(Y,\sE)\rightarrow M_k^{i-1}(Y,\sE)$ by setting
	\[(h\alpha)_{I,\underline{\lambda},\mathbf{n}}:=\begin{cases}
	(-1)^i\displaystyle\sum_{a=n_r}^{n_s-1}\alpha_{\{1\},\underline{\lambda},a}&\text{if $I=\emptyset$ and $n_r<n_s$},\\
	(-1)^{i+1}\displaystyle\sum_{a=n_s}^{n_r-1}\alpha_{\{1\},\underline{\lambda},a}&\text{if $I=\emptyset$ and $n_r>n_s$},\\
	0&\text{otherwise}
	\end{cases}
	\]
	for any $\alpha\in M_1^i(Y,\sE)$.
	Then we have
	\begin{align*}
	((Dh+hD)\alpha)_{I,\underline{\lambda},\mathbf{n}}&=\begin{cases}
	-\alpha_{\emptyset,\underline{\lambda},n_r}+\alpha_{\emptyset,\underline{\lambda},n_s}&\text{if $I=\emptyset$},\\
	-\alpha_{\{1\},\underline{\lambda},n_r}&\text{if $I=\{r\}$},\\
	\alpha_{\{1\},\underline{\lambda},n_s}&\text{if $I=\{s\}$},\\
	0&\text{otherwise}
	\end{cases}\\
	&=((f_{\iota_s}-f_{\iota_r})\alpha)_{I,\underline{\lambda},\mathbf{n}}.
	\end{align*}
	Thus $h$ gives a homotopy between $f_{\iota_r}$ and $f_{\iota_s}$.
\end{proof}

The following proposition allows the definition of a cup product on  $\cR(Y,\sE)$ via the complexes $M_k^\bullet(Y,\sE)$.

\begin{proposition}\label{prop: M iso R}
	For an order preserving injection $\iota\colon\{1,\ldots,k\}\hookrightarrow\{1,\ldots,\ell\}$ the morphism $f_\iota\colon M_k^\bullet(Y,\sE)\rightarrow M_\ell^\bullet(Y,\sE)$ is a quasi-isomorphism, and independent of $\iota$ as a morphism in the derived category.
	Moreover, there exists a canonical quasi-isomorphism $\cR(Y,\sE)\cong M_k^\bullet(Y,\sE)$ for any $k\geq 1$.
\end{proposition}

\begin{proof}
	That $\iota$ is a quasi-isomorphism follows by repeatedly applying 
	Lemma \ref{lem: inclusion}.
	Let $r:=\iota(1)$. 
	Then in the derived category $f_\iota$ is written as the composition
	\[M_k^\bullet(Y,\sE)\xrightarrow[f_{\iota_1}^{-1}]{\cong}M_1^\bullet(Y,\sE)\xrightarrow[f_{\iota_r}]{\cong}M_\ell^\bullet(Y,\sE).\]
	Thus it is independent of the choice of $\iota$ by Lemma \ref{lem: homotopy of inclusions}.
	Finally we have
	\begin{align*}
	&\cR(Y,\sE)\cong M_{1}^\bullet(Y,\sE)
	\xrightarrow[f_\iota]{\cong} M_k^\bullet(Y,\sE),
	\end{align*}
	where left isomorphism is given by \cite[Remark 2.39 and Proposition 4.5]{Yamada2020}, and the right isomorphism is independent of the choice of $\iota\colon\{1\}\hookrightarrow\{1,\ldots,k\}$ as proved above.
\end{proof}

Let $\sigma$ be a permutation of $\{1,\ldots,k\}$.
For $I\subset\{1,\ldots,k\}$, we denote by $\wt\sigma_I$ the composition
\[\{1,\ldots,\lvert I\rvert\}\cong \sigma^{-1}(I)\xrightarrow[\sigma]{\cong} I\cong\{1,\ldots,\lvert I\rvert\}\]
where the first and third maps are the unique order preserving bijections.
Let $\delta(\sigma,I):=\mathrm{sgn}(\wt\sigma_I)$.

\begin{lemma}\label{lem: delta}
Let $\sigma$ be a permutation of $\{1,\ldots,k\}$.
Let $I\subset\{1,\ldots,k\}$ and $r\in I$.
	We set $J:=\sigma^{-1}(I)$ and $s:=\sigma^{-1}(r)$.
	Then have
	\[\delta(\sigma,I\setminus\{r\})=(-1)^{\lvert J_s^+\rvert-\lvert I_r^+\rvert}\delta(\sigma,I).\]
\end{lemma}

\begin{proof}
	Let $\sigma'$ be another permutation and set $K:=\sigma'^{-1}(J)$ and $t:=\sigma'^{-1}(s)$.
	If the assertion holds for $\sigma$ and $\sigma'$, then we have
	\begin{align*}
	\delta(\sigma\sigma',I\setminus\{r\})&=\delta(\sigma,I\setminus\{r\})\delta(\sigma',\sigma^{-1}(I\setminus\{r\}))=(-1)^{\lvert J_s^+\rvert-\lvert I_r^+\rvert}\delta(\sigma,I)\cdot (-1)^{\lvert K_t^+\rvert-\lvert J_s^+\rvert}\delta(\sigma',\sigma^{-1}(I))\\
	&=(-1)^{\lvert K_t^+\rvert-\lvert I_r^+\rvert}\delta(\sigma\sigma',I),
	\end{align*}
	hence the assertion holds also for $\sigma\sigma'$.
	Thus we may reduce to the case that $\sigma$ is a transposition of $a,b\in\{1,\ldots,k\}$, and the assertion follows from direct computation.
	Indeed, when $r\neq a,b$ we have $\delta(\sigma,I)=\delta(\sigma,I\setminus\{r\})=-1$ and $\lvert J_s^+\rvert=\lvert I_r^+\rvert$.
	When $r=a<b$, we have $\delta(\sigma,I)=-1$, $\delta(\sigma,I\setminus\{r\})=(-1)^{b-a+1}$, and $\lvert I_r^+\rvert=\lvert J_s^+\rvert+b-a$.
	The case for $a<b=r$ is also computed similarly.
\end{proof}

For a permutation $\sigma$ of $\{1,\ldots,k\}$, we denote again by $\sigma$ the isomorphism $M_k^\bullet(Y,\sE)\xrightarrow{\cong}M_k^\bullet(Y,\sE)$ defined by
\[(\sigma\alpha)_{I,\underline{\lambda},\mathbf{n}}:=\delta(\sigma,I)\alpha_{\sigma^{-1}(I),\underline{\lambda},\sigma^\ast\mathbf{n}},\]
where we set $\sigma^*\mathbf{n}=(n_{\sigma(1)},\ldots,n_{\sigma(k)})$ as before.
This is indeed a morphism of complexes, by the computation
\begin{align*}
(\sigma D\alpha)_{I,\underline{\lambda},\mathbf{n}}=&\delta(\sigma,I)\nabla\alpha_{\sigma^{-1}(I),\underline{\lambda},\sigma\ast\mathbf{n}}+\sum_{\nu=0}^m(-1)^{i-m-\lvert I\rvert+\nu+1}\delta(\sigma,I)p_\nu^\ast\alpha_{\sigma^{-1}(I),\underline{\lambda}(\nu),\sigma^\ast\mathbf{n}}\\
&+\sum_{s\in\sigma^{-1}(I)}(-1)^{i+\lvert\sigma^{-1}(I)_s^+\rvert}\delta(\sigma,I)(\alpha_{\sigma^{-1}(I)\setminus\{s\},\underline{\lambda},\sigma^\ast\mathbf{n}}-\alpha_{\sigma^{-1}(I)\setminus\{s\},\underline{\lambda},\sigma^\ast\mathbf{n}+1_s})\\
=&\delta(\sigma,I)\nabla\alpha_{\sigma^{-1}(I),\underline{\lambda},\sigma\ast\mathbf{n}}+\sum_{\nu=0}^m(-1)^{i-m-\lvert I\rvert+\nu+1}\delta(\sigma,I)p_\nu^\ast\alpha_{\sigma^{-1}(I),\underline{\lambda}(\nu),\sigma^\ast\mathbf{n}}\\
&+\sum_{r\in I}(-1)^{i+\lvert I_r^+\rvert}\delta(\sigma,I\setminus\{r\})(\alpha_{\sigma^{-1}(I\setminus\{r\}),\underline{\lambda},\sigma^\ast\mathbf{n}}-\alpha_{\sigma^{-1}(I\setminus\{r\}),\underline{\lambda},\sigma^\ast(\mathbf{n}+1_r)})\\
=&(D\sigma\alpha)_{I,\underline{\lambda},\mathbf{n}}
\end{align*}
for any $\alpha\in M_k^i(Y,\sE)$, where the second equality follows by Lemma \ref{lem: delta}.

We define an involution $\tau$ on $M_k^\bullet(Y,\sE)$ by setting
\[(\tau\alpha)_{I,\underline{\lambda},\mathbf{n}}:=(-1)^{\frac{m(m+1)}2}\alpha_{I,(\lambda_m,\ldots,\lambda_0),\mathbf{n}}.\]

\begin{lemma}\label{lem: sigma tau}
	The endmorphisms on $M_k^\bullet(Y,\sE)$ in the derived category induced by $\sigma$ and $\tau$ are the identity.
\end{lemma}

\begin{proof}
Since $\delta(\sigma,I)=1$ when $\lvert I\rvert=0$ or $1$, we see that $\sigma\circ f_{\iota_1}=f_{\iota_{\sigma(1)}}$.
Therefore it follows by Lemma \ref{lem: homotopy of inclusions} that $\sigma$ in the derived category coincides with the identity.

That $\tau$ is a morphism of complexes and homotopic to the identity follows by computations similarly to \cite[\href{https://stacks.math.columbia.edu/tag/01FP}{Tag 01FP}]{stacks}.
\end{proof}

Now we are ready to discuss the cup product on the log rigid and Hyodo--Kato cohomology.
Let $\sE,\sE'\in\Isoc^\dagger(Y/\cT)$.
For integers $i,j,k,\ell,m,\mu$ with $i,j\geq 0$, $k,\ell\geq 1$, $0\leq \mu\leq m$ and a subset $I\subset\{1,\ldots,k+\ell\}$, let
\begin{equation}\label{eq: sign}\epsilon(i,j,k,\ell,m,\mu,I):=(-1)^{(\lvert I_k^-\rvert+\mu)(\lvert I_k^+\rvert+j)+\mu(m+1)}.\end{equation}
Note that $\epsilon(i,j,k,\ell,m,\mu,I)$ is independent of $i$, but we use this notation to facilitate the arguments in Remark \ref{rem: uniqueness of epsilon} later.

For $i,j\geq 0$, $k,\ell\geq 1$ and elements $\alpha\in M_k^i(\sE)$ and $\beta\in M_\ell^j(\sE')$, we define an element $\alpha\cup\beta\in M_{k+\ell}^{i+j}(\sE\otimes\sE')$ by
\begin{equation}\label{eq: def cup}(\alpha\cup\beta)_{I,\underline{\lambda},\mathbf{n}}:=\sum_{\mu=0}^m\epsilon(i,j,k,\ell,m,\mu,I)q_{\leq\mu}^{\ast}\alpha_{I_k^-,\underline{\lambda}^{\leq\mu},\mathbf{n}^{\leq k}}\wedge q_{\geq\mu}^{\ast}\beta_{I_k^+,\underline{\lambda}^{\geq\mu},\mathbf{n}^{>k}}\end{equation}
where we set
\begin{align*}
&\underline{\lambda}^{\leq\mu}:=(\lambda_0,\ldots,\lambda_\mu),&&\underline{\lambda}^{\geq\mu}:=(\lambda_\mu,\ldots,\lambda_m),\\
&\mathbf{n}^{\leq k}:=(n_1,\ldots,n_k),&&\mathbf{n}^{>k}:=(n_{k+1},\ldots,n_{k+\ell})
\end{align*}
for $\underline{\lambda}=(\lambda_0,\ldots,\lambda_m)\in\Lambda^{m+1}$ and $\mathbf{n}=(n_1,\ldots,n_{k+\ell})\in\bbN^{k+\ell}$, and let $q_{\leq\mu}\colon\cZ_{\underline{\lambda}[\mathbf{n}]_\bbQ}\rightarrow\cZ_{\underline{\lambda}^{\leq\mu}[\mathbf{n}]_\bbQ}$ and $q_{\geq\mu}\colon\cZ_{\underline{\lambda}[\mathbf{n}]_\bbQ}\rightarrow\cZ_{\underline{\lambda}^{\geq\mu}[\mathbf{n}]_\bbQ}$ denote the canonical projections.
The wedge product $\wedge$ is induced by tensor products of sections of $\sE$ and $\sE'$ and the product structure on $\omega^\bullet_{\cZ_{\underline{\lambda}}[\mathbf{n}]/\cT,\bbQ}$ or $\omega^\bullet_{\cZ_{\underline{\lambda}}[\mathbf{n}]/W^\varnothing,\bbQ}[u]$.

We have equalities
\begin{small}
\begin{align}
\label{eq: D cup calculation 1}&\hspace{-8pt}(D(\alpha\cup\beta))_{I,\underline{\lambda},\mathbf{n}}=\\
\nonumber&\sum_{\mu=0}^m\epsilon(i,j,k,\ell,m,\mu,I)q_{\leq\mu}^\ast\nabla\alpha_{I_k^-,\underline{\lambda}^{\leq\mu},\mathbf{n}^{\leq k}}\wedge q_{\geq \mu}^\ast\beta_{I_k^+,\underline{\lambda}^{\geq\mu},\mathbf{n}^{>k}}\\
\nonumber&+\sum_{\mu=0}^m(-1)^{i-\mu-\lvert I_k^-\rvert}\epsilon(i,j,k,\ell,m,\mu,I)q_{\leq\mu}^\ast\alpha_{I_k^-,\underline{\lambda}^{\leq\mu},\mathbf{n}^{\leq k}}\wedge q_{\geq \mu}^\ast\nabla\beta_{I_k^+,\underline{\lambda}^{\geq\mu},\mathbf{n}^{>k}}\\
\nonumber&+\sum_{\mu=0}^m\sum_{\nu=0}^{\mu-1}(-1)^{i+j-m+\nu-\lvert I\rvert+1}\epsilon(i,j,k,\ell,m-1,\mu-1,I)p_\nu^\ast q_{\leq\mu}^\ast\alpha_{I_k^-,\underline{\lambda}^{\leq\mu}(\nu),\mathbf{n}^{\leq k}}\wedge q_{\geq \mu}^\ast\beta_{I_k^+,\underline{\lambda}^{\geq\mu},\mathbf{n}^{>k}}\\
\nonumber&+\sum_{\mu=0}^m\sum_{\nu=1}^{m-\mu}(-1)^{i+j-m+\mu+\nu-\lvert I\rvert+1}\epsilon(i,j,k,\ell,m-1,\mu,I)q_{\leq\mu}^\ast\alpha_{I_k^-,\underline{\lambda}^{\leq\mu},\mathbf{n}^{\leq k}}\wedge p_\nu^\ast q_{\geq \mu}^\ast\beta_{I_k^+,\underline{\lambda}^{\geq\mu}(\nu),\mathbf{n}^{>k}}\\
\nonumber&+\sum_{\mu=0}^m\sum_{r\in I_k^-}(-1)^{i+j+\lvert I_r^+\rvert}\epsilon(i,j,k,\ell,m,\mu,I\setminus\{r\})
\times q_{\leq\mu}^\ast\left(\alpha_{I_k^-\setminus\{r\},\underline{\lambda}^{\leq\mu},\mathbf{n}^{\leq k}}-\alpha_{I_k^-\setminus\{r\},\underline{\lambda}^{\leq\mu},\mathbf{n}^{\leq k}+1_r}\right)\wedge q_{\geq\mu}^\ast\beta_{I_k^+,\underline{\lambda}^{\geq\mu},\mathbf{n}^{>k}}\\
\nonumber&+\sum_{\mu=0}^m\sum_{r\in I_k^+}(-1)^{i+j+\lvert I_r^+\rvert}\epsilon(i,j,k,\ell,m,\mu,I\setminus\{k+r\}) 
\times q_{\leq\mu}^\ast\alpha_{I_k^-,\underline{\lambda}^{\leq\mu},\mathbf{n}^{\leq k}}\wedge q_{\geq\mu}^\ast\left(\beta_{I_k^+\setminus\{r\},\underline{\lambda}^{\geq\mu},\mathbf{n}^{>k}}-\beta_{I_k^+\setminus\{r\},\underline{\lambda}^{\geq\mu},\mathbf{n}^{>k}+1_r}\right),
\end{align}
\end{small}

\begin{small}
\begin{align}
\label{eq: D cup calculation 2}&\hspace{-8pt}(D\alpha\cup\beta+(-1)^i\alpha\cup D\beta)_{I,\underline{\lambda},\mathbf{n}}=\\
\nonumber&\sum_{\mu=0}^m\epsilon(i+1,j,k,\ell,m,\mu,I)q_{\leq\mu}^\ast\nabla\alpha_{I_k^-,\underline{\lambda}^{\leq\mu},\mathbf{n}^{\leq k}}\wedge q_{\geq \mu}^\ast\beta_{I_k^+,\underline{\lambda}^{\geq\mu},\mathbf{n}^{>k}}\\
\nonumber&+\sum_{\mu=0}^m(-1)^i\epsilon(i,j+1,k,\ell,m,\mu,I)q_{\leq\mu}^\ast\alpha_{I_k^-,\underline{\lambda}^{\leq\mu},\mathbf{n}^{\leq k}}\wedge q_{\geq \mu}^\ast\nabla\beta_{I_k^+,\underline{\lambda}^{\geq\mu},\mathbf{n}^{>k}}\\
\nonumber&+\sum_{\mu=0}^m\sum_{\nu=0}^{\mu-1}(-1)^{i-\mu+\nu-\lvert I_k^-\rvert+1}\epsilon(i+1,j,k,\ell,m,\mu,I)p_\nu^\ast q_{\leq\mu}^\ast\alpha_{I_k^-,\underline{\lambda}^{\leq\mu}(\nu),\mathbf{n}^{\leq k}}\wedge q_{\geq \mu}^\ast\beta_{I_k^+,\underline{\lambda}^{\geq\mu},\mathbf{n}^{>k}}\\
\nonumber&+\sum_{\mu=0}^m\sum_{\nu=1}^{m-\mu}(-1)^{i+j-m+\mu+\nu-\lvert I_k^+\rvert+1}\epsilon(i,j+1,k,\ell,m,\mu,I)q_{\leq\mu}^\ast\alpha_{I_k^-,\underline{\lambda}^{\leq\mu},\mathbf{n}^{\leq k}}\wedge p_\nu^\ast q_{\geq \mu}^\ast\beta_{I_k^+,\underline{\lambda}^{\geq\mu}(\nu),\mathbf{n}^{>k}}\\
\nonumber&+\sum_{\mu=0}^m\sum_{r\in I_k^-}(-1)^{i+\lvert I_k^-\rvert-\lvert I_r^-\rvert}\epsilon(i+1,j,k,\ell,m,\mu,I\setminus\{r\})\\
\nonumber&\hspace{60pt}
\times q_{\leq\mu}^\ast\left(\alpha_{I_k^-\setminus\{r\},\underline{\lambda}^{\leq\mu},\mathbf{n}^{\leq k}}-\alpha_{I_k^-\setminus\{r\},\underline{\lambda}^{\leq\mu},\mathbf{n}^{\leq k}+1_r}\right)\wedge q_{\geq\mu}^\ast\beta_{I_k^+,\underline{\lambda}^{\geq\mu},\mathbf{n}^{>k}}\\
\nonumber&+\sum_{\mu=0}^m\sum_{r\in I_k^+}(-1)^{i+j+\lvert I_r^+\rvert}\epsilon(i,j+1,k,\ell,m,\mu,I\setminus\{k+r\})\\
\nonumber&\hspace{60pt}
\times q_{\leq\mu}^\ast\alpha_{I_k^-,\underline{\lambda}^{\leq\mu},\mathbf{n}^{\leq k}}\wedge q_{\geq\mu}^\ast\left(\beta_{I_k^+\setminus\{r\},\underline{\lambda}^{\geq\mu},\mathbf{n}^{>k}}-\beta_{I_k^+\setminus\{r\},\underline{\lambda}^{\geq\mu},\mathbf{n}^{>k}+1_r}\right)\\
\nonumber&+\sum_{\mu=1}^m\left((-1)^{i-\lvert I_k^-\rvert+1}\epsilon(i+1,j,k,\ell,m,\mu,I)+(-1)^{i+j-m+\mu-\lvert I_k^+\rvert}\epsilon(i,j+1,k,\ell,m,\mu-1,I)\right)\\
\nonumber&\hspace{60pt}
\times p_\mu^\ast q_{\leq\mu}^\ast\alpha_{I_k^-,\underline{\lambda}_{\leq\mu}(\mu),\mathbf{n}^{\leq k}}\wedge q_{\geq\mu}^\ast\beta_{I_k^+,\underline{\lambda}^{\geq\mu},\mathbf{n}^{>k}}.
\end{align}
\end{small}

By the definition of $\epsilon(i,j,k,\ell,m,\mu,I)$, we see that they are equal to each other.
Thus we have
\begin{equation}\label{eq: D cup}
D(\alpha\cup\beta)=D\alpha\cup\beta+(-1)^i\alpha\cup D\beta,
\end{equation}
in other words, $\cup$ defines a morphism of complexes $M_k^\bullet(Y,\sE)\otimes M_\ell^\bullet(Y,\sE')\rightarrow M_{k+\ell}^\bullet(Y,\sE\otimes\sE')$.

For order preserving injections $\iota\colon\{1,\ldots,k\}\hookrightarrow\{1,\ldots,k'\}$ and $\rho\colon\{1,\ldots,\ell\}\hookrightarrow\{1,\ldots,\ell'\}$, define a map $\iota\ast\rho\colon \{1,\ldots,k+\ell\}\hookrightarrow\{1,\ldots,k'+\ell'\}$ by putting $\iota\ast\rho(r):=\iota(r)$ for $r\leq k$ and $\iota\ast\rho(r):=k'+\rho(r-k)$ for $r>k$.
Then we have
\begin{align*}
&(f_{\iota\ast\rho}(\alpha\cup\beta))_{I,\underline{\lambda},\mathbf{n}}\\
&=\begin{cases}
\displaystyle\sum_{\mu=0}^m\epsilon(i,j,k,\ell,m,\mu,(\iota\ast\rho)^{-1}(I))\alpha_{\iota^{-1}(I_{k'}^-),\underline{\lambda}^{\leq\mu},\iota^\ast\mathbf{n}^{\leq k'}}\wedge\beta_{\rho^{-1}(I_{k'}^+),\underline{\lambda}^{\geq\mu},\rho^\ast\mathbf{n}^{>k'}}&\text{if $I\subset(\iota\ast\rho)(\{1,\ldots,k+\ell\})$},\\
0&\text{otherwise},
\end{cases}
\end{align*}
\begin{align*}
(&f_\iota(\alpha)\cup f_\rho(\beta))_{I,\underline{\lambda},\mathbf{n}}\\
&=\begin{cases}
\displaystyle\sum_{\mu=0}^m\epsilon(i,j,k',\ell',m,\mu,I)\alpha_{\iota^{-1}(I_{k'}^-),\underline{\lambda}^{\leq\mu},\iota^\ast\mathbf{n}^{\leq k'}}\wedge\beta_{\rho^{-1}(I_{k'}^+),\underline{\lambda}^{\geq\mu},\rho^\ast\mathbf{n}^{>k'}}&\text{if $I\subset(\iota\ast\rho)(\{1,\ldots,k+\ell\})$},\\
0&\text{otherwise}.
\end{cases}
\end{align*}
One can see that they are equal to each other again by the definition of $\epsilon(i,j,k,\ell,m,\mu,I)$, hence we have
\begin{equation}\label{eq: iota cup}
f_{\iota\ast\rho}(\alpha\cup\beta)=f_\iota(\alpha)\cup f_\rho(\beta).
\end{equation}
Consequently, according to Proposition \ref{prop: M iso R} $\cup$ is well-defined as a morphism
\begin{align*}
\cup\colon \cR(Y,\sE)\otimes^L \cR(Y,\sE')\rightarrow \cR(Y,\sE\otimes\sE'),
\end{align*}
which induces homomorphisms
\[\cup\colon \cH^i(Y,\sE)\otimes \cH^j(Y,\sE')\rightarrow \cH^{i+j}(Y,\sE\otimes\sE')\]
where $\cH^i(Y,\sE)$ denotes the $i$-th cohomology group of $\cR(Y,\sE)$. 

In what follows, we show properties of the cup product, such as (anti) commutativity, associativity and compatibility with base change.

\begin{proposition}\label{prop: comm cup}
	There is a commutative diagram
	\[\xymatrix{
	\cR(Y,\sE)\otimes^L \cR(Y,\sE')\ar[d]_-\theta^-\cong\ar[r]^-\cup&
	\cR(Y,\sE\otimes\sE'),\\
	\cR(Y,\sE')\otimes^L \cR(Y,\sE)\ar[ur]_-\cup&
	}\]
	where $\theta$ maps $\alpha\otimes\beta\mapsto(-1)^{ij}\beta\otimes\alpha$ if $\alpha$ and $\beta$ are cochains of $\cR(Y,\sE)$ and $\cR(Y,\sE')$ in degree $i$ and $j$, respectively.
	This means that we have $[\beta]\cup[\alpha]=(-1)^{ij}[\alpha]\cup[\beta]$ on cohomology groups, if $\alpha$ and $\beta$ are cocycles.
\end{proposition}

\begin{proof}
	One can calculate that
	\begin{align*}
	(\tau(\tau\beta\cup\tau\alpha))_{I,\underline{\lambda},\mathbf{n}}
	=\sum_{\mu=0}^m(-1)^{\mu(m-\mu)}\epsilon(j,i,\ell,k,m,m-\mu,I)q_{\geq\mu}^\ast\beta_{I_\ell^-,\underline{\lambda}^{\geq\mu},\mathbf{n}^{\leq\ell}}\wedge q_{\leq\mu}^\ast\alpha_{I_\ell^+,\underline{\lambda}^{\leq\mu},\mathbf{n}^{>\ell}}.
	\end{align*}
	Let $\sigma$ be the permutation of $\{1,\ldots,k+\ell\}$ defined by $\sigma(r)=r+k$ for $r\leq \ell$ and $\sigma(r)=r-\ell$ for $r>\ell$.
	Then we have $\delta(\sigma,I)=\lvert I_k^-\rvert\lvert I_k^+\rvert$ and
	\begin{align*}
	(\sigma\tau(\tau\beta\cup\tau\alpha))_{I,\underline{\lambda},\mathbf{n}}
	&=\sum_{\mu=0}^m(-1)^{\lvert I_k^-\rvert (j-m+\mu)+\lvert I_k^+\rvert(i-\mu)+i(m-\mu)+j\mu+ij}\epsilon(j,i,\ell,k,m,m-\mu,\sigma^{-1}(I))\\
	&\hspace{35pt}\times q_{\leq\mu}^\ast\alpha_{I_k^-,\underline{\lambda}^{\leq\mu},\mathbf{n}^{\leq k}}\wedge q_{\geq\mu}^\ast\beta_{I_k^+,\underline{\lambda}^{\geq\mu},\mathbf{n}^{>k}}\\
	&=\sum_{\mu=0}^m(-1)^{ij}\epsilon(i,j,k,\ell,m,\mu,I) q_{\leq\mu}^\ast\alpha_{I_k^-,\underline{\lambda}^{\leq\mu},\mathbf{n}^{\leq k}}\wedge q_{\geq\mu}^\ast\beta_{I_k^+,\underline{\lambda}^{\geq\mu},\mathbf{n}^{>k}}\\
	&=((-1)^{ij}\alpha\cup\beta)_{I,\underline{\lambda},\mathbf{n}}.
	\end{align*}
	Thus we have $\sigma\tau(\tau\beta\cup\tau\alpha)=(-1)^{ij}\alpha\cup\beta$, and this shows the assertion by Lemma \ref{lem: sigma tau}.
\end{proof}

\begin{proposition}\label{prop: ass cup}
	There is a commutative diagram
	\[\xymatrix{
	\cR(Y,\sE)\otimes^L \cR(Y,\sE')\otimes^L \cR(Y,\sE'')\ar[r]^-{\cup\otimes\id}\ar[d]_-{\id\otimes\cup} &\cR(Y,\sE\otimes\sE')\otimes^L \cR(Y,\sE'')\ar[d]^-\cup\\
	\cR(Y,\sE)\otimes^L \cR(Y,\sE'\otimes\sE'')\ar[r]^-\cup& \cR(Y,\sE\otimes\sE'\otimes\sE'').
	}\]
\end{proposition}

\begin{proof}
	Let $i,j,g\geq 0$, $k,\ell,h\geq 1$ and $\alpha\in M_k^i(Y,\sE)$, $\beta\in M_\ell^j(Y,\sE')$, and $\gamma\in M_h^g(Y,\sE'')$.
	Then one can calculate that
	\begin{align*}
	&((\alpha\cup\beta)\cup\gamma)_{I,\underline{\lambda},\mathbf{n}}\\
	&=\sum_{\mu=0}^m\sum_{\eta=0}^\mu\epsilon(i+j,g,k+\ell,h,m,\mu,I)\epsilon(i,j,k,\ell,\mu,\eta,I_{k+\ell}^-)\\
	&\hspace{30pt}\times\alpha_{I_k^-,\underline{\lambda}^{\leq\eta},\mathbf{n}^{\leq k}}\wedge\beta_{(I_{k+\ell}^-)_k^+,(\underline{\lambda}^{\leq\mu})^{\geq\eta},(\mathbf{n}^{\leq k+\ell})^{>k}}\wedge\gamma_{I_{k+\ell}^+,\underline{\lambda}^{\geq\mu},\mathbf{n}^{>k+\ell}},\\
	&(\alpha\cup(\beta\cup\gamma))_{I,\underline{\lambda},\mathbf{n}}\\
	&=\sum_{\mu=0}^m\sum_{\eta=0}^\mu\epsilon(i,j+g,k,\ell+h,m,\eta,I)\epsilon(j,g,\ell,h,m-\eta,\mu-\eta,I_k^+)\\
	&\hspace{30pt}\times\alpha_{I_k^-,\underline{\lambda}^{\leq\eta},\mathbf{n}^{\leq k}}\wedge\beta_{(I_{k+\ell}^-)_k^+,(\underline{\lambda}^{\leq\mu})^{\geq\eta},(\mathbf{n}^{\leq k+\ell})^{>k}}\wedge\gamma_{I_{k+\ell}^+,\underline{\lambda}^{\geq\mu},\mathbf{n}^{>k+\ell}}.
	\end{align*}
	By definition of $\epsilon(i,j,k,\ell,m,\mu,I)$ we may conclude that $(\alpha\cup\beta)\cup\gamma=\alpha\cup(\beta\cup\gamma)$ in the level of cochains of $M_{k+\ell+h}^\bullet(Y,\sE\otimes\sE'\otimes\sE'')$.
\end{proof}

\begin{proposition}\label{prop: cup functorial}
	Let $k'$ be a finite extension of $k$, $Y'$ a fine log scheme over $k'^0$, and $f\colon Y'\rightarrow Y$ a morphism over $k^0$.
	Then for $\sE,\sE'\in\cI(Y)$, we have a commutative diagram
	\[\xymatrix{
	\cR(Y,\sE)\otimes^L \cR(Y,\sE')\ar[r]^-{\cup}\ar[d]_-{f^\ast\otimes f^\ast}
	&\cR(Y,\sE\otimes\sE')\ar[d]^-{f^\ast}\\
	\cR(Y',f^\ast\sE)\otimes^L \cR(Y',f^\ast\sE')\ar[r]^-{\cup}
	&\cR(Y',f^\ast\sE\otimes f^\ast\sE').}\]
\end{proposition}

\begin{proof}
	We only prove the case that $\cR(Y,\sE)=R\Gamma^\rig_\HK(Y,\sE)$, since the other cases can be proved similarly.
	We define a widening $(k'^0,\cS',\tau')$ similarly to $(k^0,\cS,\tau)$.
	Let $(Z_\lambda,\cZ_\lambda,i_\lambda,h_\lambda,\theta_\lambda)_{\lambda\in\Lambda}$ and $(Z'_\xi,\cZ'_\xi,i'_\xi,h'_\xi,\theta'_\xi)_{\xi\in\Xi}$ be local embedding data for $Y$ and $Y'$ over $(k^0,\cS,\tau)$ and $(k'^0,\cS',\tau')$, respectively.
	For $m_1,m_2\in\bbN$, $\underline{\lambda}\in\Lambda^{m_1+1}$ and $\underline{\xi}\in\Xi^{m_2+1}$, let $Z'_{\underline{\lambda},\underline{\xi}}:=Z_{\underline{\lambda}}\times_YZ'_{\underline{\xi}}$ and let $i'_{\underline{\lambda},\underline{\xi}}\colon Z'_{\underline{\lambda},\underline{\xi}}\hookrightarrow \cZ'_{\underline{\lambda},\underline{\xi}}$ be the exactification of $(i_{\underline{\lambda}},i'_{\underline{\xi}})\colon Z'_{\underline{\lambda},\underline{\xi}}\hookrightarrow\cZ_{\underline{\lambda}}\times_\cS\cZ'_{\underline{\xi}}$.
	Then together with the natural morphisms $h_{\underline{\lambda},\underline{\xi}}\colon \cZ'_{\underline{\lambda},\underline{\xi}}\rightarrow\cS'$ and $\theta_{\underline{\lambda},\underline{\xi}}\colon Z'_{\underline{\lambda},\underline{\xi}}\rightarrow Y'$ they form an object $(Z_{\underline{\lambda},\underline{\xi}},\cZ_{\underline{\lambda},\underline{\xi}},i_{\underline{\lambda},\underline{\xi}},h_{\underline{\lambda},\underline{\xi}},\theta_{\underline{\lambda},\underline{\xi}})$ in $\OC(Y/\cS')$.
Moreover
\[(Z_{m_1,m_2},\cZ_{m_1,m_2},i_{m_1,m_2},h_{m_1,m_2},\theta_{m_1,m_2}):=\coprod_{\underline{\lambda}\in\Lambda^{m_1+1}}\coprod_{\underline{\xi}\in\Xi^{m_2+1}}(Z_{\underline{\lambda},\underline{\xi}},\cZ_{\underline{\lambda},\underline{\xi}},i_{\underline{\lambda},\underline{\xi}},h_{\underline{\lambda},\underline{\xi}},\theta_{\underline{\lambda},\underline{\xi}})\]
forms a bisimplicial object, and the morphism $R\Gamma^\rig_\HK(Y,\sE)\rightarrow R\Gamma^\rig_\HK(Y',f^\ast\sE)$ is defined to be the composition
\begin{equation}\label{eq: bisimplicial holim}
\xymatrix{
\holim_{m_1}\hocolim_\ell R\Gamma(\cZ_{m_1,\bbQ},\sE_{\cZ_{m_1}}\otimes \omega^\star_{\cZ_{m_1}/W^\varnothing,\bbQ}[u]_\ell)\ar[d]\\
\holim_{m_1,m_2}\hocolim_\ell R\Gamma(\cZ'_{m_1,m_2,\bbQ},(f^\ast\sE)_{\cZ'_{m_1,m_2}}\otimes \omega^\star_{\cZ'_{m_1,m_2}/W^\varnothing,\bbQ}[u]_\ell)\\
\holim_{m_2}\hocolim_\ell R\Gamma(\cZ'_{m_2,\bbQ},(f^\ast\sE)_{\cZ'_{m_2}}\otimes \omega^\star_{\cZ'_{m_2}/W^\varnothing,\bbQ}[u]_\ell).\ar[u]^-\cong
}\end{equation}
For $k\geq 1$, we define a complex $L_{\HK,k}^\bullet(Y',f^\ast\sE)$ by
\[L_{\HK,k}^i(Y',f^\ast\sE):=\bigoplus_{I\subset\{1,\ldots,k\}}\bigoplus_{(m_1,m_2)\in\bbN^2}\prod_{\underline{\lambda}\in\Lambda^{m_1+1}}\prod_{\underline{\xi}\in\Xi^{m_2+1}}\prod_{\mathbf{n}\in\bbN^k}\Gamma(\cZ_{\underline{\lambda},\underline{\xi}}[\mathbf{n}]_\bbQ,(f^\ast\sE)_{\cZ_{\underline{\lambda},\underline{\xi}}[\mathbf{n}]}\otimes\omega^{i-m_1-m_2-\lvert I\rvert}_{\cZ_{\underline{\lambda},\underline{\xi}}[\mathbf{n}]/W'^\varnothing,\bbQ}[u])\]
with differentials $D$ given by
\begin{align*}
(D\alpha)_{I,\underline{\lambda},\underline{\xi},\mathbf{n}}
:=&\nabla\alpha_{I,\underline{\lambda},\underline{\xi},\mathbf{n}}
+\sum_{\mu_1}^{m_1}(-1)^{i-m_1-m_2+\mu_1-\lvert I\rvert+1}\alpha_{I,\underline{\lambda}(\mu_1),\underline{\xi},\mathbf{n}}
+\sum_{\mu_2}^{m_2}(-1)^{i-m_2+\mu_2-\lvert I\rvert+1}\alpha_{I,\underline{\lambda},\underline{\xi}(\mu_2),\mathbf{n}}\\
&+\sum_{r\in I}(-1)^{i+\lvert I_r^+\rvert}(\alpha_{I,\underline{\lambda},\underline{\xi},\mathbf{n}}-\alpha_{I,\underline{\lambda},\underline{\xi},\mathbf{n}+1_r}|_{\cZ_{\underline{\lambda},\underline{\xi}}[\mathbf{n}]_\bbQ})
\end{align*}
for $\alpha\in L_{\HK,k}^i(Y',f^\ast\sE)$.
Similarly to the case for $M_{\HK,k}^\bullet(Y,\sE)$, one can confirm that $L_{\HK,k}^\bullet(Y',f^\ast\sE)$ is in fact a complex and represents $\holim_{m_1,m_2}\hocolim_\ell R\Gamma(\cZ'_{m_1,m_2,\bbQ},(f^\ast\sE)_{\cZ'_{m_1,m_2}}\otimes \omega^\star_{\cZ'_{m_1,m_2}/W^\varnothing,\bbQ}[u]_\ell)$.
Then the diagram \eqref{eq: bisimplicial holim} is represented by the diagram
\[M_{\HK,k}^\bullet(Y,\sE)\xrightarrow{\psi} L_{\HK,k}^\bullet(Y',f^\ast\sE)\xleftarrow[\cong]{\psi'}M_{\HK,k}^\bullet(Y',f^\ast\sE).\]
Here $\psi$ and $\psi'$ are defined by
\begin{align*}
	(\psi\alpha)_{I,\underline{\lambda},\underline{\xi},\mathbf{n}}:=\begin{cases}
	\mathrm{pr}_1^\ast\alpha_{I,\underline{\lambda},\mathbf{n}}&\text{if $m_2=0$}\\
	0&\text{otherwise,}
	\end{cases}
	&&(\psi'\beta)_{I,\underline{\lambda},\underline{\xi},\mathbf{n}}:=\begin{cases}
	\mathrm{pr}_2^\ast\beta_{I,\underline{\xi},\mathbf{n}}&\text{if $m_1=0$}\\
	0&\text{otherwise,}
	\end{cases}
\end{align*}
	where $\mathrm{pr}_1\colon\cZ_{\underline{\lambda},\underline{\xi}}[\mathbf{n}]_\bbQ\rightarrow\cZ_{\underline{\lambda}}[\mathbf{n}]_\bbQ$ and $\mathrm{pr}_2^\ast\colon\cZ_{\underline{\lambda},\underline{\xi}}[\mathbf{n}]_\bbQ\rightarrow \cZ_{\underline{\xi}}[\mathbf{n}]_\bbQ$ are the canonical projections.
	
	To prove the assertion, it suffices to construct a morphism
	\[\cup\colon L_{\HK,k}^\bullet(Y',f^\ast\sE)\otimes L_{\HK,\ell}^\bullet(Y',f^\ast\sE')\rightarrow L_{\HK,k+\ell}^\bullet(Y',f^\ast(\sE\otimes\sE'))\]
	which is compatible with $\cup$ on $M_{\HK,k}^\bullet$ via $\psi$ and $\psi'$, respectively.
	This is in fact possible by setting
	\begin{align*}
	(\alpha\cup\beta)_{I,\underline{\lambda},\underline{\xi},\mathbf{n}}
	:=\sum_{\mu_1=0}^{m_1}\sum_{\mu_2=0}^{m_2}\wt\delta(i,j,k,\ell,m_1,\mu_1,m_2,\mu_2,I)\alpha_{I_k^-,\underline{\lambda}^{\leq\mu_1},\underline{\xi}^{\leq\mu_2},\mathbf{n}^{\leq k}}\wedge\beta_{I_k^+,\underline{\lambda}^{\geq\mu_1},\underline{\xi}^{\geq\mu_2},\mathbf{n}^{>k}}
	\end{align*}
	with
	\[\wt\delta(i,j,k,\ell,m_1,\mu_1,m_2,\mu_2,I):=(-1)^{(\lvert I_k^-\rvert+\mu_1+\mu_2)(\lvert I_k^+\rvert +j)+(\mu_1+m_2+1)(\mu_2+m_1+m_2+1)+(m_1+1)(m_2+1)}\]
	for $\alpha\in L_{\HK,k}^i(Y',f^\ast\sE)$ and $\beta\in L_{\HK,\ell}^j(Y',f^\ast\sE')$.
	That $\cup$ commutes with differentials can be confirmed by similar calculations to the case for $M_{\HK,k}^\bullet$.
\end{proof}

\begin{remark}\label{rem: uniqueness of epsilon}
	The definition \eqref{eq: sign} of $\epsilon(i,j,k,\ell,m,\mu,I)$ is the unique choice so that the definition \eqref{eq: def cup} satisfies the following properties:
	\begin{itemize}
	\item The equalities \eqref{eq: D cup} and \eqref{eq: iota cup} holds,
	\item The cup product \eqref{eq: def cup} is functorial as stated in Proposition \ref{prop: cup functorial},
	\item Suppose that $Y$ is embeddable, i.e.\, there exists a homeomorphic exact closed immersion $Y\hookrightarrow\cZ$ into a strongly log smooth weak formal log scheme over $\cT=W^\varnothing$, $W^0$, $V^\sharp$ for the case of log rigid cohomology or over $\cS$ for the case of Hyodo--Kato cohomology.
		Then for any $n\in\bbN$ there is a canonical morphism
		\begin{align}
		\label{eq: embeddable cup}&R\Gamma_\rig(Y/\cT,\sE)\rightarrow \Gamma(\cZ[n]_\bbQ,\sE_{\cZ[n]}\otimes\omega^\bullet_{\cZ/\cT,\bbQ}),\ 
		\text{or}\\
		\nonumber&R\Gamma^\rig_\HK(Y,\sE)\rightarrow \Gamma(\cZ[n]_\bbQ,\sE_{\cZ[n]}\otimes\omega^\bullet_{\cZ/W^\varnothing,\bbQ}).
		\end{align}
		The cup product on the left hand side defined by \eqref{eq: def cup} commutes with the usual product on the right hand side induced by the tensor product of coefficient sheaves and the wedge product of log differential forms.
	\end{itemize}
	
	Indeed, the equation \eqref{eq: D cup} determines the dependence of $\epsilon(i,j,k,\ell,m,\mu,I)$ on $i$, $j$, $k$, $m$, $\mu$, and $I$.
	Namely, by comparing coefficients of \eqref{eq: D cup calculation 1} and \eqref{eq: D cup calculation 2} one deduces 
	\begin{align*}
	&\epsilon(i+1,j,k,\ell,m,\mu,I)=\epsilon(i,j,k,\ell,m,\mu,I),\\
	&\epsilon(i,j+1,k,\ell,m,\mu,I)=(-1)^{\mu+\lvert I_k^-\rvert}\epsilon(i,j,k,\ell,m,\mu,I),\ \text{etc.}
	\end{align*}
	
	Moreover the equation \eqref{eq: iota cup} for the case $\iota=\id$ and $I=\emptyset$ gives the relation
	\[\epsilon(i,j,k,\ell,m,\mu,\emptyset)=\epsilon(i,j,k,\ell',m,\mu,\emptyset).\]
	Thus $\epsilon(i,j,k,\ell,m,\mu,I)$	for every $i$, $j$, $k$, $\ell$, $m$, $\mu$, and $I$ is determined from $\epsilon(0,0,1,1,0,0,\emptyset)$.
	Finally, the value of $\epsilon(0,0,1,1,0,0,\emptyset)$ is determined by the functoriality and the compatibility with \eqref{eq: embeddable cup}.
\end{remark}

\begin{remark}\label{rem: monodromy cup}
The cup product on Hyodo--Kato cohomology commutes with the monodromy operator.
Indeed, consider weak formal log schemes $\cZ_{\underline{\lambda}}[\mathbf{n}]$ as before.
For any $\omega\in\Gamma(\cZ_{\underline{\lambda}}[\mathbf{n}]_\bbQ,\omega^i_{\cZ_{\underline{\lambda}}[\mathbf{n}]/W^\varnothing,\bbQ})$, $\eta\in\Gamma(\cZ_{\underline{\lambda}}[\mathbf{n}]_\bbQ,\omega^j_{\cZ_{\underline{\lambda}}[\mathbf{n}]/W^\varnothing,\bbQ})$, and integers $r,s\in\bbN$, we have
\begin{align*}
N(\omega u^{[r]})\wedge\eta u^{[s]}+\omega u^{[r]}\wedge N(\eta u^{[s]})
&=\omega u^{[r-1]}\wedge \eta u^{[s]}+\omega u^{[r]}\wedge \eta u^{[s-1]}\\
&=\left(\frac{(r+s-1)!}{(r-1)!s!}+\frac{(r+s-1)!}{r!(s-1)!}\right)\omega\wedge\eta u^{[r+s-1]}\\
&=\frac{(r+s)!}{r!s!}\omega^\wedge\eta u^{[r+s-1]}\\
&=N\left(\frac{(r+s)!}{r!s!}\omega\wedge\eta u^{[r+s]}\right)\\
&=N(\omega u^{[r]}\wedge \eta u^{[s]}).
\end{align*}
This with the definition of the cup product \eqref{eq: def cup} shows that $\cup$ gives a morphism in $D^+(\Mod_F(N))$.
\end{remark}

\begin{remark}\label{rem: Frobenius cup}
For the log rigid cohomology over $W^\varnothing$ or over $W^0$, or the Hyodo--Kato cohomology, we may consider Frobenius structures on isocrystals.
In those cases, the cup product commutes with the Frobenius operators.
Indeed, the Frobenius operator on $R\Gamma_\rig(Y/\cT,(\sE,\Phi))$ for $\cT=W^\varnothing$ or $W^0$ is defined as the composition
\[R\Gamma_\rig(Y/\cT,\sE)\xrightarrow{F_Y^\ast}R\Gamma_\rig(Y/\cT,F_Y^\ast\sE)\xrightarrow{\Phi}R\Gamma_\rig(Y/\cT,\sE),\]
where $F_Y$ denotes the absolute Frobenius on $Y$. 
These two morphisms commute with the cup product by Proposition \ref{prop: cup functorial} and by definition \eqref{eq: def cup}, respectively.
A similar statement holds for the cup product on Hyodo--Kato cohomology.
In summary, we have morphisms
\begin{align*}
&\cup\colon R\Gamma_\rig(Y/\cT,(\sE,\Phi))\otimes^L R\Gamma_\rig(Y,(\sE',\Phi'))\rightarrow R\Gamma_\rig(Y/\cT,(\sE,\Phi)\otimes(\sE',\Phi'))&&\text{in $D^+(\Mod_F(\varphi))$},\\
&\cup\colon R\Gamma_\HK^\rig(Y,(\sE,\Phi))\otimes^L R\Gamma_\HK^\rig(Y,(\sE',\Phi'))\rightarrow R\Gamma_\HK^\rig(Y,(\sE,\Phi)\otimes(\sE',\Phi'))&&\text{in $D^+(\Mod_F(\varphi,N))$}.
\end{align*}
\end{remark}

\begin{remark}\label{rem: HK map cup}
The cup product commutes with the morphisms displayed in the following diagram
\[\xymatrix{
&R\Gamma_\rig(Y/W^\varnothing,\sE)\ar[ld]\ar[d]\ar[rd]&\\
	R\Gamma_\rig(Y/W^0,\sE)&R\Gamma_\HK^\rig(Y,\sE)\ar[l]\ar[r]^-{\Psi_{\pi,\log}}&R\Gamma_\rig(Y/V^\sharp,\sE)_\pi.	
}\]
Indeed, for $\omega\in\Gamma(\cZ_{\underline{\lambda}}[\mathbf{n}]_\bbQ,\omega^i_{\cZ_{\underline{\lambda}}[\mathbf{n}]/W^\varnothing,\bbQ})$ and $\eta\in\Gamma(\cZ_{\underline{\lambda}}[\mathbf{n}]_\bbQ,\omega^j_{\cZ_{\underline{\lambda}}[\mathbf{n}]/W^\varnothing,\bbQ})$, we denote by $\overline{\omega}$ and $\overline{\eta}$ the images of them on $\cX_{\underline{\lambda}}[\mathbf{n}]_\bbQ$ where $\cX_{\underline{\lambda}}:=\cZ_{\underline{\lambda}}\times_{\cS,j_\pi} V^\sharp$.
Then we have
\begin{align*}
\Psi_{\pi,\log}(\omega u^{[r]})\wedge\Psi_{\pi,\log}(\eta u^{[s]})
&=\frac{(-\log\pi)^r}{r!}\overline{\omega}\wedge \frac{(-\log \pi)^s}{s!}\overline{\eta}
=\frac{(-\log \pi)^{r+s}}{r!s!}\overline{\omega}\wedge\overline{\eta}\\
&=\frac{(r+s)!}{r!s!}\frac{(-\log\pi)^{r+s}}{(r+s)!}\overline{\omega}\wedge\overline{\eta}
=\Psi_{\pi,\log}\left(\frac{(-\log \pi)^{r+s}}{r!s!}\omega\wedge\eta u^{[r+s]}\right)\\
&=\Psi_{\pi,\log}(\omega u^{[r]}\wedge \eta u^{[s]}).
\end{align*}
This shows that the cup product commutes with the Hyodo--Kato map $\Psi_{\pi,\log}$ for any choice of uniformiser $\pi$ and branch of the $p$-adic logarithm $\log$.
Commutativity with the other morphisms is clear.
\end{remark}


\begin{thebibliography}{1}
\bibitem{BaldassariCailottoFiorot2004} 
\textsc{F.\,Baldassari, M.\,Cailotto, L.\,Fiorot}: 
\textit{Poincar\'e duality for algebraic de~Rham cohomology}. 
Manuscr.\,Math., vol.\,114, no.\,1, pp.\,61--116, (2004), \href{https://doi.org/10.1007/s00229-004-0448-y}{DOI: 10.1007/s00229-004-0448-y}.

\bibitem{Bannai}
\textsc{K.\,Bannai}:
\textit{Syntomic cohomology as a $p$-adic absolute Hodge cohomology}.
Math.\,Z.\,vol.\,242, pp.\,443--480, (2002), \href{https://doi.org/10.1007/s002090100351}{DOI: 10.1007/s002090100351}.

\bibitem{Beilinson1987}
\textsc{A.\,Beilinson}:
\textit{On the derived category of perverse sheaves}.
In: Yu.\,Manin, (eds) \textit{$K$-Theory, Arithmetic and Geometry}. Lect.\,Notes Math., vol.\,1289, Springer-Verlag Berlin Heidelberg, pp.\,27--41, (1987), \href{https://doi.org/10.1007/BFb0078365}{DOI: 10.1007/BFb0078365}.

\bibitem{Beilinson2013}
\textsc{A.\,Beilinson}:
\textit{On crystalline period map}.
Cambridge J.\,of Math., vol.\,1, no.\,1, pp.\,1--51, (2013), \href{https://dx.doi.org/10.4310/CJM.2013.v1.n1.a1}{DOI: 10.4310/CJM.2013.v1.n1.a1}.

\bibitem{Berthelot1974}
\textsc{P.\,Berthelot}:
\textit{Cohomologie cristalline des sch\'emas de caract\'eristique $p > 0$}.
Lect.\,Notes Math., vol.\,407, Springer, New York, (1974), \href{https://doi.org/10.1007/BFb0068636}{DOI: 10.1007/BFb0068636}.

\bibitem{Berthelot1986}
\textsc{P.\,Berthelot}:
\textit{G\'eom\'etrie rigide et cohomologie des vari\'et\'es algebriques de car-
act\'eristique $p$}.
In: \textit{Introductions aux cohomologies $p$-adiques (Luminy, 1984)}, M\'em.\,Soc.\,Math.\,France, vol.\,23, pp.\,7--32,  (1986), \href{https://doi.org/10.24033/msmf.326}{DOI: 10.24033/msmf.326}.

\bibitem{Berthelot1997} 
\textsc{P.\,Berthelot}: 
\textit{Dualiti\'{e} de Poincar\'{e} et formule de K\"{u}nneth en cohomologie rigide}. 
C.\,R.\,Acad.\,Sci.\,Paris, vol.\,325, ser.\,I, pp.\,493--498, (1997), 
\href{https://doi.org/10.1016/S0764-4442(97)88895-7}{DOI: 10.1016/S0764-4442(97)88895-7}.

\bibitem{Berthelot1997a} 
\textsc{P.\,Berthelot}: 
\textit{Finitude et puret\'e cohomologique en cohomologie rigide (with an appendix in English by A.\,J.\,de~Jong)}. 
Invent.\,Math., vol.\,128, pp.\,329--377, (1997), \href{https://doi.org/10.1007/s002220050143}{DOI: 10.1007/s002220050143}.

\bibitem{BerthelotOgus}
\textsc{P.\,Berthelot, A.\,Ogus}:
\textit{Notes on crystalline cohomology}.
Princeton University Press, Princeton, (1978) \href{https://doi.org/10.1515/9781400867318}{DOI: 10.1515/9781400867318}.

\bibitem{Besser2000}
\textsc{A.\,Besser}: 
\textit{Syntomic regulators and $p$-adic integration I: Rigid syntomic regulators}. 
Isr.\,J.\,Math., vol.\,120,  pp.\,291--334, (2000) \href{https://doi.org/10.1007/bf02834843}{DOI: 10.1007/bf02834843}.

\bibitem{DegliseNiziol}
\textsc{F.\,D\'{e}glise and W.\,Nizio\l}:
\textit{On $p$-adic absolute Hodge cohomology and syntomic coefficients,\ I}.
Comment.\,Math.\,Helv.,\ vol.\,93, no.\,1, pp.\,71--131, (2018) \href{https://doi.org/10.4171/CMH/430}{DOI: 10.4171/CMH/430}.

\bibitem{ErtlYamada1}
\textsc{ V.\,Ertl and K.\,Yamada}:
 {\it Comparison between rigid syntomic and crystalline syntomic cohomology for strictly semistable log schemes with boundary}.
Rend.\,Sem.\,Mat.\,Univ.\,Padova, vol.\,145, pp.\,213--291, (2021), \href{https://doi.org/10.4171/rsmup/81}{DOI: 10.4171/rsmup/81}.

\bibitem{ErtlYamada}
\textsc{V.\,Ertl and K.\,Yamada}: 
{\it Rigid analytic reconstruction of Hyodo--Kato theory}. 
Preprint, \href{https://arxiv.org/abs/2003.05960}{arXiv:1907.10964}.

\bibitem{Faltings1999}
\textsc{ G.\,Faltings}:
{\it Integral crystalline cohomology over very ramified valuation rings}.
J.\,Am.\,Math.\,Soc., vol.\,12, no. 1, pp.\,117--144, (1999), \href{https://doi.org/10.1090/S0894-0347-99-00273-8}{DOI: 0.1090/S0894-0347-99-00273-8}.

\bibitem{Faltings2002}
\textsc{G.\,Faltings}:
\textit{Almost \'{e}tale extensions}.
In: \textit{Cohomologies $p$-adiques et applications arithm\'{e}tics II}, Ast\'{e}risque, vol.\,279, pp.\,185--270, (2002), \href{https://doi.org/10.24033/ast.535}{10.24033/ast.535}.

\bibitem{GrosseKlonne2005}
\textsc{E.\,Gro\ss e-Kl\"{o}nne}:
{\it Frobenius and monodromy operators in rigid analysis, and Drinfel'd's symmetric space}.
J.\,Algebraic Geometry, vol.\,14, pp.\,391--437, (2005), \href{https://doi.org/10.1090/S1056-3911-05-00402-9}{DOI: 10.1090/S1056-3911-05-00402-9}.

\bibitem{Hartl2001} 
\textsc{U.\,Hartl}: 
\textit{Semi-stability and base change}. 
Arch.\,Math.\,vol.\,77, pp.\,215--221, (2001), \href{https://doi.org/10.1007/pl00000484}{DOI: 10.1007/pl00000484}.


\bibitem{HyodoKato1994} 
\textsc{O.\,Hyodo and K.\,Kato}: 
\textit{Semi-stable reduction and crystalline cohomology with logarithmic poles}.  
Ast{\'e}risque, vol.\,223, pp.\,221--268, (1994), 
\href{http://www.numdam.org/item/AST_1994__223__221_0/}{http://www.numdam.org/item/AST\_1994\_\_223\_\_221\_0/}.

\bibitem{KashiwaraSchapira}
\textsc{M.\,Kashiwara and P.\,Schapira},
\textit{Categories and sheaves}.
Grundlehren Math.\,Wiss., vol.\,332, Springer-Verlag, Berlin, (2006), \href{https://doi.org/10.1007/3-540-27950-4}{DOI: 10.1007/3-540-27950-4}.

\bibitem{Kato1989}
\textsc{K.\,Kato}: 
\textit{Logarithmic structures of Fontaine--Illusie}.
In: \textit{Algebraic analysis, geometry, and number theory}. 
Johns Hopkins University Press, Baltimore, pp.\,191--224, (1989) .

\bibitem{Kato2011} 
\textsc{K.\,Kato}: 
\textit{$p$-adic period domains and toroidal partial compactifications I}. 
Kyoto J.\,Math., vol.\,51, no.\,3, pp.\,561--631, (2011), \href{https://doi.org/10.1215/21562261-1299900}{DOI: 10.1215/21562261-1299900}.

\bibitem{KimHain2004} 
\textsc{M.\,Kim and R.\,Hain}: 
\textit{A de~Rham-Witt approach to crystalline rational homotopy theory}. 
Compo.\,Math., vol.\,140, no.\,5, pp.\,1245--1276, (2004), \href{https://doi.org/10.1112/S0010437X04000442}{DOI: 10.1112/S0010437X04000442}.

\bibitem{LoefflerZerbes}
\textsc{D.\,Loeffler and S.L.\,Zerbes}:
\textit{On the Bloch--Kato conjecture for $\mathrm{GSp}(4)$}.
Preprint, \href{https://arxiv.org/abs/2003.05960}{arXiv:2003.05960}.

\bibitem{MonskyWashnitzer}
\textsc{P.\,Monsky and G.\,Washnitzer}:
\textit{Formal cohomology: I}.
Annals of Math.\,Second Series, vol.\,88, no.\,2, pp.\,181--217, (1968), \href{https://doi.org/10.2307/1970571}{DOI: 10.2307/1970571}.

\bibitem{NakkajimaShiho2008}
\textsc{Y.\,Nakkajima and A.\,Shiho}: 
\textit{Weight Filtrations on Log Crystalline Cohomologies of Families of Open Smooth Varieties}. 
Lect.\,Notes Math., vol.\,1959, Springer-Verlag Berlin Heidelberg, (2008), \href{https://doi.org/10.1007/978-3-540-70565-9}{DOI: 10.1007/978-3-540-70565-9}. 

\bibitem{NavarroAznar1987} 
\textsc{V.\,Navarro Aznar}: 
\textit{Sur la th\'{e}orie de Hodge--Deligne}. 
Invent.\,Math., vol.\,90, pp.\,11--76, (1987), \href{https://doi.org/10.1007/BF01389031}{DOI: 10.1007/BF01389031}.
 
\bibitem{NekovarNiziol2016} 
\textsc{J.\,Nekov\'a\v{r} and W.\,Nizio\l{}}: 
\textit{Syntomic cohomology and $p$-adic regulators for varieties over $p$-adic fields}. 
Algebra Number Theory, vol.\,10, no.\,8, pp.\,1695--1790, (2016), \href{https://doi.org/10.2140/ant.2016.10.1695}{DOI: 10.2140/ant.2016.10.1695}.

\bibitem{Niziol2008}
\textsc{W.\,Niziol}:
\textit{Semi-stable conjecture via $K$-theory}.
Duke Math.\,J., vol.\,141, no.\,1, pp.\,151--179, (2008), \href{https://doi.org/10.1215/S0012-7094-08-14114-6}{DOI: 10.1215/S0012-7094-08-14114-6}.

\bibitem{Ogus1984}
\textsc{A.\,Ogus}:
\textit{$F$-isocrystals and de~Rham cohomology II -- Convergent isocrystals}.
Duke Math.\,J., vol.\,51, no.\,4, pp.\,765--850, (1984), \href{https://doi.org/10.1215/S0012-7094-84-05136-6}{DOI: 10.1215/S0012-7094-84-05136-6}.

\bibitem{PetersSteenbrink2007}
\textsc{C.\,Peters and J.\,Steenbrink}:
\textit{Mixed Hodge structures}.
Ergeb.\,Math.\,Grenzgeb., ed.\,1, Springer-Verlag Berlin Heidelberg, (2008), \href{https://doi.org/10.1007/978-3-540-77017-6}{DOI: 10.1007/978-3-540-77017-6}.

\bibitem{Shiho2000}
\textsc{A.\,Shiho}:
\textit{Crystalline Fundamental Groups I -- Isocrystals on Log Crystalline Site and Log Convergent Site}.
J.\,Math.\,Sci.\,Univ.\,Tokyo, vol.\,7, pp.\,509--656, (2000), 
\href{https://www.ms.u-tokyo.ac.jp/journal/abstract/jms070401.html}{https://www.ms.u-tokyo.ac.jp/journal/abstract/jms070401.html}.

\bibitem{Shiho2002} 
\textsc{A.\,Shiho}: 
\textit{Crystalline fundamental groups II -- Log convergent cohomology and rigid cohomology}. 
J.\,Math.\,Sci.\,Univ.\,Tokyo, vol.\,9, pp.\,1--163, (2002), 
\href{https://www.ms.u-tokyo.ac.jp/journal/abstract/jms090101.html}{https://www.ms.u-tokyo.ac.jp/journal/abstract/jms090101.html}.
 
\bibitem{Shiho2008} 
\textsc{A.\,Shiho}: 
\textit{Relative log convergent cohomology and relative rigid cohomology 1}. 
Preprint, \href{https://arxiv.org/abs/0707.1742}{arXiv:0707.1742}.

\bibitem{Tsuji1999} 
\textsc{T.\,Tsuji}: 
\textit{Poincar\'{e} duality for logarithmic crystalline cohomology}. 
Compos.\,Math., vol.\,118, pp.\,11--41, (1999), \href{https://doi.org/10.1023/A:1001020809306}{DOI: 10.1023/A:1001020809306}.

\bibitem{Tsuji1999-a} 
\textsc{T.\,Tsuji}: 
\textit{$p$-adic \'etale cohomology and crystalline cohomology in the semi-stable reduction case}. 
Invent.\,Math., vol.\,137, no.\,2, pp.\,233--411, (1999), \href{https://doi.org/10.1007/s002220050330}{DOI: 10.1007/s002220050330}.

\bibitem{Yamada2020}
\textsc{K.\,Yamada}
\textit{Hyodo--Kato theory with syntomic coefficients}.
Preprint, \href{https://arxiv.org/abs/2005.05694}{arXiv:2005.05694}.


\bibitem{Yamashita2011}
 \textsc{G.\,Yamashita}: 
\textit{$p$-adic Hodge theory for open varieties}.
C.\,R.\,Math., vol.\,349, no.\,21, pp.\,1127--1130, (2011), \href{https://doi.org/10.1016/j.crma.2011.10.016}{DOI: 10.1016/j.crma.2011.10.016}.


\bibitem{YamashitaYasuda2014} 
\textsc{G.\,Yamashita and S.\,Yasuda}: 
\textit{$p$-adic \'etale cohomology and crystalline cohomology for the open varieties with semistable reduction}.
Preprint.

\bibitem{stacks}
\textsc{Stacks Project},
\textit{The Stacks Project Authors},
\url{https://stacks.math.columbia.edu}.
\end{thebibliography}
\end{document}